\documentclass{amsart}
\usepackage{amssymb,color,epsf}
\usepackage{mathtools}
\DeclareMathAlphabet{\mathpzc}{OT1}{pzc}{m}{it}

\usepackage [cmtip,arrow]{xy}
\xyoption{all}
\usepackage {pb-diagram,pb-xy}
\usepackage[mathscr]{eucal}
\usepackage[utf8]{inputenc}
\usepackage[T1]{fontenc}
\usepackage[english]{babel}
\ifx\pdfpageheight\undefined
   \usepackage[dvips,colorlinks=true,linkcolor=blue,citecolor=red,%
      urlcolor=green]{hyperref}
   \usepackage[dvips]{graphicx}
   \makeatletter
   \edef\Gin@extensions{\Gin@extensions,.mps}
   \makeatother
\else
 \usepackage[pdftex]{graphicx}
   \usepackage[bookmarksopen=false,pdftex=true,breaklinks=true,%
      backref=page,pagebackref=true,plainpages=false,%
      hyperindex=true,pdfstartview=FitH,colorlinks=true,%
      pdfpagelabels=true,colorlinks=true,linkcolor=blue,%
      citecolor=red,urlcolor=green,hypertexnames=false%
      ]%
   {hyperref}
\fi
\usepackage{bm}

\newcommand{\bs}[1]{\boldsymbol{#1}}

\usepackage{amsmath,amssymb,amsfonts}
\newtheorem{theorem}{Theorem}[section]
\newtheorem{lemma}[theorem]{Lemma}
\newtheorem{corollary}[theorem]{Corollary}
\newtheorem{proposition}[theorem]{Proposition}
\newtheorem{conjecture}[theorem]{Conjecture}
\newtheorem{hypothesis}[theorem]{Hypothesis}
\newtheorem{question}[theorem]{Question}

\theoremstyle{definition}
\newtheorem{definition}[theorem]{Definition}
\newtheorem{example}[theorem]{Example}

\newtheorem{notation}[theorem]{Notation}
\newtheorem{warning}[theorem]{Warning}

\theoremstyle{remark}
\newtheorem{remark}[theorem]{Remark}

\definecolor{DarkBlue}{rgb}{0,0.1,0.55}

\numberwithin{equation}{section}

\newcommand {\hide}[1]{}
 \newcommand {\sign} {\mbox{\bf sign}}
 \newcommand{\zero}{\mbox{\bf zero}}

\newcommand {\junk}[1]{}

\newcommand {\R} {\mathbb{R}}

\newcommand {\D}     {\mbox{\rm D}}
\newcommand {\A}     {\mbox{\rm A}}
\newcommand {\C}     {\mathbb{C}}

\newcommand {\Sphere}{\mbox{${\bf S}$}}     
\newcommand {\Ball}{\mbox{${\bf B}$}}     

\newcommand {\Z}  {\mathbb{Z}}
 \newcommand {\N}         {\mathbb{N}}
\newcommand {\Q}         {\mathbb{Q}}
\newcommand {\kk}         {\mathbf{k}}

\newcommand {\ZZ} {{\rm Z}}
\newcommand {\RR} {{\mathcal R}}

\newcommand {\eps} {{\varepsilon}}

\newcommand {\PP}     {\mathbb{P}} 

\newcommand{\card}{\mathrm{card}}
\newcommand{\rank}{\mathrm{rank}}
\newcommand{\wdet}{\mbox{\rm weighted-det}}
\newcommand{\RHom}{R\mathpzc{Hom}}

\def\addots{\mathinner{\mkern1mu
\raise1pt\vbox{\kern7pt\hbox{.}}
\mkern2mu\raise4pt\hbox{.}\mkern2mu
\raise7pt\hbox{.}\mkern1mu}}

\newcommand{\HH}  {\mbox{\rm H}}
\newcommand {\hyperH}{\mathbb{H}}

\newcommand{\Sym}  {\mathrm{Sym}}

\newcommand{\coker}{\mathrm{coker}}
\newcommand{\coimage}{\mathrm{coimage}}
\newcommand{\Hom}{\mathrm{Hom}}
\newcommand{\image}{\mathrm{Im}}

\newcommand{\defeq}{\;{\stackrel{\text{\tiny def}}{=}}\;}

\newcommand{\x}{\mathbf{x}}
\newcommand{\X}{\mathbf{X}}

\newcommand{\y}{\mathbf{y}}
\newcommand{\Y}{\mathbf{Y}}

\newcommand{\vb}{\mathbf{v}}

\newcommand{\z}{\mathbf{z}}

\newcommand{\tb}{\mathbf{t}}

\newcommand{\Ob}{\mathrm{Ob}}
\newcommand{\supp}{\mathrm{Supp}}
\newcommand{\CF}{\mathrm{CF}}

\newcommand{\Sh}{\mathbf{Sh}}
\newcommand{\constrD}{\mathbf{D}_{\mathrm{sa}}}
\newcommand{\size}{\mathrm{size}}
\newcommand{\Kom}{\mathbf{Kom}}
\newcommand{\K}{\mathbf{K}}
\newcommand{\boldmod}{\textbf{-}\mathbf{mod}}

\newcommand{\Par}{\mathrm{Par}}
\newcommand{\Var}{\mathbf{Var}}
\newcommand{\sa}{\mathbf{sa}}

\newcommand{\kA}{\mathbf{A}}
\newcommand{\inv}{\mathrm{inv}}
\newcommand{\vP}{{\mathbf{P}_{\mathrm{virtual}}}}

\newcommand{\dd}{\mathbf{d}}

\newcommand{\level}{\mathrm{level}}
\newcommand{\Cc}{\mathrm{Cc}}

\newcommand{\pt}{\mathrm{pt}}
\begin{document}
\title[A complexity theory of constructible functions
and sheaves]
{A complexity theory of constructible functions
and sheaves}
\author{Saugata Basu}
\address{Department of Mathematics,
Purdue University, West Lafayette, IN 47906, U.S.A.}
\email{sbasu@math.purdue.edu}
\thanks{The author was supported in part by NSF grants
  CCF-0915954,  CCF-1319080 and DMS-1161629 while working on this paper. \\
  Communicated by Teresa Krick and James Renegar.}

\subjclass{Primary 14P10, 14P25; Secondary 68W30}

\dedicatory{Dedicated to Mike Shub on his 70-th birthday.}

\keywords{constructible functions, constructible sheaves, polynomial hierarchy, complexity classes, adjoint functors}

\begin{abstract}
In this paper we introduce constructible analogs of the discrete complexity classes  $\mathbf{VP}$ and $\mathbf{VNP}$ of sequences of functions. The  functions in the new definitions 
are  constructible functions on $\R^n$ or $\C^n$. We define a class of sequences of constructible functions that play a role analogous 
to that of $\mathbf{VP}$ in the more classical theory. The class analogous to $\mathbf{VNP}$ is defined
using Euler integration.  We discuss several examples, develop a theory of completeness, and pose a conjecture analogous to the  $\mathbf{VP}$ vs. $\mathbf{VNP}$ conjecture in the classical case. In the second part of the paper we extend the notions
of complexity classes to sequences of constructible sheaves over $\R^n$ (or its one point compactification). We introduce a class of sequences of simple constructible sheaves, that could be seen
as the sheaf-theoretic analog of the Blum-Shub-Smale class $\mathbf{P}_\R$. We also define a hierarchy
of complexity classes of sheaves mirroring the polynomial hierarchy, $\mathbf{PH}_\R$, in the B-S-S theory. We
prove a singly exponential upper bound on the topological complexity of the sheaves in this hierarchy mirroring a similar result in the B-S-S setting. We obtain as a result an algorithm with  singly exponential complexity for a sheaf-theoretic variant of the real quantifier elimination problem. We pose the natural sheaf-theoretic analogs of the classical $\mathbf{P}$ vs. $\mathbf{NP}$ question,  and also discuss a connection with Toda's theorem from discrete complexity theory in the context of constructible sheaves.  We also discuss possible  generalizations of  the questions in complexity theory related to separation of complexity classes  to more general categories via sequences of adjoint pairs of functors.
\end{abstract}

\maketitle
\tableofcontents
\section{Introduction}
The $\mathbf{P}$ vs. $\mathbf{NP}$ problem is considered a central problem in theoretical computer science. There have been several reformulations and generalizations of the classical discrete version
of this problem. Most notably Blum, Shub and Smale \cite{BSS89} generalized the $\mathbf{P}$ vs. $\mathbf{NP}$ problem to arbitrary fields, and posed the conjecture that $\mathbf{P}_\R\neq
\mathbf{NP}_\R$, and that   $\mathbf{P}_\C \neq
\mathbf{NP}_\C$. In another direction, Valiant \cite{Valiant79,Valiant82,Valiant84}
 introduced a non-uniform function
theoretic analog. Valiant's definition 
concerned 
classes of functions as opposed to sets  (see \cite{Burgisser-book2} 
for results on the exact relationship between Valiant's conjecture and the classical
complexity questions between complexity classes of sets). The class $\mathbf{VP}$ 
and its variants, such as the class $\mathbf{VQP}$  (see \cite{Burgisser-book2} for many subtle details), are supposed to represent functions
that are easy to compute and plays a role analogous to the role of class $\mathbf{P}$ in the
case of complexity classes of sets. While the class of functions, $\mathbf{VNP}$,  was supposed to 
play a role analogous to that of the class of languages $\mathbf{NP}$.  Valiant's theory leads to an  elegant reduction of the question 
whether  $\mathbf{VQP}= \mathbf{VNP}$ to a purely algebraic one -- 
namely, whether the 
polynomial given by the  permanent of an $n \times n$ matrix 
(with indeterminate entries) can be expressed
as the determinant of another (possibly polynomially larger) 
matrix whose entries are 
are linear combinations of the entries of the original matrix. Thus, the
question of whether $\mathbf{VQP}= \mathbf{VNP}$ reduces to a purely mathematical
question about polynomials and mathematical tools from representation theory and
algebraic geometry can be made to bear on this subject  (see \cite{MS01,BLMW2011}). 

The aim of this paper is to lay a foundation of studying complexity questions similar to the the various 
existing $\mathbf{P}$ vs. $\mathbf{NP}$ type questions in a more geometric setting. Fundamentally, the
question whether $\mathbf{P} = \mathbf{NP}$ concerns the behavior under projections $\pi_{m+n}: \kk^{m+n} \rightarrow \kk^n$ 
of sequences of sets $(S_n \subset \kk^n)_{n > 0}$ (where each $S_n$ is  
constructible or semi-algebraic, depending on whether the underlying field $\kk=\C$ or $\R$ respectively) whose membership is 
easy to test (that is admits a \emph{polynomial-time} algorithm for testing whether a given point belongs to a set in the sequence). The conjecture that $\mathbf{P}_\kk \neq \mathbf{NP}_\kk$ can then
be understood as saying that there exists sequences $(S_n \subset \kk^n)_{n > 0}$ admitting polynomial-time membership testing for which the sequence $(\pi_{m(n),n}(S_{m(n)+n}) \subset \kk^n)_{n > 0}$ does not admit polynomial time membership testing, where $m(n)$ is some non-negative polynomial in $n$ . Notice that membership of a point $\x$ in the image $\pi_{m(n)+n}(S_{m(n)+n})$ only
records information about where the fiber $\pi_{m(n)+n}^{-1}(\x) \cap S_{m(n)+n}$ is empty or not. However, in many geometric situations  it is useful to know more. In fact, one would like to have a partition of
the base space $\kk^n$ such that over a set $C$ in the partition the fibers 
$\pi_{m(n)+n}^{-1}(\x) \cap S_{m(n)+n}, \x\in C$ maintains some invariant (which could be topological or algebro-geometric in nature). Instead of studying the complexity of just the image $\pi_{m(n)+n}(S_{m(n)+n})$ one would like to understand the complexity of this partition. Thus, in our view the main geometric objects whose
complexity we study are no longer individual subsets of $\kk^n$, but rather finite (constructible or 
semi-algebraic) partitions. We then define natural notions of ``push-forwards'' of such partitions. The geometric analog of $\mathbf{P}$ vs. $\mathbf{NP}$ problem then asks whether such a push-forward
can increase (exponentially) the complexity of a sequence of partitions. 

We develop this theory in two settings.  In the first part of the paper we develop a
geometric analog  of the complexity classes $\mathbf{VP}$ and $\mathbf{VNP}$ introduced by Valiant \cite{Valiant79,Valiant82,Valiant84}.
Recall that the classes $\mathbf{VP}$ and $\mathbf{VNP}$ as defined by Valiant are \emph{non-uniform}.
The circuits or formulas whose sizes measure the complexity of functions are allowed to be
different for different sizes of the input. Also recall that  
unlike in the classical theory the elements 
of the classes $\mathbf{VP}$ and $\mathbf{VNP}$ are not languages but 
sequences of functions. In Valiant's
work the emphasis was on functions $f:\{0,1\}^n \rightarrow \kk$ (for some field
$\kk$).  Since any function on $\{0,1\}^n$ can be expressed as a polynomial, it 
makes sense to consider only  polynomial functions. 
In particular, characteristic 
functions of subsets of $\{0,1\}^n$ are also expressible as polynomials -- and
this provides a crucial link between the function viewpoint and the classical
question about languages.

\subsection{Complexity theory of constructible functions}
We formulate a certain geometric analog of Valiant's complexity classes. 
The first obstacle to overcome is that unlike in the
classical (Boolean) case, when the underlying field is infinite 
(say $\kk = \R$ or $\C$), the characteristic function of a definable set
(i.e., a constructible set in the case $\kk = \C$ and a semi-algebraic set
in case $\kk = \R$) is no longer expressible as a polynomial. 
But a class of functions that appears very naturally in the algebraic
geometry over real and complex numbers are the so called \emph{constructible 
functions}. We will see later that many discrete valued functions that appear in 
complexity theory including functions such as the characteristic functions of 
constructible as well as semi-algebraic sets, ranks of matrices and higher dimensional tensors, 
topological invariants such 
as the Betti numbers or Euler-Poincar\'e characteristics, local dimensions of 
semi-algebraic sets are all examples of such functions. 
Constructible functions in the place of so called ``counting'' functions have already
appeared in B-S-S style complexity theory over $\R$ and $\C$. For example, in \cite{BZ09} 
(respectively, \cite{Basu-complex-toda}) 
a real (respectively, complex) analog of Toda's theorem of discrete complexity theory
was obtained. The notion of counting the number of satisfying assignments of a Boolean
equation was replaced in these papers by the problem of computing the Poincar\'e polynomial
of a semi-algebraic/constructible set (see also \cite{Burgisser-Cucker06, Burgisser-Cucker-survey}). 
A first goal of this paper is to build  a (non-uniform) 
complexity theory for constructible functions over real as well as complex 
numbers that mirror Valiant's theory in the discrete case. 

The choice of constructible functions as a ``good'' class of functions is also 
motivated from another direction. First recall that in the case of languages, the languages
in the class $\mathbf{NP}$ can be thought of as the images under projections of the
languages in the class $\mathbf{P}$ (see Section \ref{subsec:recall-PH} for more precise definitions of these classes).
For classes of functions such as the class $\mathbf{VP}$, in order to define an analog of the class
$\mathbf{NP}$ one needs a way of ``pushing forward'' a function under a projection. It is 
folklore that functions (or more generally maps) can be pulled-back tautologically, but pushing
forward requires some effort.
The standard technique in mathematics is
to define such a push-forward using ``fiber-wise integration''.  In Valiant's original
definition of the class $\mathbf{VNP}$ this push-forward was implemented by taking
the sum of the function to be integrated over the Boolean cube $\{0,1\}^n$. This operation
is not very geometric and thus not completely  satisfactory in a geometric setting.  On the other hand  integration against most normal  measures
(other than finite atomic ones such as the one used by Valiant in his definition) 
will not be computable exactly as the results will not be algebraic.  It thus becomes a
subtle problem to choose the right class of functions and the corresponding 
push-forward. It turns out that the class of constructible
functions is particularly suited for this purpose, where a discrete notion of 
integration (with respect to additive invariants such as the  Euler-Poincar\'e characteristic) already exists. It makes
sense now to put these together and develop 
an analog of Valiant theory for this class, which is what we begin to do in this paper.
The complexity classes
of constructible functions and their corresponding ``$\mathbf{P}$ vs. 
$\mathbf{NP}$''-type questions that will arise in these new models, should be
considered as the ``constructible''  versions of the corresponding questions in 
Valiant's theory.
We define formally these new classes, prove the existence of certain
``complete''  sequences of functions,  give some
examples, and finally pose a ``constructible'' analog of the  $\mathbf{VP}$ vs.  $\mathbf{VNP}$ question.

Constructible functions on $\kk^n$ produce a constructible or semi-algebraic partition of $\kk^n$. Push-forwards of constructible
functions being themselves constructible also induce similar partitions. Note that 
the complexity of semi-algebraic partitions defined by the constancy of some topological or algebraic
invariant of the fibers of a map is not a new notion. Indeed, Hardt's triviality theorem in semi-algebraic
geometry \cite{Hardt} (see also \cite[Theorem 9.3.2.]{BCR}) implies the existence of such a partition where the topological invariant being preserved is the semi-algebraic homeomorphism type of the fibers. Unfortunately, the best known upper bound on the complexity of the partition in Hardt's theorem in terms of the number and degrees of the
defining polynomials of the semi-algebraic sets involved is doubly-exponential, while the best lower bound known is singly exponential \cite{BV06}. 
A doubly exponential lower bound on the complexity
on the partition of the Hardt's triviality theorem could be in principle a step towards proving a version of the ``$\mathbf{P} \neq \mathbf{NP}$'' problem (more precisely the conjecture that  $\mathbf{VP}_\R^\R \neq \mathbf{VNP}_\R^\chi$ posed below). However, such a doubly exponential lower bound is considered unlikely.
A singly exponential upper bound on the number of \emph{homotopy types} of fibers of a semi-algebraic
map $f:X \rightarrow Y$ was proved in \cite{BV06} . Even though a partition of the base space $Y$ was not constructed explicitly in that paper, a singly exponential sized partition is implicit in the proof of the main theorem in that paper.  

Another result in a similar direction that is worth mentioning here is a theorem on bounding 
the number of possible vectors of multiplicities (defined below) of the zeros of a polynomial system $\mathcal{P} = \{P_1,\ldots,P_n\} \subset\C[X_1,\ldots,X_n]$, in terms of the degrees of the $P_i$'s proved in \cite{Grigoriev-Vorobjov2000}. In this paper the authors prove a doubly exponential upper bound on the number of possible vectors of multiplicities of the zeros of $\mathcal{P}$ (assuming that the system $\mathcal{P}$ is zero-dimensional)  in terms of the degrees of the polynomials in $\mathcal{P}$. A \emph{vector of  multiplicity}  $(m_1,m_2,\ldots)$, $m_i \in \Z_{>0}$ where each $m_i$ is the multiplicity of a unique zero of the 
system $\mathcal{P}$.  In \cite{Grigoriev2001}, an explicit construction of a sequence of parametrized polynomial systems with a polynomially bounded
straight-line complexity is given that realizes a doubly exponential number (in the number of parameters) of multiplicity vectors. In the context of the current paper, this would 
imply that the partition of the base space (i.e the space of parameters) would need to have at least a doubly exponential number of elements in its partition,  so as to maintain the constancy of the vector of multiplicities in the fibers over each element of the partition. This would prove a version of ``$\mathbf{P} \neq \mathbf{NP}$'' for our model -- except that  ``multiplicity'' is not an \emph{additive invariant}, and many solutions having different vectors of multiplicities
would have the same value of additive invariants (such as the generalized Euler-Poincar\'e characteristic or the virtual Poincar\'e polynomial defined
later in the paper in Sections \ref{subsubsec:generalizedP} and \ref{subsubsec:virtualP} respectively). 
This shows that pushing forward of functions
using fiber-wise integration with respect to an \emph{additive invariant} is a very important element of the theory developed in this paper.

\subsection{Complexity theory of constructible sheaves}
The second part of the paper has no analog in discrete complexity theory but is strongly motivated
by the first part, and prior results on algorithmic complexity of various problems in semi-algebraic geometry.
One way that constructible functions appear in various applications is as the fiber-wise Euler-Poincar\'e 
characteristic of certain sheaves of complexes with bounded cohomology. The right generality to consider these objects --
namely a bounded derived category of sheaves of modules, which are locally constant on each element of
a semi-algebraic partition of the ambient manifold into locally closed semi-algebraic sets -- lead naturally to the category 
of \emph{constructible sheaves}. Constructible sheaves are a particularly simple kind of sheaves arising in algebraic geometry \cite{SGA4} and have found many applications in mathematics
(in the theory of linear systems of partial differential equations and micro-local analysis \cite{KS}, 
in the study of singularities that appear in linear differential equations with meromorphic
coefficients \cite{Deligne-regulier-singulier,Pham}, study of local systems in algebraic geometry \cite{Deligne-SGA4-half}, intersection cohomology theory \cite{Borel} amongst many others) but to our knowledge  they have not been studied yet from the structural complexity point of view. Constructible functions have also being studied by many authors from different perspectives, such as 
\cite{Parusinski-Mccrory,Cluckers-Edmundo,Cluckers-Loeser}. Recently they have also
found applications in more applied areas such as signal processing and data analysis \cite{Ghrist2010}, but to our knowledge they have not being studied from the point of view of
complexity.

The category of constructible sheaves is closed under the so called ``six
operations of Grothendieck''  -- namely $\stackrel{L}{\otimes}, \RHom, Rf_*,Rf_{!}, f^{-1},f^{!}$ \cite{SGA4} (see
\cite[Theorem 4.1.5]{Dimca-sheaves}). The closure under these operations is reminiscent of the closure of the class
of semi-algebraic sets under similar operations -- namely, set theoretic operations, direct products, pull-backs and direct images under semi-algebraic maps. Of this the closure under the last operation -- that
is the fact that the image of a semi-algebraic set is also semi-algebraic -- is the most non-trivial property
and corresponds to the Tarski-Seidenberg principle (see for example, \cite[Chapter 2]{BPRbook2} for an exposition). The computational difficulty of this
last operation -- i.e., elimination of an existential block of quantifiers -- is also at the heart of the $\mathbf{P}_\R$ vs. $\mathbf{NP}_\R$ problem in the B-S-S theory \cite{BSS89, BCSS98}. 

As mentioned above the category of constructible sheaves
is closed under taking direct sums, tensor products and pull-backs. These should be considered as the 
``easy'' operations. The statement analogous to the Tarski-Seidenberg principle is the stability under taking direct images. These observations hint at a complexity theory of such sheaves that will subsume the ordinary set theoretic complexity classes as special cases. Starting with a properly defined class, $\bs{\mathcal{P}_\R}$, of ``simple'' sheaves, a conjectural hierarchy can be built up by taking successive direct images followed by truncations, tensor products etc.  which resembles the polynomial hierarchy in the B-S-S model. The class  $\bs{\mathcal{P}_\R}$ corresponds roughly to the sequences  of constructible sheaves for which there is a compatible stratification of each underlying ambient space (which we will assume to
be spheres of various dimensions in this paper) which is singly exponential in size, and where point location
can be accomplished in polynomial time (see Definition \ref{def:sheaf-P} below for a precise definition).
In this paper we lay the foundations of such a theory. We give several examples and also prove a result on the topological
complexity of sequences of sheaves belonging to such a hierarchy.
Even though constructible sheaves can be defined over any fields -- for the purposes of this paper 
we restrict ourselves to the field of real numbers.

One important unifying theme behind the two parts of this paper is a certain 
shift of view-point -- from 
considering sequences of polynomial functions (as in Valiant's model) or sequences of  subsets of $\R^n$ or  $\C^n$ as in the B-S-S notion of a ``language'' -- to considering sequences of semi-algebraic or constructible
partitions of $\R^n$ and $\C^n$ compatible  with a sequence constructible functions or sheaves. The goal is
then to study how the ``complexity'' of such partitions increase under an appropriate push-forward
operation giving rise to constructive analogs of $\mathbf{P} \mbox{ vs. } \mathbf{NP}$ question. This ``push-forward'' operation which corresponds to ``summing fibers'' in the Valiant model, and taking the image under projection in the B-S-S model -- for us is defined geometrically as
``Euler integration'' in case of constructible functions, and taking the ``higher direct images'' in 
the case of constructible sheaves. The extra structure of a constructible functions or sheaves allow us 
to define this push-forward operation -- while the complexities of these objects are measured in terms
of the ``complexity'' of the underlying geometric partitions. 

\subsection{A functorial view of complexity questions and the role of adjoint functors}
\label{subsec:adjoint}
Another aspect of complexity theory that the sheaf-theoretic point of view brings to the forefront
(but which we do not explore in full generality in this paper)
is the role of \emph{adjoint functors}.

Recall that any map $f:X \rightarrow Y$ between sets $X$ and $Y$  induces three functors 

\[
\textbf{Pow}(X) {\xrightarrow{f^\exists} \atop {\xleftarrow{f^*}  \atop \xrightarrow{f^\forall}}}  \textbf{Pow}(Y).
\]
in the poset categories of their respective power sets $\textbf{Pow}(X), \textbf{Pow}(Y)$.
The functors
$f^*,f^\exists,f^\forall$  are defined as follows.  For all 
$A \in \Ob(\textbf{Pow}(X))$ and
$B \in \Ob(\textbf{Pow}(Y))$,  
\begin{eqnarray*}
f^*(B) & =&  f^{-1}(B), \\
f^\exists(A) &=& \{y \in Y \mid (\exists x \in X)((f(x) = y) \wedge (x \in A)) \},\\
f^\forall(A)  &=&  \{ y \in Y \mid (\forall x\in X) ( (f(x) = y)  \implies (x \in A)) \}.
\end{eqnarray*}
It is a well-known observation (see for example, \cite[page 58, Theorem~2]{Maclane-book})
that $f^\exists$ is left adjoint to $f^*$, and $f^*$  is left adjoint to $f^\forall$, i.e.,
$f^\exists \dashv f^* \dashv f^\forall$ as functors between the poset categories
$\textbf{Pow}(X), \textbf{Pow}(Y)$.

The adjunctions defined above appear in complexity theory in the following guise. 
Let $\kk$ be either a finite field or the field of real or complex numbers.
Suppose that  $(S_n \subset \kk^n)_{n >0}$  is a sequence belonging to the complexity
class $\mathbf{P}_\kk$, and let $\pi_{2n,n}: \kk^{2n} \rightarrow \kk^n, n >0$ be the 
projection maps on the first $n$ coordinates. Then,
the sequence $(\pi_{2n,n}^*(S_n) \subset \kk^{2n})_{n >0}$ is clearly also in $\textbf{P}_\kk$. The question
whether the sequences $(\pi_{2n,n}^\exists(S_{2n}) \subset \kk^{n})_{n >0}$ (respectively,
$(\pi_{2n,n}^\forall(S_{2n}) \subset \kk^{n})_{n >0}$), belong to $\mathbf{P}_\kk$ corresponds to the
$\textbf{P}_\kk$   vs. $\textbf{NP}_\kk$ (respectively,  the $\textbf{P}_\kk$   vs. $\textbf{co-NP}_\kk$) problem. 
Thus, basic questions in complexity theory  can be posed as questions 
about sequences of functors and their adjoints. 

Perhaps not surprisingly, adjoint pairs of functors appear as well in the sheaf-theoretic version of complexity theory 
studied in this paper -- in particular,  the adjoint pairs 
$f^{-1} \dashv Rf_*$,  and $ \mathcal{F}^\bullet\stackrel{L}{\otimes} \cdot  \dashv \RHom(\mathcal{F}^\bullet,\cdot)$ (see Section \ref{subsec:background-on-sheaves}  for precise definitions).  The role of  $\textbf{P}$ is played by the class
$\bs{\mathcal{P}}_\R$ defined later (see  Section \ref{subsec:define-P-sheaves} for definition), 
and we prove that it is stable under the left adjoint functor sequences
$(\pi_{2n,n}^{-1})_{n>0}$ and $(\mathcal{F}^\bullet_n\stackrel{L}{\otimes} \cdot)_{n>0}$, 
where $(\pi_{2n,n})_{n>0}$ is  an appropriate sequence of projection maps as above, 
and $( \mathcal{F}^\bullet_n)_{n > 0}$ is any
sequence of constructible sheaves belonging to the class $\bs{\mathcal{P}}_\R$
(see Proposition \ref{prop:stability-sheaf-P}  below).
The stability of $\bs{\mathcal{P}}_\R$ under the  two sequences of right adjoints of these functors, namely 
$(R\pi_{2n,n,*})_{n > 0}$ and $(\RHom(\mathcal{F}^\bullet_n,\cdot))_{n>0}$,
are left as open problems
(see Conjecture \ref{conj:main} and Question \ref{question:closure-under-hom} below), and we believe that stability does not hold in these cases. Indeed  Conjecture \ref{conj:main} is a sheaf-theoretic version of the usual
$\textbf{P}$ vs. $\textbf{NP}$, as well as the $\mathbf{P}$ vs. $\textbf{co-NP}$ problem,
in the classical theory.

Guided by the above examples, it is  interesting to speculate  whether one can define useful notions of  ``polynomially bounded complexity'' for sequences of objects in more  general categories and functor sequences.  Once the notion of ``complexity'' and the class 
of objects having ``polynomial complexity'' (i.e. the class $\textbf{P}$) are defined, one can ask,
given a sequence of functors that  preserves the class $\mathbf{P}$, whether the sequences of its (left or right)  adjoints (if they exist) also preserve the class $\textbf{P}$ (and thus obtain categorical generalizations of the classical   $\textbf{P}$ vs. $\textbf{NP}$ problem).  All this suggests the interesting  possibility of 
\emph{categorification  of complexity theory}. We do not pursue these ideas further in the current paper.
\\

The rest of the paper is organized as follows. In Section \ref{sec:constructible-functions}
we define new complexity classes of constructible functions, give some basic examples and
pose a question analogous to the $\mathbf{VP}$ vs. $\mathbf{VNP}$ conjecture in the discrete case.
In Section \ref{sec:sheaf-formulation}, we extend these notions to the category of constructible sheaves. We begin by giving in Section \ref{subsec:background-on-sheaves} a brief introduction to the basic definitions and results of sheaf
theory, especially those related to cohomology of sheaves, and derived category of complexes of sheaves with bounded cohomology, that we will need. The reader is referred to the books
\cite{KS, Dimca-sheaves, Iversen, Borel,Schurmann} for the missing details. In Section \ref{subsec:recall-PH} we recall the definitions
of the main complexity classes in the classical B-S-S setting. In Section \ref{subsec:define-P-sheaves} we define the new sheaf-theoretic complexity class $\bs{\mathcal{P}}_\R$. 
In Section \ref{subsec:define-PH-sheaves}, we extend 
the definition of $\bs{\mathcal{P}}_\R$ to a hierarchy, $\bs{\mathcal{PH}}_\R$, which mirrors the compact 
polynomial hierarchy $\mathbf{PH}^c_\R$. We also formulate the conjectures on separations
of sheaf-theoretic complexity classes analogous to the classical one and prove a relationship between these conjectures in Section \ref{subsec:inclusions}. 
In Section \ref{sec:topological-complexity-sheaves},
we prove a complexity result (Theorem \ref{thm:topological-complexity-sheaves})
bounding from above the \emph{topological complexity} (see Definition \ref{def:topological-complexity-sheaves} below)  of a sequence in the class $\bs{\mathcal{PH}}_\R$. More precisely, we prove that the topological complexity of sheaves in 
$\bs{\mathcal{PH}}_\R$ is bounded singly exponentially, mirroring a similar result in the classical
case. As a result we also obtain a singly exponential upper bound on the complexity of the ``direct image functor'' (Theorem \ref{thm:effective}) which is analogous to singly exponential upper bound results for effective quantifier elimination in the first order theory of the reals. This last result might be of interest independent of complexity theory because of its generality. Finally, in Section \ref{sec:connection-to-Toda}, we revisit Toda's theorem in the discrete as well as B-S-S setting, and conjecture a similar theorem in the sheaf-theoretic setting.

\section{Complexity theory of constructible functions}
\label{sec:constructible-functions}

\subsection{Main definitions}
\label{sec:main-definitions}

Our first goal is to develop a complexity theory
for constructible functions on $\R^n$ and $\C^n$  in the style of
Valiant's algebraic complexity theory. In particular, we define in the case of real and
complex numbers separately, and for each ``additive'' invariant on the classes of semi-algebraic
and constructible sets respectively (see Definition \ref{def:additive-invariant} below for the definition of additive invariants), 
two new 
complexity classes of sequences of such functions which should be considered
as ``constructible analogs'' of the classes $\mathbf{VP}$ and $\mathbf{VNP}$ defined by Valiant. 
Since our goal is to be as geometric as possible, the underlying geometric objects in these complexity classes are certain semi-algebraic (respectively, constructible) partitions of $\R^n$ (respectively, $\C^n$).
We formulate a series of conjectures (depending on the additive invariant that is chosen) that the class
of partitions corresponding to the second class (namely the analog of $\mathbf{VNP}$) is strictly larger
than the class of partitions corresponding to the first (namely the analog of $\mathbf{VP}$). Since
we are interesting in measuring complexity of partitions (induced by constructible functions) and not
in the complexity of computing polynomials our approach in defining these classes is different from that
of Valiant. Still, the definitions of the new classes which correspond to the Valiant's class $\mathbf{VP}$ is
very closely related to the original definition of Valiant. It is in the passage of going from the class $\mathbf{VP}$ to the potentially larger $\mathbf{VNP}$ where our theory diverges and is hopefully more geometric. 
We also develop a theory of completeness parallel to Valiant's theory and prove the existence of certain complete sequences of functions.   

We begin with a few definitions and some notation.

\begin{definition}[Constructible and semi-algebraic sets]
A \emph{constructible} subset of $\C^n$ is a finite union of subsets of $\C^n$, each  defined by a finite number of polynomial equations and inequations.
A  \emph{semi-algebraic} subset of $\R^n$ is a finite union of subsets of $\R^n$, each defined by a finite number of polynomial equations and \emph{inequalities}.
\end{definition}

\begin{remark}
\label{rem:constructible-is-sa}
Notice that by identifying $\C^n$ with $\R^{2n}$ by separating the real and imaginary parts, any constructible subset of $\C^n$ can
be thought of as a semi-algebraic subset of $\R^{2n}$.
\end{remark}

\begin{definition}[The value algebra]
For $\kk=\R,\C$, we say that a $\Z$-graded $\kk$-algebra $\kA=\oplus_{i \geq 0} \kA_i$ is polynomially bounded if for each $n\geq0$,
$\kA_{\leq n} = \oplus_{0 \leq i \leq n} \kA_i$ is finite dimensional as a $\kk$-vector space, and 
the Hilbert function of $\kA$, namely $\mathrm{Hilb}_{\kA}(n) = \dim_\kk(\kA_{n})$, is bounded by a polynomial in $n$. 
\end{definition}

\begin{remark}
In fact the only two graded algebras that will be important for us are the algebra $\kk$ itself (with the trivial grading), and the
algebra of polynomials $\kk[T]$ graded by degree.
\end{remark}

\begin{definition}[Constructible functions]
\label{def:constructible-function}
For $\kk=\R,\C$, let $\kA$ denote a polynomially bounded graded $\kk$-algebra.

A function $f:\C^n \rightarrow \kA$ is said to be a 
\emph{$\kA$-valued constructible function}  if it is a $\kA$-linear combination of 
the characteristic functions of a finite number of constructible subsets of $\C^n$.
Similarly,
a function $f:\R^n \rightarrow \kA$ is said to be a 
\emph{ $\kA$-valued constructible function} if it is a $\kA$-linear combination of 
the characteristic functions of a finite number of semi-algebraic subsets of $\R^n$.

For any semi-algebraic or a constructible set $S$, we will denote by $\mathbf{1}_S$ the characteristic function of the set $S$, which is a constructible function for any $\kA$.
\end{definition}

\begin{notation}
\label{not:partition-function}
Let $\kk$ be the field $\R$ or $\C$, and let $\kA$ denote a polynomially bounded graded $\kk$-algebra.
For any constructible function $f:\kk^n\rightarrow \kA$, we will denote by $\Par(f)$ the finite partition of $\kk^n$ into the different
level sets of $f$. Note that these level sets are semi-algebraic  subsets of $\R^n$ in the case $\kk=\R$ 
(respectively, constructible subsets of $\C^n$ in the case $\kk=\C$). 
If $\mathcal{A},\mathcal{A}'$ are two finite partitions of $\kk^n$, then we say $\mathcal{A}'$ is \emph{finer} than
$\mathcal{A}$ (denoted $\mathcal{A}'\prec\mathcal{A}$), if for every set  $A'$ belonging to the partition $\mathcal{A}'$, there exists a (unique) set $A$ belonging to $\mathcal{A}$, with $A'\subset A$. 
\end{notation}

\begin{remark}
\label{rem:finite-values}
Note that it is an immediate consequence of Definition \ref{def:constructible-function} that a $\kA$-valued constructible function $f:\kk^n\rightarrow \kA$ 
takes only a finite number of values in $\kk$. Moreover, it is often the induced partition, $\Par(f)$,  that
is of geometric interest, and not the precise values that the function takes. 
Given two $\kA$-valued constructible functions $f,\bar{f}:\kk^n \rightarrow \kA$, 
$\Par(\bar{f}) \prec \Par(f)$  if and only if for all $\x,\y\in\kk^n$,  $\bar{f}(\x) = \bar{f}(\y) \implies f(\x) = f(\y)$. Notice that if $\Par(\bar{f}) \prec \Par(f)$, then there exists a
function $h:\kA \rightarrow \kA$, such that $f = h\circ \bar{f}$. Moreover, in case 
$\kA = \kk$, since $f$ takes only a finite number of values, by Lagrange
interpolation formula, $h$ can be chosen to be a polynomial in $\kk[T]$, whose degree equals the number of distinct values taken by $\bar{f}$.
\end{remark}

\begin{remark}
Since the sum, product and constant multiples of 
constructible functions are again 
constructible, 
the set of constructible functions on $\R^n$ (respectively, $\C^n$)  is 
an (infinite-dimensional)
$\R$-algebra (respectively, $\C$-algebra).
\end{remark}

\begin{example}
The constant function $\mathbf{1}_{\R^n}$  (respectively, $\mathbf{1}_{\C^n}$)
as well as any multiple of it, are constructible.
\end{example}

\begin{example}[Rank function on matrices and tensors]
The function $\mbox{rk}_{m,n}: \R^{m \times n} \rightarrow \R$  (respectively, $\mbox{rk}_{m,n}: \C^{m \times n} \rightarrow \C$) which evaluates
to the rank of an $m \times n$ matrix with entries in $\R$  (respectively, $\C$) 
is an example of $\R$-valued (respectively, $\C$-valued) constructible function.
Similarly, the rank function of higher order tensors (see Definition \ref{def:ranksoftensors} below)  are also constructible.
\end{example}

We next define a notion of \emph{size} of a formula defining a constructible function that will be used
in defining different complexity classes. We first need some notation.

\begin{notation}
For $x\in\R$, we denote 
\begin{eqnarray*}
\sign(x) &=& 1, \mbox{ if } x>0,\\
             &=&  -1, \mbox{ if } x < 0,\\
             &=&  0,  \mbox{ if } x=0.
 \end{eqnarray*}
 For $x\in\C$, we denote 
\begin{eqnarray*}
\zero(x) &=& 1, \mbox{ if } x\neq 0,\\
            &=&  0,  \mbox{ if } x=0.
 \end{eqnarray*}
 \end{notation}

\begin{notation}
\label{not:realization}
Let $\mathcal{P} \subset \R[X_1,\ldots,X_n]$ be a finite family. We call $\sigma \in \{0,1,-1\}^{\mathcal{P}}$ to be a \emph{sign condition} on $\mathcal{P}$. 
Similarly, for a finite family  $\mathcal{P} \subset \C[X_1,\ldots,X_n]$, we call $\rho \in \{0,1\}^{\mathcal{P}}$ to be a \emph{zero pattern} on $\mathcal{P}$.
Given a sign condition $\sigma \in \{0,1,-1\}^{\mathcal{P}}$ and a semi-algebraic subset $S \subset\R^n$, we denote by $\RR(\sigma,S)$  the semi-algebraic set defined by 
\[
\RR(\sigma,S) = \{ \x \in S \; \mid \; \sign(P(\x)) = \sigma(P), \forall P\in\mathcal{P}\},
\]
and call $\RR(\sigma,S)$ the \emph{realization} of $\sigma$ on $S$.

Similarly, 
given a zero pattern $\rho \in \{0,1\}^{\mathcal{P}}$ and a constructible  subset $S \subset\C^n$, we denote by $\RR(\rho,S)$  the constructible set defined by 
\[
\RR(\rho,S) = \{ \x \in S \; \mid \; \zero(P(\x)) = \rho(P), \forall P\in\mathcal{P}\},
\]
and call $\RR(\rho,S)$ the \emph{realization} of $\rho$ on $S$.
We say that a sign  condition  $\sigma$ (respectively, zero pattern $\rho$) is \emph{realizable} on a semi-algebraic (respectively, constructible) set $S$ if $\RR(\sigma,S) \neq \emptyset$ (respectively, $\RR(\rho,S)\neq \emptyset$). 

More generally, for any first order formula $\phi$ with atoms of the form $P \{=,>,<\}0, P\in \mathcal{P}$ 
(respectively, $P \{=,\neq\}0, P\in \mathcal{P}$)
we denote by $\RR(\phi,S)$ (the \emph{realization of $\phi$ on $S$}) the semi-algebraic (respectively, constructible) set defined  by
\[
\RR(\phi,S) = \{ \x\in S \;\mid\; \phi(\x)\}.
\]

Given any polynomial $P \in \R[X_1,\ldots,X_n]$ (respectively, $P \in \C[X_1,\ldots,X_n]$),
 and a semi-algebraic set $S \subset \R^n$ (respectively, constructible subset $S \subset \C^n$), we will
denote by $\ZZ(P,S)$ the set of zeros of $P$ in $S$. More generally, for any
finite family of polynomials $\mathcal{P} \subset \R[X_1,\ldots,X_n]$ (respectively, $\mathcal{P} \subset \C[X_1,\ldots,X_n]$)
we will denote by 
$\ZZ(\mathcal{P},S)$ the set of common zeros of $\mathcal{P}$ in $S$. 
\end{notation}

\begin{definition}[Formulas defining constructible functions]
Let $\kA$ denote a polynomially bounded graded $\R$-algebra.
Formulas for $\kA$-valued constructible functions defined over $\R$ are defined inductively as follows.
\begin{enumerate}
\item
If $P\in\R[X_1,\ldots,X_n]$, 
then
$\mathbf{1}_{P = 0}, \mathbf{1}_{P > 0},\mathbf{1}_{P < 0}$ are formulas 
defining the characteristic function 
of the semi-algebraic sets $\RR(P=0,\R^n), \RR(P<0,\R^n), \RR(P>0,\R^n)$ respectively. 
\item
If $F_1,F_2$ are formulas defined over $\R$, and $c \in \kA$,  then so are $F_1 + F_2, F_1\cdot F_2, 
c \cdot F_1$.
\end{enumerate}
Formulas for constructible functions defined over $\C$ are defined similarly.
Let $\kA$ now denote a polynomially bounded graded $\C$-algebra.

\begin{enumerate}
\item
If $P\in\C[X_1,\ldots,X_n]$,  
then
$\mathbf{1}_{P = 0}, \mathbf{1}_{P \neq 0}$ are formulas 
defining the characteristic function 
of the constructible sets $\RR(P=0,\C^n), \RR(P\neq0,\C^n)$ respectively. 
\item
If $F_1,F_2$ are formulas defined over $\C$, and $c \in \kA$,  then so are $F_1 + F_2, F_1\cdot F_2, 
c \cdot F_1$.
\end{enumerate}
\end{definition}

We now define the size of a formula defining a constructible function.
We begin by defining the size of a polynomial over $\R$ and $\C$.

\begin{definition}[Size of a polynomial $P$]
\label{def:size-of-polynomial}
For $P \in \R[X_1,\ldots,X_n]$ (respectively, $P\in \C[X_1,\ldots,X_n]$), we define 
$\size(P)$ as the maximum of $\deg(P)$ and the length of the smallest straight-line program 
(see \cite[page 105, Definition 4.2]{Burgisser-book1} for definition) computing $P$.
\end{definition}

\begin{remark}
In particular, note that $\deg(P)$ is bounded by $\size(P)$, and thus for any sequence of polynomials 
$(P_n \in \kk[X_1,\ldots,X_n])_{n >0}$, such that the sequence $(\size(P_n))_{n > 0}$ is bounded by a polynomial in $n$,  we 
automatically have that the sequence $(\deg(P_n))_{n > 0}$ is also bounded by the same polynomial.
\end{remark}

We can now define a notion of size of a formula defining a constructible function.

\begin{definition}[Size of a formula defining a constructible function]
\label{def:size-of-formula}
Let $\kA_\R$ (respectively, $\kA_\C$) denote a polynomially bounded $\R$-algebra (respectively, 
$\C$-algebra).
The size of a formula defining a $\kA_\R$-valued (respectively, $\kA_\C$-valued) 
constructible function is defined inductively as follows.

\begin{enumerate}
\item
If $P \in \R[X_1,\ldots,X_n]$ (respectively, $P \in \C[X_1,\ldots,X_n]$), then
\[
\size(\mathbf{1}_{P=0})=\size(\mathbf{1}_{P>0})=\size(\mathbf{1}_{P<0}) := \size(P)
\] 
(respectively,
$
\size(\mathbf{1}_{P=0})=\size(\mathbf{1}_{P\neq0})\leq \size(P)
$). 
\item
If $F_1,F_2$ are formulas defined over $\R$  (respectively, $\C$), then
\[
\size(F_1 + F_2), \size(F_1\cdot F_2) :=  \size(F_1) +\size(F_2),
\]
\item
if $F$ is a formula defined over $\R$  (respectively, $\C$), and $h \in \kA_\R[T]$ (respectively, $h \in \kA_\C[T])$, then
\[
\size(h(F)) := \size(F).
\]
\end{enumerate}
\end{definition}

\begin{remark}
The last item in the above definition requires a remark.
Note that if $f$ is a constructible function defined over $\kk=\R,\C$, and $h \in \kA_\R[T],\kA_\C[T]$, then 
$\Par(f) \prec \Par(h\circ f)$. Since, it is the complexity of $\Par(f)$ that is geometrically more meaningful, rather than the
values taken by $f$, we choose to ignore the cost of the last operation. It also makes the theory simpler in places.  
\end{remark}

\subsection{Constructible analogs of $\mathbf{VP}_\kk$}
We now define sequences of constructible functions that will play a role similar to that
of $\mathbf{VP}_\R$ or $\mathbf{VP}_\C$ in Valiant's theory.

\begin{definition}[The classes  $\mathbf{VP}_\R^\kA$, $\mathbf{VP}_\C^\kA$]
Let $\kk$ denote either the field $\R$ or the field $\C$, and $\kA$ a polynomially bounded 
$\R$ or $\C$-algebra.
Let $m(n),\bar{m}(n)\in \Z[n]$ be any non-negative polynomial. 
We say that a sequence of $\kA$-valued constructible functions 
$(f_n:\kk^{m(n)} \rightarrow \kA_{\leq \bar{m}(n)})_{n>0}$ is in the class $\mathbf{VP}_{\kk}^\kA$
if for each $n > 0$ there exists a formula $F_n$ defining the constructible function $f_n$,
and such that $\size(F_n)$ is bounded polynomially in $n$.
\end{definition}

\begin{proposition}
\label{prop:equivalent-def}
Let $\kk$ denote either the field $\R$ or the field $\C$.
Let $m(n)\in \Z[n]$ be any non-negative polynomial. 
A sequence of constructible functions 
$(f_n:\kk^{m(n)} \rightarrow \kk)_{n>0}$ is in the class $\mathbf{VP}_{\kk}^\kk$
if for each $n > 0$ there exists a formula $F_n$ defining the constructible function $\bar{f}_n:\kk^{m(n)}\rightarrow \kk$,
such that $\size(F_n)$ is bounded polynomially in $n$,
and $\Par(\bar{f}_n) \prec \Par(f_n)$.
\end{proposition}

 \begin{proof}
 It follows from Remark \ref{rem:finite-values} that for each $n$ there exists $h_n \in \kk[T]$, such that
 $f_n = h_n \circ \bar{f}_n$. Moreover, from Definition \ref{def:size-of-formula}, it follows that 
 $\size(h_n \circ F_n) = \size(F_n)$.
 \end{proof} 
 
Let $\kk$ be the field $\R$ or $\C$. As an example of a sequence in the class $\mathbf{VP}_\kk^\kk$ we have the following. 
Recall that for $m,n \geq 0$, we denote by $\mathrm{rk}_{m,n}: \kk^{m \times n} \rightarrow \kk$ the function that
maps an $m\times n$ matrix $A$ to its rank.

\begin{theorem}
\label{thm:rank}
The sequence of functions 
$
\left(\mathrm{rk}_{n,n}: \kk^{n \times n} \rightarrow \kk\right)_{n > 0}
$
belongs to the class $\mathbf{VP}_\kk^\kk$.
\end{theorem}
\begin{proof}
Follows from the existence of efficient circuits for computing ranks of matrices \cite{Mulmuley}.
\end{proof}

Another illustrative example of the power  of this class is given by the following example.

\begin{example}
Let $\kk$ be either the field $\R$ or $\C$.
For $0 \leq i \leq n$, let $S_{n,i} \subset \R^n$ be 
the set of points of $\kk^n$ with exactly $i$ non-zero coordinates. Then 
the sequence of functions
\[ 
\left(f_n = \mathbf{1}_{S_{n,\lceil{n/2}\rceil}}\right)_{n > 0} 
\]
belongs to the class $\mathbf{VP}_{\kk}^\kk$.
To see this observe that  the constructible function 
\[
\bar{f}_n = \sum_{i=1}^{n} \mathbf{1}_{(X_i \neq 0)},
\]
is defined by a formula of linear size, and clearly $\Par(\bar{f}_n) \prec \Par(f_n)$. Now apply Proposition \ref{prop:equivalent-def}.
\end{example}

A small modification of the above example shows that a sequence of constructible functions 
$(f_n: \kk^n \rightarrow \kk)_{n>0}$ can belong to the class $\mathbf{VP}_{\kk}^\kk$, even
though $\card(\Par(f_n))$ grows exponentially.

\begin{example}
Let $f_n: \kk^n \rightarrow \kk$ be defined by 
\[
f_n = \sum_{i=0}^{n-1} 2^i \cdot\mathbf{1}_{X_i \neq 0}.
\]
Then, clearly the sequence 
$(f_n)_{n>0}$ belongs to the class $\mathbf{VP}_{\kk}^\kk$,
but $\card(\Par(f_n)) = 2^n$.
\end{example}

\begin{remark}
Notice that our definition of the classes $\mathbf{V}_\kk^\kA$ is very geometric, and it is the partitions induced by
a sequence of functions that play the key role in this definition. Consider the following example.
Let for each $n >0$
\[
P_n = \sum_{i=1}^{n} X_i^2(X_i-1)^2,
\]
and consider
the following sequence of constructible functions
$(f_n: \R^n \rightarrow \R)_{n>0}$ defined by  
\[
f_n = \mathbf{1}_{P_n = 0}\cdot(\mathbf{1}_{\R^n}+ \sum_{i=1}^n 2^i \cdot \mathbf{1}_{X_i=0}).
\]
Clearly, the sequence 
$(f_n: \R^n \rightarrow \R)_{n>0}$
belongs to the class $\mathbf{VP}_\R^\R$, and 
$\Par(f_n)$ consists of singletons each containing a point in $\{0,1\}^n$, as well as the set $\R^n \setminus \{0,1\}^n$.
It follows that for \emph{any}  sequence of subsets $(S_n \subset \{0,1\}^n\subset \R^n)_{n > 0}$, the corresponding sequence of characteristic functions 
$(\mathbf{1}_{S_n})_{n > 0}$ belongs to $\mathbf{VP}_\R^\R$. Thus, membership in $\mathbf{VP}_\R^\R$
is \emph{not} a good measure of tractability of Boolean functions. 
\end{remark}

\subsection{Constructible analogs of $\mathbf{VNP}_\kk$}
 In this section we define the analogs of the Valiant complexity classes $\mathbf{VNP}_\R$ and
 $\mathbf{VNP}_\C$. Before doing so we first need to introduce the Grothendieck rings,
 $K_0(\sa_\R)$ and $K_0(\Var_\C)$, of semi-algebraic and constructible sets, and additive (as well as
 multiplicative) invariants semi-algebraic and constructible sets.
 
\subsection{Grothendieck rings and additive invariants}
\begin{definition}[Grothendieck rings]
\label{def:Grothendieck-ring}
The Grothendieck ring $K_0(\sa_\R)$ of semi-algebraic sets is defined as follows. The
underlying additive group is generated by the classes $[X]$, where $X$ is a semi-algebraic set
with the following relations:
\begin{enumerate}
\item $[X] = [Y]$ if $X$ and $Y$ are isomorphic as varieties;
\item $[X] = [X\setminus Y] + [Y]$ for each closed subset $Y$ of $X$ (in the Euclidean topology).
\end{enumerate}The multiplication operation is defined by 
\[
[X]\cdot[Y] = [X \times Y].
\]

Similarly, the Grothendieck ring $K_0(\Var_\C)$ of complex algebraic varieties is defined as follows. The
underlying additive group is generated by the classes $[X]$, where $X$ is a complex algebraic variety, 
with the following relations:
\begin{enumerate}
\item $[X] = [Y]$ if $X$ and $Y$ are isomorphic as varieties;
\item $[X] = [X\setminus Y] + [Y]$ for each closed subset $Y$ of $X$ (in the Zariski topology).
\end{enumerate}
The multiplication operation is defined by 
\[
[X]\cdot[Y] = [X \times Y].
\]
\end{definition}

 We also need the notion of an additive invariant of classes of semi-algebraic or
 constructible sets.
 
 \begin{definition}[Additive invariants of classes of semi-algebraic or
 constructible sets]
 \label{def:additive-invariant}
 Let $\kk=\R,\C$ and $\kA$ a $\kk$-algebra. Let $K_0 = K_0(\sa_\R)$ if $\kk=\R$ and
 $K_0 = K_0(\Var_\C)$ if $\kk = \C$. We call a ring homomorphism $\inv:K_0 \rightarrow \kA$
 to be \emph{additive invariant} of the class of semi-algebraic or constructible sets.
 \end{definition}
 
 Given an additive invariant, there is a well defined notion of integrating a constructible function
 with respect to the invariant. 
 
 More precisely:
 \begin{definition}[Integration with respect to an additive invariant]
\label{def:integral-EP}
Let $\kk$ be either the field $\R$, or the field $\C$, and $K_0 = K_0(\sa_\R)$ in case $\kk =\R$, and
$K_0= K_0(\Var_\C)$ in case $\kk = \C$.   Let $\inv: K_0\rightarrow \kA$  be an
additive invariant, and 
let $f: \kk^n \rightarrow \kA$ be a $\kA$-valued constructible function 
defined by 
\[
f = \sum_{i=1}^{N} a_i \mathbf{1}_{X_i},
\]
where the $X_i$'s are semi-algebraic sets in case $\kk=\R$ and constructible sets if $\kk = \C$, and
each $a_i \in \kA$.
We define the \emph{integral of $f$ 
with respect to the additive invariant $\inv$ } (following \cite{Viro-euler,Schapira89, Schapira91}) to be 
\[
\int_{\kk^n} f \dd\inv :=  \sum_{i=1}^{N} a_i \inv(X_i).
\]
\end{definition}

\begin{remark}
The fact that the definition of $\int_{\kk^n} f \dd\inv$ is independent of the particular representation of the
constructible function $f$ (which is far from being unique) is a classical fact \cite{Viro-euler,Schapira89,Schapira91}. 
The  integral defined above satisfies all the usual properties (of say the Lebesgue integral) such as
additivity, Fubini-type theorem etc. \cite{Viro-euler, Schapira89, Schapira91}, and in particular can be used to define ``push-forwards'' 
of (constructible) functions via fiber-wise integration.
\end{remark}

We are now in a position to define the geometric analogs of Valiant's classes $\mathbf{VNP}_\R$ and
$\mathbf{VNP}_\C$.

 \subsubsection{The classes  $\mathbf{VNP}_\kk^\inv$}
 Let $\kk$ be the field $\R$ (respectively, $\C$),
 $K_0 = K_0(\sa_\R)$ (respectively, $K_0=K_0(\Var_\C)$), 
 $\kA$ a polynomially bounded graded $\kk$-algebra, 
 and $\inv:K_0\rightarrow \kA$ be an additive invariant.

\begin{definition}[The class  $\mathbf{VNP}_\kk^\inv$]
\label{def:VNP}
Let $m(n),\bar{m}(n),\bar{m}_1(n)$ be non-negative polynomials with
integer coefficients.
We say that a sequence of constructible functions 
$(f_n:\kk^{m(n)} \rightarrow \kA_{\leq \bar{m}(n)})_{n > 0}$ is in the class $\mathbf{VNP}_\kk^\inv$
if there exists a sequence of constructible functions
$(g_n:\kk^n \rightarrow \kA_{\leq \bar{m}_1(n)})_{n > 0}$ belonging to the class $\mathbf{VP}_\kk^\kA$,
and a non-negative polynomial $m_1(n)$ such that
for each $n > 0$, $\Par(\bar{f}_n) \prec \Par(f_n)$,
where $\bar{f}_n:\kk^{m(n)} \rightarrow \kA_{\leq \bar{m}(n)}$ is the $\kA$-valued constructible function defined by
\begin{equation}
\bar{f}_n(\x) =  \int_{\kk^{m_1(n)}} g_{m(n) + m_1(n)}(\cdot,\x) \dd\inv.
\end{equation}
\end{definition}

 Here are two key examples.
 
 \subsubsection{The (generalized) Euler-Poincar\'e characteristic of semi-algebraic sets}
 \label{subsubsec:generalizedP}
\begin{definition}[Generalized Euler-Poincar\'e characteristic]
\label{def:gen_EP}
The generalized Euler-Poincar\'e characteristic, $\chi(S)$,  of a semi-algebraic set $S \subset \R^k$ 
is uniquely defined by the following properties \cite[Chapter 4]{Dries}:
\begin{enumerate}
\item
$\chi$ is invariant under semi-algebraic homeomorphisms.
\item
\[ \chi(\{\pt\}) = \chi([0,1]) = 1.\]
\item
$\chi$ is multiplicative, i.e., $\chi(A \times B) = \chi(A)\cdot\chi(B)$ for any pair of semi-algebraic sets $A,B$.
\item
\label{item:additive}
$\chi$ is additive, i.e., $\chi(A \cup B) = \chi(A) + \chi(B) - \chi(A\cap B)$ for any pair of semi-algebraic subsets $A,B \subset \R^n$.
\end{enumerate}
\end{definition}

\begin{remark}
\label{rem:alternating-sum}
Note that the generalized Euler-Poincar\'e characteristic is a homeomorphism
(but \emph{not} a homotopy) invariant. 
For a locally  closed semi-algebraic set $X$, 
\[
\chi(X) = \sum_{i \geq 0} (-1)^i \dim_\Q \HH_c^i(X,\Q),
\]
where $\HH_c^i(X,\Q)$ is the the $i$-th co-homology group of $X$ with compact support
(see Definition \ref{def:cohomology-compact-support}).
Thus, the definition agrees with the usual 
Euler-Poincar\'e characteristic as an alternating sum of the Betti numbers 
for locally closed semi-algebraic sets.
\end{remark}

A few illustrative examples are given below.

\begin{notation}
\label{not:ball-and-sphere}
We denote by $\Ball^n(0,r)$ the open ball in $\R^n$ of radius $r$ centered at the
origin. We will denote by $\Ball^n$ the  open unit ball $\Ball^n(0,1)$.
Similarly, we denote by $\Sphere^{n-1}(0,r)$ the sphere in $\R^{n}$ 
of radius $r$ centered at the origin, and 
by $\Sphere^{n-1}$ the unit sphere $\Sphere^{n-1}(0,1)$.
\end{notation}

\begin{example}
For every $n \geq 0$,
\begin{enumerate}
\item
 \[
\chi(\overline{\Ball^n}) = \chi([0,1]^n) = \chi([0,1])^n = 1.
\]
\item
\[
\chi(\Ball_n) = \chi((0,1)^n) = (\chi(0,1))^n = (\chi([0,1]) - \chi({0}) - \chi({1}))^n = (-1)^n.
\]
\item
\[
\chi(\Sphere^{n-1}) = \chi(\overline{\Ball^n}) - \chi(\Ball^n) = 1 - (-1)^n.
\]
\end{enumerate}
\end{example}

It is obvious from its definition that
\begin{proposition}
The generalized Euler-Poincar\'e characteristic is an additive invariant of the class
of semi-algebraic sets.
\end{proposition}

It was mentioned in the introduction that one difficulty in defining a push-forward of functions in the
B-S-S model had to do with the impossibility of computing exactly integrals with respect to usual
measures on $\R^n$ (such as the Lebesgue measure), since such integrals could be transcendental
numbers or might not converge. 
In contrast, we have the following effective upper bound on the complexity of 
computing integrals with respect to the generalized Euler-Poincar\'e characteristic.

\begin{theorem}
\label{thm:algorithm-for-integration-EP}
There exists an algorithm that takes as input a formula $F$ describing 
a constructible function $f:\R^n\rightarrow \R$, and computes
\[
\int_{\R^n} f \dd\chi.
\] 
The complexity of the algorithm measured as the number of arithmetic
operations over $\R$ as well as comparisons 
is bounded singly exponentially in $n$ and the size of the formula $F$. 
\end{theorem}

\begin{proof}
It is easy to verify using
induction on the size of the formula $F$, that 
there exists  a family of polynomials 
$\mathcal{P}_F \subset \R[X_1,\ldots,X_n]$, 
such that $\card(\mathcal{P}_F)$, as well as the degrees of the polynomials in
$\mathcal{P}_F$, are bounded by $\size(F)$. Moreover
$f$ can be expressed as a linear combination of the characteristic functions of the realizations
of the various sign conditions on $\mathcal{P}_F$. More precisely, there is an expression
\[
f = \sum_{\sigma \in \{0,1,-1\}^{\mathcal{P}_F}} a_\sigma \mathbf{1}_{\RR(\sigma,\R^n)},
\]
where the $a_\sigma \in \R$. Moreover, the set of $a_\sigma$'s can be computed from $F$ with complexity
singly exponential in $\size(F)$.
From Definition \ref{def:integral-EP} it follows that
\[
\int_{\R^n}  f \dd\chi = \sum_{\sigma \in \{0,1,-1\}^{\mathcal{P}_F}} a_\sigma \chi(\RR(\sigma,\R^n)).
\]
It follows from the main result in \cite{BPR-euler-poincare} (see also \cite[Algorithm 13.5]{BPRbook2}) 
that the list 
\[
\Big(\chi(\RR(\sigma,\R^n))\Big)_{\sigma \in \{0,1,-1\}^{\mathcal{P}_F}, \RR(\sigma,\R^n) \neq\emptyset}
\] 
can be computed with complexity  
\[
\Big(\card(\mathcal{P}_F) (\max_{P \in \mathcal{P}_F} \deg(P))\Big)^{O(n)}.
\] 
Since, $\card(\mathcal{P}_F), \max_{P \in \mathcal{P}_F} \deg(P) \leq \size(F)$,  the result follows.
\end{proof}

\subsubsection{Uniform bounds on the generalized Euler-Poincar\'e characteristic of semi-algebraic sets}
\label{subsubsec:bounding-generalized-chi}

The following proposition giving a uniform bound on the generalized Euler-Poincar\'e characteristic of semi-algebraic sets 
will be useful later.

We first need a notation. 
\begin{notation}
\label{not:P-formula}
Let $\mathcal{P} \subset \R[X_1,\ldots,X_n]$ be a finite set of polynomials. We say that a semi-algebraic subset $S \subset \R^n$ is a \emph{$\mathcal{P}$-semi-algebraic set},  if $S$ is defined by a quantifier-free first order formula $\Phi$ with atoms $P \{=,<,>\} 0, P \in \mathcal{P}$. We call $\Phi$ a \emph{$\mathcal{P}$-formula}.
We say that $S$ is a 
\emph{$\mathcal{P}$-closed semi-algebraic
set},  if $S$ is defined by a quantifier-free first order formula $\Psi$ with no negations and 
atoms $P \{\leq,\geq\} 0, P \in \mathcal{P}$. We call $\Psi$ a \emph{$\mathcal{P}$-closed formula}.

Similarly, if $\mathcal{P} \subset  \C[X_1,\ldots,X_n]$ is a finite set of polynomials, we say that a constructible subset $S \subset \C^n$ is a \emph{$\mathcal{P}$-constructible set},  if $S$ is defined by a quantifier-free first order formula $\Phi$ with atoms $P \{=,\neq\} 0, P \in \mathcal{P}$. 
\end{notation}

\begin{proposition}
\label{prop:uniform-bound-real}
Let $\mathcal{P} \subset \R[X_1,\ldots,X_n]$ be a finite set of polynomials, with $\card(\mathcal{P}) = s$, and $d = \max(\max_{P \in \mathcal{P}} \deg(P),2)$. 
Let $\sigma \in \{0,1,-1\}^{\mathcal{P}}$. Then, 
\[
|\chi(\RR(\sigma,\R^n)| \leq \sum_{i=0}^{n-1} \sum_{j=0}^{i+1}  \binom{s+1}{j}d(2d -1)^{n-1}  = (O(sd))^n.
\]
More generally, let $S$ be any $\mathcal{P}$-semi-algebraic subset of $\R^n$. Then,
\[
|\chi(S)| =  (O(s d))^{2n}.
\]
\end{proposition}

\begin{remark}
A  result  similar to Proposition \ref{prop:uniform-bound-real} in the case of locally closed semi-algebraic sets can also be deduced from \cite[Theorem 1.10]{Pardo96}. However, the implied constants are slightly better in Proposition \ref{prop:uniform-bound-real}, and it applies to all semi-algebraic sets, not just to locally closed ones.
\end{remark}

In the proof of Proposition \ref{prop:uniform-bound-real} we will need the following notation and
result.
 
\begin{notation}
\label{not:betti}
If $X$ is a locally closed semi-algebraic set then we denote
\begin{eqnarray*}
b^i(X) &=& \dim_\Q \HH^i(X,\Q), \\
b^i_c(X) &=& \dim_\Q \HH^i_c(X,\Q), 
\end{eqnarray*}
where $\HH^i(X,\Q)$ (respectively, $\HH^i_c(X,\Q)$) is the $i$-th cohomology group (respectively,
the $i$-th cohomology group with compact support) of $X$ with coefficients in $\Q$
(see Definition \ref{def:cohomology-compact-support}).

We denote 
\begin{eqnarray*} 
b(X)    &=& \sum_i b^i(X), \\
b_c(X)    &=& \sum_i b^i_c(X).
\end{eqnarray*}
\end{notation}

One important property of cohomology groups with compact supports is the following.

\begin{theorem}
\label{thm:inequality-compact-support}
\cite[III.7, page 185]{Iversen}
Let $X$ be a locally closed semi-algebraic set and $Z\subset X$ a closed subset and let $U = X\setminus Z$.
Then, there exists a long exact sequence 
\[
\cdots\rightarrow \HH^{p-1}_c(Z,\Q) \rightarrow \HH_c^p(U,\Q) \rightarrow \HH_c^p(X,\Q) \rightarrow \HH_c^p(Z,\Q) \rightarrow \cdots
\]
In particular,
\begin{equation}
\label{eqn:inequality-compact-support}
b_c(U) \leq b_c(X) + b_c(Z).
\end{equation}
\end{theorem}

We are now ready to prove Proposition \ref{prop:uniform-bound-real}.
\begin{proof}[Proof of Proposition \ref{prop:uniform-bound-real}]
Let $R>0$ be a real number which will be chosen sufficiently large later.

It follows from Hardt's
triviality theorem (see \cite[Theorem 9.3.2.]{BCR}) that for all large enough $R>0$, $\RR(\sigma,\R^n)$ is semi-algebraically
homeomorphic to $\RR(\sigma,\R^n) \cap \Ball_n(0,R)$. 

Hence, in particular since $\chi(\cdot)$ is a homeomorphism invariant,
\begin{equation}
\label{eqn:inside-big-ball}
\chi(\RR(\sigma,\R^n)) =\chi(\RR(\sigma,\R^n) \cap \Ball_n(0,R)).
\end{equation}

Let 
$U_{\sigma,R}$ be the semi-algebraic subset of $\R^n$ defined by
\[
\biggl(\bigwedge_{P \in \mathcal{P}, \sigma(P) = 0} (P=0)\biggr) \wedge \biggl(\bigwedge_{P \in \mathcal{P}, \sigma(P) \neq 0} (P\neq 0) \biggr)\wedge  \bigl( |\X^2| < R^2 \bigr), 
\]
$Z_{\sigma,R}$ be the semi-algebraic subset of $\R^n$ defined by
\[
\biggl(\bigwedge_{P \in \mathcal{P}, \sigma(P) = 0} (P=0)\biggr)  \wedge \biggl(\bigvee_{P \in \mathcal{P}, \sigma(P) \neq 0} (P= 0) \vee \bigl(|\X|^2=R^2\bigr)\biggr),
\]
and let $X_{\sigma,R}$ be the semi-algebraic subset of $\R^n$ defined by 
 \[
 \biggl(\bigwedge_{P \in \mathcal{P}, \sigma(P) = 0} (P=0)\biggr) \wedge \bigl(|\X^2|\leq R^2\bigr).
 \]
 Note that $X_{\sigma,R}$ is compact, $Z_{\sigma,R} \subset X_{\sigma,R}$ is a closed subset, and $U_{\sigma,R} = X_{\sigma,R} \setminus Z_{\sigma,R}$.
  
 It follows from an application of standard bounds on the Betti numbers of real algebraic varieties \cite{OP,T,Milnor2} that
 \begin{equation}
 \label{eqn:X}
 b(X_{\sigma,R})  \leq  d(2d -1)^{n-1}.
 \end{equation}
 
Using Mayer-Vietoris inequalities \cite[Proposition 7.33]{BPRbook2} and the bound on the Betti numbers of  real varieties used above we obtain
 
 \begin{equation}
 \label{eqn:Z}
 b(Z_{\sigma,R})  \leq  \sum_{i=0}^{n-1} \sum_{j=1}^{i+1}  \binom{s+1}{j}d(2d -1)^{n-1}.
 \end{equation}
 
 It follows that for all $R$ large enough and  positive that,
\begin{eqnarray}
\label{eqn:uniform-bound}
\nonumber
|\chi(\RR(\sigma,\R^n))| &=& |\chi(\RR(\sigma,\R^n) \cap \Ball_n(0,R))|  \mbox{ (using \ref{eqn:inside-big-ball})} \\
\nonumber
&=& |\sum_{i \geq 0} (-1)^i b^i_c(\RR(\sigma,\R^n) \cap \Ball_n(0,R))|  \mbox{ (using Remark \ref{rem:alternating-sum}) }\\ 
\nonumber
& \leq & b_c(\RR(\sigma,\R^n) \cap \Ball_n(0,R))  \\
\nonumber
&\leq &  b_c(U_{\sigma,R})  \mbox{ (using inequality \ref{eqn:inequality-compact-support} in Theorem \ref{thm:inequality-compact-support})} \\
\nonumber
&\leq & b_c(Z_{\sigma,R}) + b_c(X_{\sigma,R})  \mbox{ (using inequality \ref{eqn:inequality-compact-support} in Theorem \ref{thm:inequality-compact-support})} \\
\nonumber
&\leq & b(Z_{\sigma,R}) + b(X_{\sigma,R})  \mbox{ (since $Z_{\sigma,R}$ and $X_{\sigma,R}$ are both compact)} \\
\nonumber
&\leq& \sum_{i=0}^{n-1} \sum_{j=1}^{i+1}  \binom{s+1}{j}d(2d -1)^{n-1} + d(2d -1)^{n-1} \mbox{ (using \ref{eqn:Z} and \ref{eqn:X})}\\
\nonumber
&\leq& \sum_{i=0}^{n-1} \sum_{j=0}^{i+1}  \binom{s+1}{j}d(2d -1)^{n-1} \\
&= & (O(sd))^n. 
\end{eqnarray}

Recall that every $\mathcal{P}$-semi-algebraic set is a disjoint union of sets of the form
$\RR(\sigma,\R^n), \sigma \in \{0,1,-1\}^{\mathcal{P}}$.  Also, the cardinality of the set
$\{\sigma \in \{0,1,-1\}^{\mathcal{P}} \;\mid\; \RR(\sigma,\R^n) \neq \emptyset\}$ is bounded by 
$O(sd)^n$ (see for example \cite[Theorem 7.30]{BPRbook2}).
It now follows from Property (\ref{item:additive}) in Definition \ref{def:gen_EP} 
(additive property)  of the generalized Euler-Poincar\'e characteristic 
and \eqref{eqn:uniform-bound} above, that
for any $\mathcal{P}$-semi-algebraic set $S$, 
\[
|\chi(S)| =  (O(sd))^{2 n}.
\]
\end{proof}

\begin{question}
Is it true that (using the notation as in Proposition \ref{prop:uniform-bound-real}) that for any $\mathcal{P}$-semi-algebraic set $S \subset \R^n$,
$\chi(S) = (O(sd))^n$ (instead of $(O(sd))^{2n}$) ?
Notice that this is true when $S$ is $\mathcal{P}$-closed,  and this follows from  known bounds on the Betti
numbers of such sets \cite[Theorem 7.38]{BPRbook2}). 
For the problem of bounding the ordinary Betti numbers of an arbitrary $\mathcal{P}$-semi-algebraic set, one first reduces to the compact
case using a result of Gabrielov and Vorobjov  
\cite{GV07}, by replacing the given set by a compact one homotopy equivalent to it (see \cite{GV07}).
However, since the generalized Euler-Poincar\'e characteristic is \emph{not}  a homotopy invariant but only a homeomorphism invariant, 
this technique is not directly applicable in this situation.
\end{question}

\begin{remark}
Also, note that for the purposes of the current paper it is only the existence of a uniform bound in Proposition \ref{prop:uniform-bound-real}
(i.e., a bound depending just on the parameters
$s$ and $d$ and not on the coefficients of the polynomials in $\mathcal{P}$) which is important. However, since in our opinion proving a 
tight singly exponential bound on the generalized Euler-Poincar\'e characteristic is an interesting problem on its own,  and to our knowledge
does not appear in the literature, we decided to include this result.
In fact, as we will see later,  in the complex case we will need a
corresponding uniform bound on the virtual Betti numbers of constructible sets
(see Definition  \ref{def:virtual-Betti-numbers}).  The bound that we will use in the complex case
is not effective and thus substantially weaker (see Proposition \ref{prop:uniform-bound} and Question \ref{question:effective-bound}).
\end{remark}

A  first example of a sequence in $\mathbf{VNP}_\R^\inv$ is as follows.

\begin{example}
For each $n,d > 0$ let $V_{n,d} \cong \R^{\binom{n+d}{d}}$ denote the
vector-space of polynomials
in $\R[X_1,\ldots,X_n]$ of degree at most $d$, and let 
$\chi_{n,d}: V_{n,d} \rightarrow \R$ be the constructible function defined by 

\begin{equation}
\label{eqn:def:EP}
\chi_{n,d}(P) = \chi(\ZZ(P,\R^n)).
\end{equation}

It is an easy exercise to show that for each fixed $d > 0$, the
sequence of functions $(\chi_{n,d}: V_{n,d} \rightarrow \R)_{n > 0}$ 
is in $\mathbf{VNP}_\R^\chi$. Indeed, letting 
$S_{n,d} \subset \R^n \times V_{n,d} $ the incidence variety defined by 
\[
S_{n,d} = \{(\x, P) \; \mid\; P(\x) = 0\},
\]
it is easy to check that the sequence $(\mathbf{1}_{S_{n,d}})_{n>0} \in \mathbf{VP}_\R^\R$
(notice that $\dim(V_{n,d})  = \binom{n+d}{d}$ is bounded by a polynomial in $n$ of degree $d$).

Then, 
\[
\chi_{n,d}(P) = \int_{\R^n} \mathbf{1}_{S_{n,d}}(\cdot, P) \dd\chi,
\]
which implies that the  sequence $(\chi_{n,d}: V_{n,d} \rightarrow \R)_{n > 0} \in
\mathbf{VNP}_\R^\chi$ using
Definition \ref{def:VNP}.
\end{example}

\subsubsection{The virtual Poincar\'e polynomial of constructible sets}
\label{subsubsec:virtualP}
We now consider the complex case. In this case the generalized Euler-Poincar\'e characteristic is still an additive invariant
(considering a constructible subset of $\C^n$ as a semi-algebraic set cf. Remark \ref{rem:constructible-is-sa}), 
but a more sophisticated algebro-geometric
invariant, originating in the theory of mixed Hodge structures on complex varieties due to Deligne \cite{Del71,Del74} 
is available. This is the virtual Poincar\'e polynomial whose existence and main properties are stated in the theorem below.
 
\begin{theorem}
\label{thm:properties-of-vP}
There exists an additive invariant 
\[
\vP: K_0(\Var_\C) \rightarrow \Z[T],
\]
such that for any smooth, projective variety $X$, 
\[
\vP([X]) = P_X
\]
where $P_X(T) = \sum_{i\geq 0} b^i(X) T^i$ is the usual Poincar\'e polynomial
of $X$.
\end{theorem}

\begin{proof}
See for example \cite[Corollary 2.1.8]{peters2010motivic}.
\end{proof}

\begin{definition}
\label{def:virtual-Betti-numbers}
For a constructible subset $X \subset \C^n$, the coefficients of $\vP([X])$ are referred to as the 
\emph{virtual Betti numbers of $X$}.
\end{definition}

Any constructible subset $X \subset \C^n \subset \PP^n_\C$, can be written as a finite, disjoint
union of quasi-projective varieties $X_i, i\in I$, each $X_i = V_i\setminus W_i$  where $V_i,W_i$ are
sub-varieties of of $\PP^n_\C$, and thus represents a well-defined class,
\[
[X] := \sum_{i \in I} ([V_i] - [W_i])
\]
in $K_0(\Var_\C)$.
Moreover, as a consequence of the weak factorization theorem \cite{AKMW} 
there exists for each 
$i$, a smooth projective variety $\tilde{V}_i$ and a simple, normal, crossing divisor $\tilde{W}_i$ such
that $[V_i] -[W_i] = [\tilde{V}_i] - [\tilde{W}_i]$ in  $K_0(\Var_\C)$ (see also \cite{Bittner}). In particular,
\[
\vP([V_i] - [W_i]) = \vP([\tilde{V_i}]) - \vP([\tilde{W_i}]) = P_{V_i}(T) - \vP([\tilde{W_i}]).
\]
Since $\dim(\tilde{W}_i)  < \dim(\tilde{V_i})$, we have that 
the leading coefficient of $\vP([V_i] - [W_i])$ is $>0$. This also shows that
\begin{proposition}
\label{prop:non-zero}
If $X \subset \C^n$ is a non-empty constructible set, then
$\vP([X]) \neq 0$. 
\end{proposition}

\begin{remark}
Note that Proposition \ref{prop:non-zero} does not hold if we replace $\vP$ by $\chi$. For example,
\begin{eqnarray*}
\chi(\C \setminus \{0\}) &=& \chi(\PP^1_\C - \{\infty,0\}) \\
& =& \chi(\Sphere^2) - 2 \chi(\{\mathrm{\pt}\}) \\
&=& 2-2\cdot 1 \\
&=& 0.
\end{eqnarray*}
On the other hand using the properties of $\vP$ stated in Theorem  \ref{thm:properties-of-vP} we have that
\begin{eqnarray*}
\vP(\C\setminus \{0\}) &=& \vP(\PP^1_\C \setminus \{\infty,0\}) \\
&=& 1 + T^2 -2 \\
&=& -1 + T^2.
\end{eqnarray*}
\end{remark}

We will also need the following proposition.
\begin{proposition}
\label{prop:uniform-bound}
There exists a function $M(s,d,n)$ such that if
$V,W \subset \PP^n_\C$ are projective varieties defined by a family of $s$ homogeneous polynomials of degree at most $d$, then the absolute values of the coefficients of $\vP([V \setminus W])$ are bounded by $M(s,d,n)$.
\end{proposition}
\begin{proof}
The existence of the uniform bound $M(s,d,n)$ follows from the complexity analysis of effective desingularization algorithms (see for example \cite{BGMW2012})  and bounds on the ordinary Betti numbers of smooth varieties in terms of the degrees of the defining polynomials \cite{BPRbook2}.
\end{proof}

\begin{question}
\label{question:effective-bound}
Is it possible to prove  a better (say even singly exponential) bound on the virtual Betti numbers (i.e. the coefficients of $\vP([S])$) of a $\mathcal{P}$-constructible
subset $S \subset \C^n$ in terms of the degrees and the number of polynomials in $\mathcal{P}$ ? 
\end{question}

\begin{remark}
Notice that unlike the usual Betti numbers of complex varieties, we do not have good singly exponential
bounds on the virtual Betti numbers. 
\end{remark}

An interesting example of a sequence of $\C \hookrightarrow\C[T]$-valued constructible functions which belong to
the class $\mathbf{VNP}_\C^{\vP}$ is the following.  

\begin{definition}[Ranks of tensors]
\label{def:ranksoftensors}
For $p,n \geq 0$, and $\tb \in \C^n \underbrace{\otimes \cdots \otimes}_{p} \C^n = (\C^n)^{\otimes p}$, we define $\rank_{n,p}(\tb)$ as the least number $r$, such that $\tb$ can be expressed as a sum
of $r$ decomposable tensors. A tensor $\tb \in (\C^n)^{\otimes p}$ is decomposable if there exists
$\vb_1,\ldots,\vb_p \in \C^n$ such that
\[
\tb = \vb_1 \otimes \cdots \otimes \vb_p.
\]  
\end{definition}

\begin{theorem}
\label{thm:rankoftensor}
For any fixed $p \geq 0$, the sequence of constructible functions
$(\rank_{n,p}: (\C^n)^{\otimes p}\rightarrow \C\hookrightarrow\C[T])_{n > 0}$ belongs to the class 
$\mathbf{VNP}_\C^\vP$.
\end{theorem}

The proof of Theorem \ref{thm:rankoftensor} will depend on the following simple lemma.

\begin{lemma}
\label{lem:injective}
Let $M \geq 2$ be an integer. Then, for any $n >0$, the map $\phi_n: (-M, M)^n \rightarrow \Z$ defined
by $(a_0,\ldots,a_{n-1}) \mapsto \sum_{i=0}^{n-1} a_i M^{2i}$ is injective.
\end{lemma}

\begin{proof}
Let $\mathbf{ a}= (a_0,\ldots,a_{n-1}),\mathbf{b} = (b_0,\ldots,b_{n-1}) \in (-M,M)^n$, $\mathbf{a} \neq \mathbf{b}$. Let 
$p$ be the largest index such that $a_i \neq b_i$. Without loss of generality assume that $a_p > b_p$. Then,
\begin{eqnarray*}
\phi_n(\mathbf{a}) -\phi_n(\mathbf{b}) &=& (a_p - b_p)M^{2p} + \sum_{i=0}^{p-1} (a_i - b_i)M^{2i} \\
                            &\geq& M^{2p} - 2(M-1) \sum_{i=0}^{p-1}  M^{2i} \\
                            &=&  M^{2p} - \frac{2(M^{2p} -1)}{M+1} \\
                            &=& \frac{M^{2p}(M -1) +2}{M+1} \\
                            &>& 0.
\end{eqnarray*}                            
\end{proof}

\begin{proof}[Proof of Theorem \ref{thm:rankoftensor}]
Let $V_n$ denote the vector space $\C^n$.
First note that for $\tb \in V_n^{\otimes p}$, $0\leq \rank_{n,p}(\tb) \leq n^{p-1}$.
For each $i, 0 \leq i \leq n^{p-1}$, let 
\[
S_{n,i} \subset V_n^{\otimes p} \times V_n^{\oplus i p}
\]
be the constructible set defined by 
\[
S_{n,i} = \{ (\tb, \vb^{(1,1)},\ldots,\vb^{(i,p)} \;\mid\; \tb = \sum_{j=1}^{i} \vb^{(j,1)}\otimes\cdots\otimes \vb^{(j,p)} \}.
\]
Let
\[
S_n = \bigcup_{i=0}^{n^{p-1}} X_{n,i} \subset V_n^{\otimes p} \times \bigoplus_{i=0}^{n^{p-1}}V_n^{ \oplus i p}
\] 
where the constructible subsets $X_{n,i}$ are defined as follows.

Let 
\[
\Pi_{n,i}: V_n^{\otimes p} \times \bigoplus_{i=0}^{n^{p-1}}V_n^{ \oplus i p} \longrightarrow
V_n^{\otimes p} \times V_n^{ \oplus i p}
\]
be the obvious projections,
and let
\[
X_{n,i} = \Pi_{n,i}^{-1}(S_{n,i}).
\]
Now let 
\[
\pi_n: V_n^{\otimes p} \times \bigoplus_{i=0}^{n^{p-1}}V_n^{ \oplus i p} \rightarrow V_n^{\otimes p}
\] 
be the projection to the first factor.

Note that for each $\tb \in V_n^{\otimes p}$, $\pi_n^{-1}(\tb) \cap S_n$ is a constructible subset of $V_n^{\otimes p} \times \bigoplus_{i=0}^{n^{p-1}}V_n^{ \oplus i p}$,
and can be defined by a formula involving $n^p(n^{p-1}+1)$ polynomial equations, each of degree at most $p$, in 
$\dim(\bigoplus_{i=0}^{n^{p-1}}V_n^{ \oplus i p}) = np \binom{n^p+1}{2}$ variables.
Let 
\[
g_n:V_n^{\otimes p} \times \bigoplus_{i=0}^{n^{p-1}}V_n^{ \oplus i p} \rightarrow \C
\] 
be the constructible function defined by
\[
g_n = \sum_{i=0}^{n^{p-1}} M_n^{2i} \cdot\mathbf{1}_{X_{n,i}},
\]
where 
\begin{equation}
\label{eqn:M_n}
M_n = M(n^p(n^{p-1}+1), p,np\binom{n^{p-1}+1}{2})
\end{equation} 
(cf. Proposition \ref{prop:uniform-bound}).

Clearly, the sequence of functions 
\[
g_n: V_n^{\otimes p} \times \bigoplus_{i=0}^{n^{p-1}}V_n^{\oplus i p}  \rightarrow \C\hookrightarrow \C[T]
\] 
belong to $\mathbf{VP}_\C^\kA$ with $\kA = \C[T]$.

We now prove that the constructible function defined by 
\[
\bar{f_n}(\tb) = \int_{\bigoplus_{i=0}^{n^{p-1}}V_n^{\oplus i p}} g_n(\tb,\cdot) \dd\vP
\]
has the property that 
\[\Par(\bar{f_n}) \prec \Par(\rank_{n,p}).
\]

Let $\tb,\tb' \in V_n^{\otimes p}$ and suppose that
$\rank_{n,p}(\tb) = r < r' = \rank_{n,p}(\tb')$. Then, 
\begin{eqnarray*}
\pi_n^{-1}(\tb) \cap X_{n,i} &=& \emptyset, \mbox{ for $i < r$}, \\
\pi_n^{-1}(\tb') \cap X_{n,i} &=& \emptyset, \mbox{ for $i < r'$}, \\
\pi_n^{-1}(\tb)\cap X_{n,i} &\neq& \emptyset, \mbox{ for $i \geq  r$}, \\
\pi_n^{-1}(\tb') \cap X_{n,i} &\neq& \emptyset, \mbox{ for $i \geq r'$}.\\
 \end{eqnarray*}
 In particular, since $r< r'$ we have that,
 \begin{eqnarray}
 \label{eqn:nonzero}
 \pi_n^{-1}(\tb) \cap X_{n,i} &\neq&\emptyset, \mbox{ for $i = r$}, \\
 \label{eqn:zero}
 \pi_n^{-1}(\tb') \cap X_{n,i} &= & \emptyset, \mbox{ for $i = r$}. 
 \end{eqnarray}
 
 Since, $\pi_n^{-1}(\tb) \cap X_{n,i} \neq \emptyset$ by (\ref{eqn:nonzero}), it follows from Proposition \ref{prop:non-zero} that there exists
 $j\geq 0$ such that $\vP(\pi_n^{-1}(\tb) \cap X_{n,r})_j \neq 0$ (where for any polynomial 
 $\mathbf{P} \in \Z[T]$ we denote by $\mathbf{P}_j$ the coefficient of $T^j$ in $\mathbf{P}$). Using
 \eqref{eqn:zero} we know that $\vP(\pi_n^{-1}(\tb') \cap X_{n,r})_j = 0$. 
Moreover,  for  every $j \geq 0$,  
$\vP(\pi_n^{-1}(\tb) \cap X_{n,i})_j,\vP(\pi_n^{-1}(\tb') \cap X_{n,i})_j$ are both bounded in absolute value by $M_n$ (cf. \eqref{eqn:M_n} and Proposition \ref{prop:uniform-bound}).

Applying Lemma \ref{lem:injective} to the two sequences,
\[\left(\vP(\pi_n^{-1}(\tb) \cap X_{n,0})_j,\ldots,
\vP(\pi_n^{-1}(\tb) \cap X_{n,n^{p-1}})_j\right)\]
and 
\[
\left(\vP(\pi_n^{-1}(\tb') \cap X_{n,0})_j,\ldots,
\vP(\pi_n^{-1}(\tb') \cap X_{n,n^{p-1}})_j\right),
\]
along with the prescribed value of $M_n$ \eqref{eqn:M_n},
and using the definition of the function $\bar{f}_n$,
we deduce that,
$\bar{f}_n(\tb)_j \neq \bar{f}_n(\tb')_j$. This implies that
that $\Par(\bar{f}_n) \prec \Par(\rank_{n,p})$.

It follows that $(\rank_{n,p}: (\C^n)^{\otimes p}\rightarrow \C)_{n > 0}$ belongs to the class 
$\mathbf{VNP}_\C^\vP$.
\end{proof}
  
\subsection{Reduction and complete problems}
We adapt a standard notion of reduction (see for example \cite[Chapter 2]{Burgisser-book2}). 

Let $\kk$ be either the field  $\R$ or $\C$, and $\kA$ a polynomially bounded graded $\kk$-algebra.

\begin{definition}[Polynomial reduction]
\label{def:reduction}
For two sequences of constructible functions $(f_n: \kk^{m_1(n)} \rightarrow \kA)_{n >0},
(g_n: \kk^{m_2(n)} \rightarrow \kA)_{n >0}$, we will say that $(f_n)_{n >0} \leq_p (g_n)_{n > 0}$,
if there exists a polynomially bounded function $t: \mathbb{N} \rightarrow \mathbb{N}$, and
for each  $n>0$, a map $Z: [1,m_2(t(n))] \rightarrow \{X_1,\ldots,X_{m_1(n)}\} \cup \kk$,
such that 
the sequence of functions $(\bar{f}_n:\kk^{m_1(n)} \rightarrow \kA)_{n >0}$ defined by 
\[
\bar{f}_n(X_1,\ldots,X_{m_1(n)}) = g_{t(n)}(Z(1),\ldots,Z(m_2(t(n)))),
\] 
has the property that $\Par(\bar{f}_n) \prec \Par(f_n)$ for all $n>0$.
\end{definition}

The following two properties of the relation $\leq_p$ are obvious, and we omit their proofs.
\begin{lemma}[Composition]
\label{lem:substitution}
Let $m_1(n),m_2(n),m_3(n)$ be non-negative polynomials, 
$(g_n: \kk^{m_1(n)} \rightarrow \kA)_{n >0}$
a sequence of constructible functions, and 
\[
\left(H_n= (H_{n,1},\ldots,H_{n,m_2(n)}) \in \kk[X_1,\ldots,X_{m_1(n)}]^{m_2(n)}\right)_{n >0}
\]
a sequence of tuples of polynomials such that $\deg(H_{n,i}), \size(H_{n,i}) \leq m_3(n), 1\leq i \leq m_2(n)$, for $n>0$.
For each $n>0$, let $f_n = g_n\circ H_n$. Then,
$(f_n)_{n >0} \leq_p (g_n)_{n > 0}$. 
\end{lemma}

\begin{proof}
Omitted.
\end{proof}

\begin{lemma}[Transitivity]
\label{lem:transitivity}
Let $(f_n)_{n >0}, (g_n)_{n > 0}, (h_n)_{n >0}$ be sequences of constructible functions such that
$(f_n)_{n >0} \leq _p (g_n)_{n >0}$, and $(g_n)_{n > 0} \leq_p (h_n)_{n >0}$. Then,
$(f_n)_{n>0} \leq_p (h_n)_{n>0}$.
\end{lemma}

\begin{proof}
Omitted.
\end{proof}
\begin{definition}[Completeness]
\label{def:complete}
We say that a sequence of constructible functions $(g_n)_{n>0}$ is 
\emph{$\mathbf{VP}_\kk^\kA$-complete} if the sequence $(g_n)_{n>0}$ belongs to 
$\mathbf{VP}_\kk^\kA$, and for every sequence $(f_n)_{n > 0}$ in 
$\mathbf{VP}_\kk^\kA$, we have that $(f_n)_{n>0} \leq_p (g_n)_{n>0}$.

Similarly, we say that a sequence of constructible functions $(g_n)_{n>0}$ is 
\emph{$\mathbf{VNP}_\kk^\inv$-complete} if the sequence $(g_n)_{n>0}$ belongs to 
$\mathbf{VNP}_\kk^\inv$, and for every sequence $(f_n)_{n > 0}$ in 
$\mathbf{VNP}_\kk^\inv$, we have that $(f_n)_{n>0} \leq_p (g_n)_{n>0}$.
\end{definition}

\subsubsection{$\mathbf{VP}_\C^\C$-complete problems}
We do not have an example of a natural $\mathbf{VP}_\kk^\kA$-complete problem for $\kk = \R,\C$.
But for $\kk=\kA=\C$ we have the following conditional result.

The following hypothesis on the circuit complexity of the determinant polynomial is not known to be true
but plausible \cite[Section 2.5]{Burgisser-book2}. See \cite[Definition 2.4]{Burgisser-book2} 
for the precise definition of the class $\mathbf{VP}_\C$. 

\begin{hypothesis}
\label{hyp:detinVP}
The sequence of polynomials $(\det_n: \C^{n \times n} \rightarrow \C)_{n>0}$  is $\mathbf{VP}_\C$-complete.
\end{hypothesis}

 In other words, Hypothesis \ref{hyp:detinVP} asserts that 
for any sequence of polynomials $(P_n \in \C[X_1,\ldots,X_n])_{n >0} \in \mathbf{VP}_\C$, we have $(P_n)_{n>0} \leq_p (\det_n)_{n > 0}$
(see \cite[Definition 2.8]{Burgisser-book2} for a precise definition of $\leq_p$).
 
\begin{notation}
\label{not:list-of-det}
Given a set of $n$ matrices $A_1,\ldots,A_n \in \C^{n \times n}$, we denote by $\wdet_n$ the constructible function 
\[
\wdet_n: \C^{n\times n\times n} \rightarrow \C
\] 
defined by
\[
\wdet_n(A_1,\ldots,A_n) = \sum_{i=1}^{n} 2^i\cdot \mathbf{1}_{\det(A_i)=0}.
\]
\end{notation}

\begin{theorem}
Suppose that  \emph{\textbf{Hypothesis \ref{hyp:detinVP}}} holds.
Then, the sequence 
\[
\left( \wdet_n:\C^{n \times n \times n} \rightarrow \C \right)_{n >0}
\] 
is   $\mathbf{VP}_\C^\C$-complete.
\end{theorem}

\begin{proof}
We first prove (assuming Hypothesis \ref{hyp:detinVP} holds) that given any sequence of constructible functions $(f_n:\C^{m(n)} \rightarrow \C)_{n>0}$, there exists a polynomial
$m_1(n)$,
such that 
the sequence of functions $(g_n = \wdet_{m_1(n)}:\C^{m_1(n) \times m_1(n) \times m_1(n)} \rightarrow \C)_{n >0}$ 
has the property that $\Par(g_n) \prec \Par(f_n)$. This implies that $(f_n)_{n>0} \leq_p (g_n)_{n > 0}$ by Definition \ref{def:reduction}.

Since the sequence $(f_n)_{n>0}$ belongs to the class $\mathbf{VP}_\C^\C$, for each $n>0$ there exists a formula, $F_n$,
of size bounded polynomially in $n$, defining $f_n$. Thus, there exists a polynomial $m_2(n)$, such that at most $m_2(n)$ polynomials of degree at most $m_2(n)$ appear in $F_n$. Let the set of polynomials appearing in $F_n$ be denoted by $\mathcal{P}_n = (P_{n,i})_{i \in I_n}$,
where $I_n = [1,m_2(n)]$. By Hypothesis \ref{hyp:detinVP}, there exists a non-negative polynomial $m_1(n) \geq m_2(n)$, and for each $i \in I_{n}$, a square matrix, 
$A_{n,i}$  of size $m_1(n)$ whose entries belong to the set $\{X_1,\ldots,X_{m(n)}\} \cup \C$, and
such that $P_{n,i} = \det(A_{n,i})$. For $m_2(n) < i \leq m_1(n)$ let $A_{n,i} = \mathbf{Id}_{m_1(n)}$.
It is now clear from the definition of the matrices $A_{m,i}$ that  
$\wdet_{m_1(n)}(A_n)$ determines $\zero(\mathcal{P}_n(\x))$.
Defining $g_n = \wdet_{m_1(n)}:\C^{m_1(n) \times m_1(n) \times m_1(n) } \rightarrow \C$, it follows that
$f_n \leq_p g(n)$.

The fact that the sequence $(g_n)_{n > 0}$ belongs to the class $\mathbf{VP}_\C^\C$ follows from Theorem \ref{thm:rank}.
\end{proof}

\begin{remark}
By modifying the definition of the class $\mathbf{VP}_\C^\C$ (for example by defining the size of a polynomial to be the infimum of the sizes of all
formulas expressing the polynomial), or by generalizing the notion of reduction slightly by allowing the function $t(n)$ to be 
a quasi-polynomial, one could hope to prove an unconditional but weaker completeness result --  mirroring similar results in the
classical Valiant model (see for example  \cite[page 569, Corollary 21.40]{Burgisser-book1}). 
\end{remark}

The following problems are examples of  $\mathbf{VNP}_\R^\inv$-complete
and $\mathbf{VNP}^\inv_\C$-complete problems (for certain choices of the additive invariant
$\inv$).

\begin{notation}
\label{not:quadraticforms}
We identify quadratic polynomials in $n$ variables with coefficients in a field $\kk$ with the space of quadratic forms in $n+1$ 
variables, which are in turn represented by $(n+1) \times (n+1) $ symmetric matrices with entries in $\kk$. 
More precisely, given a polynomial $Q \in \kk[X_1,\ldots,X_n]$ with $\deg(Q) \leq 2$, we denote by $Q^h \in \kk[X_0,\ldots,X_n]$,
the homogenization of $Q$, which is a quadratic form in $n+1$ variables.  Note that, for $\x \in \kk^n$, 
$Q(\x) = Q^h(1,\x)$. We will denote the symmetric matrix $(n+1)\times (n+1)$ matrix representing $Q^h$ also by the letter $Q$.
\end{notation}

\begin{theorem}
\label{thm:VNPR-complete}
There exists a sequence $(\Omega_n \in \Z_+)_{n>0}$, such that
the sequence of functions $(h_n: \R^{(n+1) \times (n+1) \times n} \rightarrow \R)_{n > 0}$ defined
by 
\[h_n(Q_1,\ldots,Q_n) =  \int_{\R^{n}} g_n(Q_1,\ldots,Q_n,\cdot) \dd\chi,
\]
where
\[
g_n = 
\prod_{i=1}^n  \left(\Omega_n^{2^{3i}}\cdot \mathbf{1}_{(Q_i=0)} +  \Omega_n^{2^{3i+1}}\cdot\mathbf{1}_{(Q_i>0)}+ \Omega_n^{2^{3i+2}} \cdot\mathbf{1}_{(Q_i<0)}\right)
\]
is  $\mathbf{VNP}_\R^\chi$-complete.
\end{theorem}

\begin{proof}
Suppose that the $(f_n: \R^{m(n)}\rightarrow \R)_{n > 0}$ belongs to the class $\mathbf{VNP}_\R^\chi$. This implies (see Definition \ref{def:VNP}) that
there exists a a sequence of constructible functions
$(\bar{g}_n:\R^n \rightarrow \R)_{n > 0}$ belonging to the class $\mathbf{VP}_\R^\R$,
and a non-negative polynomial $m_1(n)$ such that
for each $n > 0$, $\Par(\bar{f}_n) \prec \Par(f_n)$, 
where $\bar{f}_n:\R^{m(n)} \rightarrow \R$ is the function defined by
\begin{equation*}
\bar{f}_n(\x) =  \int_{\R^{m_1(n)}} \bar{g}_{m(n) + m_1(n)}(\cdot,\x) \dd\chi.
\end{equation*}
Since the sequence $(\bar{g}_n:\R^n \rightarrow \R)_{n > 0}$ belongs to the class $\mathbf{VP}_\R^\chi$, there exists
for each $n>0$ a formula $G_n$ defining the constructible function $\bar{g}_n$, with $\size(G_n)$ bounded polynomially in $n$.

This implies that
for each $n>0$, there exists a polynomial $m_2(n)$, and a family $\mathcal{P}_n = (P_{n,i} \in \R[\X,\Y])_{i \in I_n}$,
where $\X= (X_1,\ldots,X_{m(n)}), \Y=(Y_1,\ldots,Y_{m_1(n)})$, satisfying:
\begin{enumerate}
\item
$\deg(P_{n,i}),\size(P_{n,i})$ are bounded by $m_2(n)$ for $i \in I_n$,
\item
$\card(I_n)$ is also bounded by $m_2(n)$, and 
\item
the set of polynomials appearing in $G_{m(n)+m_1(n)}$ is contained in $\mathcal{P}_n$.
\end{enumerate}
Now for each $i \in I_n$, there exists a straight line program
of length at most $m_2(n)$ computing $P_{i,n}$. The $j$-th line of such a straight line program is of the form,
\[
 Z_{i,j} = A \star B, \star \in \{+,\cdot\}
 \]
where $Z_{i,j}$ is a new variable, and each of $A$ and $B$ is either equal to an element of $\R$, an element of 
 $\{X_1,\ldots,X_{m(n)}, Y_1,\ldots,Y_{m_1(n)} \}$ or some $Z_{i,j'}$ for $j' < j$.  Denoting $Q_{i,j} = Z_{i,j} - A\star B$,
 we have $Q_{i,j} \in \R[\X,(Z_{i,j})_{i \in I_n, 1\leq j \leq m_2(n)},\Y]$, and  $\deg(Q_{i,j}) \leq 2$.  Now let 
 the total number of variables in the tuples of variables 
 $(Z_{i,j})_{i \in I_n, 1\leq j \leq m_2(n)}$ and $\Y$, be denoted by $m_3(n)$. Also, let
\begin{eqnarray*}
 \mathcal{Q}_{n,i}  &= &\bigcup_{1\leq j\leq m_2(n)} \{Q_{i,j}\}, \cr
 \mathcal{Q}_n  &=& \bigcup_{i \in I_n} \mathcal{Q}_{n,i}. 
\end{eqnarray*}

Let
\[
\mathcal{Q}_n = \left(Q_1,\ldots,Q_{m_3(n)}\right)
\] 
(augmenting the list by constant polynomials if needed).

First notice that for each $\x \in \R^{m(n)}$, and $\sigma \in \{0,1,-1\}^{\mathcal{P}_n(\cdot,\x)}$, there exists
$\sigma' \in \{0,1,-1\}^{\mathcal{Q}_n(\cdot,\x)}$, such that 
$\RR(\sigma,\R^{m_1(n)})$ is semi-algebraically homeomorphic to $\RR(\sigma',\R^{m_3(n)})$, and in particular
$\chi(\RR(\sigma,\R^{m_1(n)})) = \chi(\RR(\sigma',\R^{m_3(n)}))$.

It follows that for all $\x,\x' \in \R^{m(n)}$,
\[
\Big(\chi(\RR(\sigma,\R^{m_3(n)}))\Big)_{\sigma \in \{0,1,-1\}^{\mathcal{Q}_n(\cdot,\x)}} 
=
\Big(\chi(\RR(\sigma,\R^{m_3(n)}))\Big)_{\sigma \in \{0,1,-1\}^{\mathcal{Q}_n(\cdot,\x')}},
\] 
implies that 
\[
\Big(\chi(\RR(\sigma,\R^{m_1(n)}))\Big)_{\sigma \in \{0,1,-1\}^{\mathcal{P}_n(\cdot,\x)}} 
=
\Big(\chi(\RR(\sigma,\R^{m_1(n)}))\Big)_{\sigma \in \{0,1,-1\}^{\mathcal{P}_n(\cdot,\x')}}, 
\] 
which in turn implies that $\bar{f}_n(\x) = \bar{f}_n(\x')$. 

Now let $\Omega_n = (100 m_3(n))^{2m_3(n)}$.
It follows from 
 Proposition \ref{prop:uniform-bound-real} that for any $\x \in \R^{m(n)}$, and  $\sigma \in \{0,1,-1\}^{\mathcal{Q}_n(\cdot,\x)}$,
\begin{equation}
\label{eqn:Omega}
|\chi(\RR(\sigma,\R^{m_3(n)}))|^2 < \Omega_n.
\end{equation}

Define the constructible function $\bar{h}_n:\R^{m(n)} \rightarrow \R$ by
\begin{equation}
\label{eqn:barhn}
\bar{h}_n(\x) = h_{m_3(n)}(Q_1(\cdot,\x),\ldots,Q_{m_3(n)}(\cdot,\x)).
\end{equation}

Note that each $Q_i(\cdot,\x)$ is a polynomial in $m_3(n)$ variables of degree at most $2$, and is identified with a $(m_3(n)+1) \times  (m_3(n)+1)$ matrix (cf. Notation \ref{not:quadraticforms}) 
with entries which are polynomials in $\x$, each of degree at most $2$, and of size
bounded by a polynomial in $n$.

From Lemma \ref{lem:substitution} it follows that 
\begin{equation}
\label{eqn:hn}
(\bar{h}_n)_{n >0} \leq_p (h_n)_{n >0}.
\end{equation}

Using the definition of $h_n$ and \eqref{eqn:barhn}, we have
\begin{equation}
\label{eqn:barhn2}
\bar{h}_n(\x) = \int_{\R^{m_3(n)}} {g}_{m_3(n)}(Q_1(\cdot,\x),\ldots,Q_{m_3(n)}(\cdot,\x)) \dd\chi,
\end{equation}
and the integrand in \eqref{eqn:barhn2} equals  
\[
\prod_{i=1}^{m_3(n)}  \left(\Omega_n^{2^{3i}}\cdot \mathbf{1}_{(Q_i=0)} +  \Omega_n^{2^{3i+1}}\cdot\mathbf{1}_{(Q_i>0)}+ \Omega_n^{2^{3i+2}}\cdot \mathbf{1}_{(Q_i<0)}\right).
\]

It follows from Lemma \ref{lem:injective} and \eqref{eqn:Omega} that
the constructible function $\bar{h}_n$ has the property that  
two points $\x,\x' \in \R^{m(n)}$ belong to the same set of the partition
$\Par(\bar{h}_n)$,
if and only if,
\[
\Big(\chi(\RR(\sigma,\R^{m_3(n)}))\Big)_{\sigma \in \{0,1,-1\}^{\mathcal{Q}_n(\cdot,\x)}} 
=
\Big(\chi(\RR(\sigma,\R^{m_3(n)}))\Big)_{\sigma \in \{0,1,-1\}^{\mathcal{Q}_n(\cdot,\x')}}.
\]
It follows that $\Par(\bar{h}_n) \prec \Par(\bar{f}_n)$. Since $\Par(\bar{f}_n) \prec \Par(f_n)$, it implies that
$\Par(\bar{h}_n) \prec \Par(f_n)$
as well.
This proves that $(f_n)_{n <0} \leq_p (\bar{h}_n)_{n>0}$, 
and it follows from \eqref{eqn:hn}, and Lemma \ref{lem:transitivity} 
that $(f_n)_{n > 0} \leq_p ({h}_n)_{n>0}$.

The fact that the sequence $({h}_n)_{n >0}$ belongs to the
class $\mathbf{VNP}_\R^\chi$ is obvious from definition.
\end{proof}

A similar result holds over $\C$.
 \begin{theorem}
 \label{thm:VNPC-complete}
There exists a sequence $(\Omega_n \in \Z_+)_{n>0}$, such that
the sequence of functions $(h_n: \C^{(n+1) \times (n+1)\times n} \rightarrow \C[T])_{n > 0}$ defined
by

\[h_n(Q_1,\ldots,Q_n) =  \int_{\C^n}  g_n(Q_1,\ldots,Q_n,\x) \dd\vP,
\]
where
\[
g_n = 
\prod_{i=1}^n  \left(\Omega_n^{2^{2i}}\cdot \mathbf{1}_{(Q_i=0)} +  \Omega_n^{2^{2i+1}}\cdot\mathbf{1}_{(Q_i\neq0)}\right),
\]
is  $\mathbf{VNP}_\C^\vP$-complete.
\end{theorem}

\begin{proof}
The proof is very similar to the proof of Theorem \ref{thm:VNPR-complete} and is omitted.
\end{proof}

We can now formulate conjectures analogous to that of Valiant's.

For $\kk =\R$ (respectively, $\kk= \C$), and $K_0=K_0(\sa_\R)$ (respectively, $K_0 = K_0(\Var_\C)$),
and for for each additive invariant $\inv: K_0\rightarrow \kA$, where $\kA$ is a polynomially
bounded graded $\kk$-algebra, one can ask
\begin{question}
\label{ques:VP-BSS-inv}
Is $
\mathbf{VP}_\kk^\kA \neq
 \mathbf{VNP}_\kk^\inv 
$ ?
\end{question}

Note that  Question \ref{ques:VP-BSS-inv} is not interesting  if we choose $\inv$ to be the trivial additive invariant (i.e., $\inv = 0$). 
However, we make following  conjecture.
\begin{conjecture}[Analog of Valiant's conjecture]
\label{conj:VP-BSS}
Let $\kk =\R$ (respectively, $\kk= \C$), and $K_0=K_0(\sa_\R)$ (respectively, $K_0 = K_0(\Var_\C)$).
Then, for each \emph{non-trivial} additive invariant $\inv: K_0\rightarrow \kA$, where $\kA$ is a polynomially bounded graded $\kk$-algebra,
\[
\mathbf{VP}_\kk^\kA \neq  \mathbf{VNP}_\kk^\inv.
\]
\end{conjecture}

\section{Complexity theory of constructible sheaves}
\label{sec:sheaf-formulation}
\subsection{Motivation}
Constructible functions on semi-algebraic sets are intimately related to constructible sheaves. In fact, one way constructible functions on a semi-algebraic set $S$ appear is by taking the Euler-Poincar\'e characteristic of stalks of some constructible sheaf on $S$  (see Proposition \ref{prop:sheaves-to-functions} below). This already hints that the language of sheaves (or more accurately, constructible sheaves) might be useful in formulating certain questions in complexity theory in a more direct and geometric fashion. At the same time such an approach could generalize the existing questions to a more general, geometric setting.

Before delving into sheaf theory (or at least the fragment of it that will be relevant for us) let us consider an example that we formulate more precisely later. Let $\Phi(Y, X)$  be a first order formula in the language of the reals that defines a semi-algebraic subset $S \subset \R^m \times \R^n$. Let $\pi :\R^m \times \R^n \rightarrow \R^n$  be the projection to the second factor. Now consider the two semi- algebraic subsets $T, W \subset \R^n$  defined by the formulas $(\exists Y ) \Phi(Y, X)$  and $ (\forall Y )\Phi(Y, X)$. The basic problem of separating the complexity classes $\mathbf{P}_\R$ from the class $\mathbf{NP}_\R$ (respectively, $\textbf{co-NP}_\R$) is related to proving that while testing membership in S could be easy (in polynomial time), testing membership in $T$ (respectively, $W$) could be hard. The B-S-S polynomial time hierarchy, $\mathbf{PH}_\R$  (whose precise definition appears later in the paper, see Section \ref{sec:sheaf-formulation}) is built up by taking alternating blocks of existential and universal quantifiers, and it is conjectured that each new quantifier alternation allowed strictly increases the corresponding complexity class. The increase in complexity caused by taking image under a projection (of any “easy to compute'' semi-algebraic map) is at the heart of these complexity questions.

While the number of quantifier alternation is a well known measure of the complexity of a logical formula, quantifiers are not completely geometric in the following sense. Using the notation of the previous paragraph, for any fixed set 
$S \subset \R^m \times  \R^n$, the existential and universal quantifiers can only characterize points in $\x \in \R^n$, whose corresponding fibers $S_\x = \pi^{-1}(\x) \cap S$ are either empty or equal to $\R^m$. However, one might want to characterize the set of points $\x \in \R^n$, whose corresponding fiber $S_\x$ has a certain topological property – for example, non-vanishing homology groups of a certain dimension. More generally, it would be useful to include in the study of complexity of projections, the complexity of the coarsest possible partition of $\R^n$ such that the fibers $S_\x$  are locally “topologically” constant over each element of the partition. The word “topological” can be used to denote several notions -- the strictest notion being that of ``semi-algebraic homeomorphism''. In the formulation, that we define in this paper we will use the much weaker notion of “homological equivalence”.
Going back to the notions of $\mathbf{P}_\R$, $\mathbf{NP}_\R$ and $\textbf{co-NP}_\R$, the sheaf-theoretic interpretation is the following.
Suppose that  $S\subset \Sphere^m \times \Sphere^n$ (denoting the unit sphere in $\R^{n+1}$ by $\Sphere^n$) is a compact semi-algebraic set defined by a first order formula, $\Phi(\Y,\X)$,  in the theory of the reals. We can identify a given semi-algebraic set with a very special kind of \emph{constructible sheaf} -- namely, the constant sheaf   $\Q_S$  (see Definitions \ref{def:sheaf} and \ref{def:constant-sheaf}). 
The constructible sheaves $\Q_T$, $\Q_W$ corresponding to the compact sets $T$ and $W$ defined by $(\exists \Y )\Phi(\Y,\X)$ and 
$(\forall \Y )\Phi(\Y,\X)$ respectively, have sheaf-theoretic descriptions, namely:

\begin{eqnarray}
\label{eqn:crux}
\nonumber
\Q_T &\cong& \tau^{\geq 0}\tau^{\leq 0} R\pi_{*}\Q_{J_\pi(S)} ,\\
\Q_W &\cong& \tau^{\geq m}\tau^{\leq m} R\pi_{*} \Q_S[m]
\end{eqnarray}
(see Example \ref{eg:forall} below for a proof of the isomorphism in \eqref{eqn:crux}).

The definitions of the, \emph{fibered join} $J_\pi(S)$, the \emph{truncation} functors $\tau$, and the \emph{derived direct image functor} $R\pi_*$  in the \emph{derived category of constructible sheaves} 
is given later (see Section \ref{sec:main-definitions} below).
But  \eqref{eqn:crux} contains the crux of the idea behind defining sheaf-theoretic complexity classes. We first define a class of “simple” sheaves  (an analog of the set theoretic B-S-S class $\textbf{P}_\R$) which contain in particular the class of constant sheaves $\Q_S$ for all $S$ belonging to the class $\textbf{P}_\R$ (see Definition \ref{def:sheaf-P} below).
We show that this class of constructible sheaves is
stable under certain sheaf-theoretic operations mirroring the stability of the class $\mathbf{P}_\R$ under operations such as union, intersections, products and pull-backs.  We then show how to build (conjecturally) more  complicated sheaves using the direct image and truncation functors.  We build a hierarchy of sequences of sheaves, mirroring that of $\mathbf{PH}_\R$, where the place of a sequence in the hierarchy  depends on the number of direct image functors used in the definition of the sheaves in the sequence (in lieu of the number quantifier alternations in the set-theoretic case).

\subsection{Background on sheaves and sheaf cohomology} 
\label{subsec:background-on-sheaves}
We give here a brief redux of the definitions and main results from sheaf theory that will be necessary for the rest of the paper. We refer the reader to \cite{Godement} for more details concerning sheaves including the basic definitions, and to the books \cite{Iversen,Gelfand-Manin, Borel, KS} for details regarding derived categories and hypercohomology. In particular, the books \cite{KS,Dimca,Schurmann} are good
references for constructible sheaves in the context of the current paper.
 
Let $A$ be a fixed commutative ring. Later on for simplicity we will specialize to the case when $A=\Q$.
\begin{definition} (Pre-sheaf of $A$-modules)
\label{def:presheaf}
A \emph{pre-sheaf} $\mathcal{F}$ of $A$-modules over a topological space $X$ associates to each open subset 
$U\subset X$ an $A$-module $\mathcal{F}(U)$, such that that for all pairs of open subsets 
$U,V$ of $X$, with  $V \subset U$, there exists a \emph{restriction} homomorphism 
$r_{U,V}: \mathcal{F}(U) \rightarrow \mathcal{F}(V)$ satisfying:

\begin{enumerate}
\item
$r_{U,U} = \mathrm{Id}_{\mathcal{F}(U)}$,
\item
for $U,V,W$ open subsets of $X$, with $W\subset V \subset U$,
\[
r_{U,W} = r_{V,W} \circ r_{U,V}.
\]
\end{enumerate}
(For open subsets $U,V \subset X$, $V\subset U$, and $s \in \mathcal{F}(U)$, we will sometimes denote the
element $r_{U,V}(s) \in \mathcal{F}(V)$ simply by $s|_V$.)
\end{definition}

\begin{definition}(Sheaf of $A$-modules)
\label{def:sheaf}
A pre-sheaf $\mathcal{F}$ of $A$-modules on $X$ is said to be a \emph{sheaf} if it satisfies the following two axioms.
For any collection of open subsets $\{U_i\}_{i \in I}$ of $X$  with 
$
U = \bigcup_{i \in I} U_i ;
$
\begin{enumerate}
\item
if $s \in \mathcal{F}(U)$ and $s|_{U_i} = 0$ for all $i \in I$, then $s=0$;
\item
if for all $i\in I$ there exists $s_i \in \mathcal{F}(U_i)$ such that 
\[
s_i|_{U_i \cap U_j} = s_j|_{U_i \cap U_j}
\]
for all $i,j \in I$,
then there exists $s \in \mathcal{F}(U)$ such that $s|_{U_i} = s_i$ for each $i \in I$.
\end{enumerate}
\end{definition}

\begin{definition}[Stalk of a sheaf at a point]
\label{def:stalk}
Let $\mathcal{F}$ be a (pre)-sheaf of $A$-modules on $X$ and $\x\in X$. The \emph{stalk} 
$\mathcal{F}_\x$ of $\mathcal{F}$ at $\x$ is defined as the inductive limit 
\[
\mathcal{F}_\x = \varinjlim_{U \ni \x} \mathcal{F}(U).
\]
For any $U \ni \x$, there exists a canonical homomorphism $\mathcal{F}(U) \rightarrow \mathcal{F}_\x$, and the image of $s\in \mathcal{F}(U)$ under this homomorphism is 
denoted by $s_\x$. For $s_1,s_2 \in \mathcal{F}(U)$, we have $(s_1)_\x = (s_2)_\x$ if and only if there exists an open $V\subset U$ with $\x\in V$, such that $s_1|_V = s_2|_V$.
\end{definition}

\begin{definition}[Support of a section and of a sheaf]
\label{def:support}
Let $\mathcal{F}$ be a (pre)-sheaf of $A$-modules on $X$, and $U$ an open subset of $X$.
For $s \in \mathcal{F}(U)$, the support of $s$ (denoted $\supp(s)$) is the complement of the union of all subsets $V$ of
$U$, $V$ open in $U$, such that $s|_V = 0$. Clearly, $\supp(s)$ is a closed subset of $U$.

The complement of the union of open subsets $V$ of $X$ such that $\mathcal{F}|_V=0$ is called the \emph{support} of the sheaf $\mathcal{F}$,
and is denoted $\supp(\mathcal{F})$. Clearly, $\supp(\mathcal{F})$ is a closed subset of $X$.  
\end{definition}

There is a canonically defined sheaf associated to any pre-sheaf. This is important since certain
operations on a sheaf such as taking co-kernels of a sheaf morphism or the inverse image (see below)  only produces a pre-sheaf which then needs to sheafified if we are to stay within the category of sheaves.

\begin{definition}[Sheaf associated to a pre-sheaf]
\label{def:sheafification}
Let $\mathcal{F}$ be a pre-sheaf of $A$-modules over $X$. Then, the sheaf $\mathcal{F}'$ associated to $\mathcal{F}$ (the \emph{sheafification} of $\mathcal{F}$) is defined by associating to each open subset $U \subset X$, the $A$-module $\mathcal{F}'(U)$ consisting of all maps 
\[
s: U \rightarrow \bigcup_{\x\in U} \mathcal{F}_\x
\] 
satisfying the condition that for 
every $x \in U$, and any open neighborhood $U'$ of $\x$ in $U$, there exists 
$s' \in \mathcal{F}(U')$, such that for all $\x'\in U'$, $s(\x') = s'_{\x'}$. 
\end{definition}

\begin{notation}
\label{not:category}
We will denote by $\Sh(X,A)$ the category whose objects are sheaves of $A$-modules on $X$. 
\end{notation}

\begin{definition}[Global section functors]
We denote by 
by $\Gamma(X,\cdot)$ the functor (\emph{the global sections functor}) from
$\Sh(X,A)$ to the category $A\textbf{-mod}$ of $A$-modules, obtained by setting $\Gamma(X,\mathcal{F}) = \mathcal{F}(X)$.
We denote by $\Gamma_c(X,\cdot)$ the functor (\emph{global sections with compact support functor}) from
$\Sh(X,A)$ to the category $A\textbf{-mod}$ of $A$-modules, obtained by setting 
$\Gamma_c(X,\mathcal{F})$ to be the sub-module of $\Gamma(X,\mathcal{F})$ consisting of sections  $s \in \Gamma(X,\mathcal{F})$,
such that $\supp(s)$ (see Definition \ref{def:support}) is compact. Notice that, since $\supp(s)$ is always closed, 
if $X$ is compact then $\Gamma(X,\cdot) = \Gamma_c(X,\cdot)$.
\end{definition}

\begin{example}
The simplest sheaf of $A$-modules on a topological space $X$ is the constant sheaf (denoted $A_X$) defined 
as the sheaf associated to the presheaf  which associates the module $A$
to every open subset $U \subset X$. Each stalk
$(A_X)_\x$ of $A_X$ is then isomorphic to $A$.
\end{example}

\begin{definition}[Direct and inverse images]
\label{def:inverse-direct-image}
Let $X,Y$ be topological spaces and $f:X\rightarrow Y$ a continuous map. Let $\mathcal{F}$ (respectively, $\mathcal{G}$) be a sheaf of $A$-modules on $X$ (respectively, $Y$). 
\begin{itemize}
\item
The
association to each open set $U \subset Y$, the $A$-module $\mathcal{F}(f^{-1}(U)$ defines a sheaf of $A$-modules on $Y$, and this sheaf is denoted by $f_*\mathcal{F}$ and is called the \emph{direct image} of $\mathcal{F}$ under $f$.
\item
A sub-sheaf of $f_*\mathcal{F}$, namely the \emph{direct image under $f$ with proper support}, denoted $f_{!}\mathcal{F}$, will sometimes also be important for us. It is defined by associating to each open set $U \subset Y$, the sub-module of the $A$-module $\mathcal{F}(f^{-1}(U)$, consisting only of those elements $s \in \mathcal{F}(f^{-1}(U)$, such that $f|_{\supp(s)}: \supp(s) \rightarrow U$ is a proper map (i.e., inverse image of any compact set is compact). It is clear from the definition that $f_{!}\mathcal{F}$ is a sub-sheaf of $f_{*}\mathcal{F}$. 
\item
The association to each open set $U$ of $X$, the $A$-module 
\[
\varinjlim_{V\supset f(U)}\mathcal{G}(V),
\]
 defines a \emph{pre-sheaf}  on $X$. The sheaf associated to this pre-sheaf (Definition
\ref{def:sheafification}) is denoted by $f^{-1}\mathcal{G}$ and is called the \emph{inverse image}
of $\mathcal{G}$ under $f$.
\end{itemize}
\end{definition}

\begin{definition}[Restriction to a locally closed subset]
\label{def:restriction-to-locally-closed}
Suppose that $W$ is a locally closed semi-algebraic subset of $X$, and $\mathcal{F}$ a sheaf of $A$-modules on $X$. Let $j:W \hookrightarrow X$ be the inclusion map. We will denote by 
$\mathcal{F}|_W$ the sheaf $j^{-1}\mathcal{F}$ on $W$. 
\end{definition}

\begin{remark}
\label{rem:restriction-to-locally-closed}
It follows from the functoriality of the inverse
image functor that, using the same notation as above, if $W' \subset W$ is a locally closed semi-algebraic subset of $W$ (and hence $W'$ is also a
locally closed subset of $X$), then
\[
\mathcal{F}|_{W'} \cong \Big(\mathcal{F}|_{W}\Big)|_{W'}.
\]
Also, if $\x \in W$, then 
\begin{equation}
\label{eqn:stalk}
\mathcal{F}_\x \cong (\mathcal{F}|_W)_\x.
\end{equation}
\end{remark}

\begin{definition}[Morphisms of sheaves]
Let $\mathcal{F},\mathcal{G}$ be two sheaves on $X$. Then a morphism $\phi:\mathcal{F} \rightarrow \mathcal{G}$ is given by associating to each open subset $U$ of $X$, an element
$\phi(U) \in \Hom_A(\mathcal{F}(U),\mathcal{G}(U))$ such that for all pairs of open subsets
$U,V$ of $X$ with $V \subset U$ the following diagram commutes.
\[
\xymatrix{
\mathcal{F}(U) \ar[r]^{r_{U,V}} \ar[d]^{\phi(U)}& \mathcal{F}(V)
\ar[d]^{\phi(V)}
\\
\mathcal{G}(U) \ar[r]^{r_{U,V}} & \mathcal{G}(V)
}
\]
If $\phi: \mathcal{F}\rightarrow \mathcal{G}$ is a morphism, then for each $\x \in X$ there is
an induced homomorphism  $\phi_\x: \mathcal{F}_\x \rightarrow \mathcal{G}_\x$. 
\end{definition}

\begin{definition} [Kernel, co-kernel of sheaf morphisms]
If $\phi: \mathcal{F} \rightarrow \mathcal{G}$ is a morphism of sheaves on $X$, then the \emph{kernel} of $\phi$, denoted  $\ker(\phi)$, is the sheaf which associates to each open subset $U$ of $X$, the $A$-module
$\ker(\phi)(U) = \ker(\phi_U:\mathcal{F}(U)\rightarrow \mathcal{G}(U))$. In particular, note that
for each $x\in X$, the stalk $(\ker(\phi))_\x = \ker(\phi_\x: \mathcal{F}_\x \rightarrow \mathcal{G}_\x)$.

The \emph{co-kernel} of $\phi$, denoted $\coker(\phi)$ is the sheaf associated to the pre-sheaf (this pre-sheaf in general is not a sheaf)  which associates to each open $U\subset X$, the $A$-module $\mathcal{G}(U)/\textrm{Im}(\phi(U))$.
\end{definition}

\begin{definition}[Direct sum, tensor product and $\mathpzc{Hom}$ of two sheaves]
Let $\mathcal{F},\mathcal{G} \in \Ob(\Sh(X,A))$. Then, $\mathcal{F}\oplus\mathcal{G}$,(respectively, $\mathcal{F}\otimes\mathcal{G}$) is the sheaf obtained by associating to any open $U\subset X$, the $A$-module $\mathcal{F}(U)\oplus\mathcal{G}(U)$
(respectively, $\mathcal{F}(U)\otimes_A\mathcal{G}(U)$).

The sheaf $\mathpzc{Hom}(\mathcal{F},\mathcal{G})$) is the sheaf obtained by associating to any open $U\subset X$, the $A$-module $\Hom_A(\mathcal{F}|_U,\mathcal{G}|_U)$.
\end{definition}

\begin{warning}
\label{warning:stalk}
Given sheaves $\mathcal{F},\mathcal{G}$ on $X$, and $\x \in X$, it is true that 
$ (\mathcal{F}\oplus\mathcal{G})_\x \cong (\mathcal{F})_\x \oplus (\mathcal{G})_\x$, and
$ (\mathcal{F}\otimes\mathcal{G})_\x \cong (\mathcal{F})_\x \otimes (\mathcal{G})_\x$. But in general
it is not true that
$ (\mathpzc{Hom}(\mathcal{F},\mathcal{G}))_\x \cong \Hom_A(\mathcal{F}_\x ,\mathcal{G}_\x)$
(see \cite[2.2.5, also Ex. II.3]{KS}).
\end{warning}

\begin{definition}[Constant sheaves]
\label{def:constant-sheaf}
Let $M$ be a $A$-module.
The sheaf 
associated to the presheaf
on $X$ that associates to each open subset $U$ of $X$, the $A$-module
$M$ is denoted by $M_X$, and is called the \emph{constant sheaf on $X$ with stalk $M$} . 
\end{definition}

\begin{warning}
\label{warning:abuse}
More generally, 
if $Z$ is a closed subset of $X$, and $i:Z \hookrightarrow X$ the corresponding inclusion map, then we will abuse notation slightly and 
will denote by $M_Z$ the sheaf $i_*(M_Z)$ on $X$. 
\end{warning}

\begin{definition}[Locally constant sheaves]
\label{def:locally-constant}
A sheaf $\mathcal{F}$ on a semi-algebraic set $X$ is called \emph{locally constant} if there exists an open
cover of $X$ by open semi-algebraic sets $(U_i)_{i \in I}$ such that for each $i \in I$,
$\mathcal{F}|_{U_i}$ is a constant sheaf. 
\end{definition}

We will need the following proposition later in the paper.

\begin{proposition}
\label{prop:contractibility-implies-constant}
If $X$ is contractible then any locally constant sheaf on $X$ is isomorphic to a constant sheaf.
\end{proposition}

\begin{proof}
See \cite[Proposition 4.20]{Iversen} or \cite[page 132]{KS}.
\end{proof}

\begin{notation}
\label{not:Gamma}
For any locally closed subset $W \subset X$,
we denote by $\Gamma(W,\cdot)$ the functor from $\Sh(X,A)$ to $A\textbf{-mod}$ defined by $\mathcal{F}\mapsto \mathcal{F}|_W(W)$. This agrees with the prior definition of $\Gamma(U,\cdot)$
when $U$ is an open subset of $X$.
\end{notation}

\begin{remark}
\label{rem:abelian}
The category $\mathbf{Sh}(X,A)$ is an \emph{abelian category} \cite[Proposition 2.2.4]{KS}. 
Roughly, this means that every
morphism between $\phi:\mathcal{F}\rightarrow \mathcal{G}$ between sheaves admits a kernel and a co-kernel, and gives rise to the following diagram in which the induced morphism $u$ is an isomorphism (see \cite[Definition 1.2.1]{KS}).
\[
\xymatrix{
0 \ar[r] &\ker(\phi)  \ar[r]& \mathcal{F} \ar[r]^{\phi}\ar@{->>}[d] & \mathcal{G} \ar[r]& \coker(\phi)\ar[r]& 0 \\
&&\coimage(\phi)\ar[r]^u &\image(\phi)\ar@{>->}[u]&&
}
\]

In particular, this means we have the notion of exactness in the category $\Sh(X,A)$. Recall that a sequence of morphisms $\mathcal{F} \xrightarrow{f} \mathcal{G} \xrightarrow{g} \mathcal{H}$ is exact if $\ker(g) = \image(f)$.
\end{remark}

\begin{remark}
\label{rem:exactness}
Since the category $\mathbf{Sh}(X,A)$ is an abelian category, it makes sense to talk about \emph{exactness}
of the functors $f^{-1},f_*,f_{!},\Gamma(X,\cdot),\Gamma_c(X,\cdot)$
defined earlier. It turns out that $f^{-1}$ is an exact functor (i.e., it takes
an exact sequence of sheaves to an exact sequence), while the functors $f_*,f_{!},\Gamma(X,\cdot),\Gamma_c(X,\cdot)$
are only \emph{left exact}
(see \cite[page 15]{Godement} for definition of exactness as well as left exactness of functors).
\end{remark}

\subsubsection{Sheaf cohomology}
The functors $\Gamma(Z,\cdot)$
defined above are left-exact but not exact -- i.e.,  they do not preserve exactness when applied to an exact sequence of morphism on sheaves. Sheaf cohomology is a measure of this deviation from exactness, and is defined by applying one of these  (non-exact) functor to an \emph{injective resolution} (\cite{Godement} of a given sheaf (which always exists since the category $\Sh(X,A)$ has enough \emph{injective objects} \cite[II, 2.6]{KS}), and then taking the homology of the resulting (not necessarily exact) complex. 

We also obtain in this way a functor
$\HH^i(X,\cdot),\HH^i_c(X,\cdot)$ from the category of $\Sh(X,A)$ 
to the categories $A\textbf{-mod}$ as follows.  

\begin{definition}
\label{def:cohomology-compact-support}
$\HH^i(X,\cdot)= R^i\Gamma(X,\cdot)$ (respectively, $\HH^i_c(X,\cdot)$) is the $i$-th \emph{right derived functor} \cite{Godement} of  $\Gamma(X,\cdot)$ (respectively, $\Gamma_c(X,\cdot)$).
\end{definition}

It is a classical fact that \cite{Iversen}
\begin{proposition}
\label{prop:cohomology-compact-support}
Let $X$ be a locally closed semi-algebraic set.
Then, $\HH^i(X,\Q_X)$ (respectively, $\HH^i_c(X,\Q_X)$) is isomorphic to $i$-th singular cohomology group (respectively, 
cohomology group with compact support) with coefficients in $\Q$.
\end{proposition}

We record here also the following simple facts.
\begin{eqnarray}
\label{eqn:cohomology-of-sphere}
\nonumber
\HH^i(\Sphere^n, \Q) &\cong & \Q \mbox{ if $i=n,0$},\\
& \cong& 0 \mbox{ else.}
\end{eqnarray} 

Moreover, suppose that $S \subset \Sphere^n, S\neq \Sphere^n$ is a closed subset. Then, it follows
using an argument (that we omit) using the long exact sequence in
Theorem \ref{thm:inequality-compact-support}, and the Alexander duality theorem that
\begin{eqnarray}
\label{eqn:cohomology-of-sphere2}
\HH^n(S, \Q) &\cong& 0. 
\end{eqnarray}

Now suppose that $f : X \rightarrow Y$ is a continuous semi-algebraic map. Then, $f_*,f_{!}$  are functors from 
$\Sh(X,A)$  to $\Sh(Y,A)$ which are left exact, and carry injectives to injectives.  Thus one obtains for any 
$\mathcal{F} \in \Ob(\Sh(X,A))$ and $i \geq 0$, 
the \emph{higher direct image sheaves} 
$R^if_{*}(\mathcal{F}), R^if_{!}(\mathcal{F}) \in \Ob(\Sh(Y, A))$.

Similarly, for a sheaf $\mathcal{G} \in \Ob(\Sh(X,A))$, $\cdot \otimes_{A_X} \mathcal{G}$ is a right exact functor,
we have a (left) derived functor $\cdot \stackrel{L}{\otimes} \mathcal{G}$.
In case $A$ is a field (and hence all $A$-modules are free and so in particular flat), 
the functor $\otimes$ is exact and extends directly to the derived category (similar to $f^{-1}$).
Since, the functor  $\mathpzc{Hom}(\mathcal{G},\cdot)$ is left exact we have a right derived functor $\RHom(\mathcal{G},\cdot)$ in the derived category.

The following example is instructive.

\begin{example}
\label{eg:forall}
Let $S$ be a 
closed
semi-algebraic subset of $\Sphere^m \times \Sphere^n$ defined by a first order formula 
$\Phi(Y,X)$, where $\Y = (Y_0,\ldots,Y_m)$ and $\X = (X_0,\ldots, X_n)$. Let $T$ be the semi-algebraic subset of $\Sphere^n$ defined by the formula  $(\forall \Y \in \Sphere^m)\Phi(\Y, \X)$. Let $\pi : \Sphere^m \times \Sphere^n \rightarrow \Sphere^n$ denote the projection to the second factor, and for each $\x \in  \Sphere^n$, let $S_\x = \pi^{-1}(\x) \cap S$. 

We claim that 
\[
\Q_T \cong R^m \pi_* \Q_S.
\]

First notice that
$\x \in T$  if and only if $S_\x = \Sphere^n$,  if and only if  $\HH^m(S_\x,\Q) \cong \Q$
(the last equivalence by virtue of  \eqref{eqn:cohomology-of-sphere} and 
\eqref{eqn:cohomology-of-sphere2}).
In particular, this implies by Theorem \ref{thm:proper-base-change} that the stalks
of the sheaf $R^m\pi_* \Q_S$ vanish outside $T$, and
\begin{equation}
\label{eqn:proper}
\Q_T \cong R^m\pi_* \Q_{\pi^{-1}(T)}.
\end{equation}

Next consider the the following Cartesian square (recall Definition \ref{def:cartesian})
\[
\xymatrix{
{\pi^{-1}(T)} \ar[r]^{j} \ar[d]^{\pi}& S \ar[d]^{\pi}
\\
T \ar[r]^{i} & \pi(S),\\
}
\]
where $i:T \hookrightarrow \pi(S)$, $j: \pi^{-1}(T) \hookrightarrow S$ denote the inclusion maps.

By Theorem \ref{thm:proper-base-change}, we have that
\begin{eqnarray*}
i^{-1} R^m \pi_*\Q_S &\cong&  R^m\pi_*j^{-1}\Q_S  \mbox{ (Theorem \ref{thm:proper-base-change})} \\
&\cong& R^m \pi_*\Q_{\pi^{-1}(T)} \mbox{ (since $\pi^{-1}(T)$ is compact and $j$ is the inclusion map)} \\
&\cong & \Q_T \mbox{ (using \eqref{eqn:proper})}.
\end{eqnarray*}
Finally, since the stalks of the sheaf $i^{-1} R^m \pi_*\Q_S$ are zero outside $T$,
and letting $i':T \rightarrow \Sphere^n$ denote the inclusion map,
we have that $i'_* i^{-1} R^m \pi_*\Q_S \cong R^m \pi_*\Q_S$.
Hence,
$\Q_T \cong R^m \pi_* \Q_S$ (recall that we are denoting $i'_*\Q_T$ also by 
$\Q_T$ cf. Warning \ref{warning:abuse}).
Notice that the universal quantifier has been effectively 
replaced by the higher direct image functor $R^m\pi_{*}$.
\end{example}

It turns out that for our purposes it is better to enlarge the category $\Sh(X,A)$ by letting
for each open $U\subset X$, $\mathcal{F}(U)$ to be not just an $A$-module, but a  \emph{complex}
of $A$-modules.

Recall that:
\begin{definition}
\label{def:complex}
A complex $C^\bullet$ of $A$-modules is a sequence of homomorphisms 
$(C^i \xrightarrow{\phi^i} C^{i+1})_{i \in \Z}$ such that $\phi^{i+1}\circ\phi^i = 0$ for each $i$. 
Given a complex $C^\bullet$, the $i$-th cohomology module $\HH^i(\C^\bullet)$ is defined as the quotient $\ker(\phi^i)/\image(\phi^{i-1})$. 
\end{definition}

\begin{definition} (Truncation and shift operators). 
\label{def:truncation-et-shift}
Given a complex
\[
C^\bullet = \left(\cdots \xrightarrow{\delta^{i-1}} C^i \xrightarrow{\delta^i} C^{i+1}\xrightarrow{\delta{i+1}} \cdots\right)
\]
of $A$-modules, we define for any $n \in \Z$ the truncated complex $\tau^{\leq n}(C^\bullet)$  to
be the complex
\[
\tau^{\leq n}(C^\bullet) = \left( \cdots \xrightarrow{\delta^{n-3}} C^{n-2} \xrightarrow{\delta^{n-2}} C^{n-1}\xrightarrow{\delta^{n-1}} \ker(\delta^n) \rightarrow 0 \rightarrow 0\cdots \right).
\]

Similarly, we define the truncated complex $\tau^{\geq n}(C^\bullet)$  to be the complex
\[
\tau^{\geq n}(C^\bullet) = 
\left( \cdots 0 \rightarrow 0 \rightarrow \coker(\delta^n) \xrightarrow{\delta^n} C^{n+1} \xrightarrow{\delta^{n+1}} C^{n+2} \xrightarrow{\delta^{n+2}}  \cdots \right).
\]
 
We define the shifted complex
\[
C[n]^\bullet = \left(\cdots \xrightarrow{\delta[n]^{i-1}} C[n]^i \xrightarrow{\delta[n]^i} C[n]^{i+1}\xrightarrow{\delta[n]^{i+1}} \cdots\right),
\]
by setting $C[n]^i = C^{i+n}$ and $\delta[n]^i = \delta^{i+n}$.
\end{definition}

\begin{definition}(Tensor product and $\Hom$ of complexes).
Let $C^\bullet, D^\bullet$ be complexes of $A$-modules. Then, the tensor product
$C^\bullet \otimes D^\bullet$ is the complex defined by 
\[
(C^\bullet \otimes D^\bullet)^n = \bigoplus_{i+j=n} C^i \otimes D^j,
\]
with differentials defined in the obvious way.

The dual complex $\Hom^\bullet(C^\bullet,A)$ (denoting by $A$ the complex isomorphic to $A$ and concentrated at degree $0$ )
is defined by 
\[
\Hom^n(C^\bullet,A) = \Hom_A(C^{-n},A),
\]
with the obvious differentials.

More generally, $\Hom^\bullet(C^\bullet\otimes D^\bullet)$ is the complex defined by 
\[\Hom^n(C^\bullet,D^\bullet) = \prod_{j-i = n} \Hom(C^i,D^j).\]
Note that, using the isomorphism $\Hom(M,N) \cong \Hom(M,A) \otimes N$ for any two $A$-modules $M$ and $N$, we have the isomorphism,
\begin{equation*}
\Hom^n(C^\bullet,D^\bullet) = \prod_{j-i = n} \Hom(C^i,A) \otimes D^j,
\end{equation*}
and hence for bounded complexes we have the isomorphism,
\begin{equation}
\label{eqn:dual_hom}
\Hom^\bullet(C^\bullet,D^\bullet) \cong \Hom(C^\bullet,A)\otimes D^\bullet.
\end{equation}
\end{definition}

\begin{definition}
The category $\Kom(A\boldmod)$ is defined as the category whose objects are complexes of 
$A$-module and whose morphisms are morphisms of complexes. We say that a morphism
$f^\bullet: C^\bullet \rightarrow D^\bullet$ (i.e., $ f^\bullet \in \Hom_{\Kom(A\boldmod)}(C^\bullet,D^\bullet)$) is a \emph{quasi-isomorphism} if the induced homomorphism,
$f^i: \HH^i(C^\bullet) \rightarrow \HH^i(D^\bullet)$ is an isomorphism for all $i \in \Z$.
\end{definition}

 \begin{definition}
 Two morphisms of complexes $f^\bullet,g^\bullet: C^\bullet \rightarrow D^\bullet$ are
 said to be \emph{homotopically equivalent} ($f^\bullet \sim g^\bullet$)  if there exists a collection of homomorphisms 
 $h^i: C^i \rightarrow D^{i-1}, i\in \Z$, such that 
 \[
 f^i - g^i = h^{i+1} \circ \delta_{C^\bullet}^i + \delta_{D^\bullet}^{i-1}\circ h^i.
 \]
 The \emph{homotopy category of complexes}  $\K(A\boldmod)$ is the category whose objects are
 the same as the objects of $\Kom(A\boldmod)$, but whose morphisms are defined  by 
 \[
 \Hom_{\K(A\boldmod)}(C^\bullet,D^\bullet) = \Hom_{\Kom(A\boldmod)}(C^\bullet,D^\bullet)/\sim.
 \]
 \end{definition}
 
 Finally (we are being slightly imprecise here in the interest of space and readability; the reader should refer to \cite[Chapter III]{Gelfand-Manin} for a more precise definition):
 \begin{definition}
 The \emph{derived category} of $A$-modules, $\mathbf{D}(A\boldmod)$, is the category whose objects 
 are the same as those of $\K(A\boldmod)$, but whose set of morphisms 
 $\Hom_{\mathbf{D}(A\boldmod})(C^\bullet,D^\bullet)$ are equivalence classes of diagrams of the form
 $$
 \xymatrix{
 &  E^\bullet \ar[ld]^{qis} \ar[rd] & \\
 C^\bullet &&  D^\bullet
 }
 $$
 where the left arrow is an quasi-isomorphism.
 The derived category  $\mathbf{D}(A\boldmod)$ is no longer an abelian category, but is an example
 of what is called a \emph{triangulated category} \cite[Chapter IV]{Gelfand-Manin}.
 Finally, restricting to complexes with bounded cohomology (i.e.,  complexes $C^\bullet$ such that
 $\HH^i(\C^\bullet) = 0$ for $|i| \gg 0$)  we obtain the corresponding categories,
 $\Kom^b(A\boldmod), \K^b(A\boldmod), \mathbf{D}^b(A\boldmod)$.
 \end{definition}
 
 \begin{proposition}
 \label{prop:basic-isomorphism-in-derived}
 If $F^\bullet \in \mathbf{D}^b(A\boldmod)$, then we have the isomorphism (in the derived category) 
 \[
 F^\bullet \cong \bigoplus_{m \in \Z} \HH^m(F^\bullet)[-m] .
 \]
 (In other words, $F^\bullet$ is quasi-isomorphic to the direct sum of the shifted complexes 
 $
 \HH^m(F^\bullet)[-m],
 $ 
 where 
 $\HH^m(F^\bullet)[-m]$ is the complex concentrated at degree $m$, with all differentials equal to $0$, and whose degree $m$ part is isomorphic to $\HH^m(F^\bullet)$ (cf. Definition \ref{def:truncation-et-shift}).)
 \end{proposition}
 
 \begin{proof}
 See \cite[Ex. 1.18]{KS}.
 \end{proof}
 
\subsubsection{The categories $\Kom^b(X,A)$, $\K^b(X,A)$ and $\mathbf{D}^b(X,A)$}
Passing to sheaves over a semi-algebraic set $X$, the definitions of the last section gives 
successively the category
$\Kom^b(X,A)$ of sheaves of complexes of $A$-modules on $X$ whose stalks have bounded cohomology and which contains the category of $\Sh(X,A)$ as a sub-category of sheaves of complexes concentrated in degree $0$. Furthermore,  considering complexes only up to homotopy gives rise to the homotopy category $\K^b(X,A)$ of sheaves of complexes  on $X$ whose stalks have bounded cohomology, and
finally by a localization process, needed to formally invert arrows which are homotopy equivalences, we arrive at the derived category $\mathbf{D}^b(X,A)$ of sheaves on $X$.  This passage from
$\Sh(X,A)$ to $\mathbf{D}^b(X,A)$ while fairly standard, takes up a couple of chapters in textbooks
and we refer the reader to \cite{KS} for the details. While the definition of the derived category
might seem too cumbersome and unnecessary at first glance, once its existence is
taken for granted, it provides a very useful and concise geometric language to express relationships 
especially pertaining to cohomology of sheaves which is very useful in many applications.    

\begin{notation}
\label{not:abuse1}
Considering a sheaf $\mathcal{F} \in \Ob(\Sh(X,A))$, as a complex of sheaves concentrated at degree $0$,
we will continue to denote by $\mathcal{F}$ the corresponding object in the categories $\Kom^b(X,A)$, $\K^b(X,A)$ and $\mathbf{D}^b(X,A)$.
\end{notation}

\subsubsection{Extension of operations on sheaves to the derived category}
The sheaf operations of taking direct sums, tensor products, $Hom$, and direct and inverse
images under maps $f:X\rightarrow Y$ extend to the category $\mathbf{D}^b(X,A)$ in the same way that 
these operations extend from the category of $A$-modules to the category of complexes of 
$A$-modules. 

Since the functor
$\Gamma(Z,\cdot)$
is
left exact (see Remark \ref{rem:exactness}) 
 and take injective objects to injective objects, it  induces corresponding derived functors 
$R\Gamma(Z,\cdot)$
which take objects in $\mathbf{D}^b(X,A)$ to objects in 
$\mathbf{D}^b(A)$. 
Similarly, for any map $f:X\rightarrow Y$,
we have the derived functors $Rf_{!},Rf_*: \mathbf{D}^b(X,A) \rightarrow \mathbf{D}^b(Y,A)$. The functor
$f^{-1}$ being exact extends directly to a functor
$f^{-1}: \mathbf{D}^b(Y,A) \rightarrow \mathbf{D}^b(X,A)$ in the derived category.

\begin{definition}
\label{def:hypercohomology}
The images of the higher derived functors, $R^i\Gamma$  of the global section functor $\Gamma: \Sh(X,A) \rightarrow A\boldmod$, will be denoted by $\HH^i(X,\mathcal{F})$ and we call the $A$-modules
$\HH^i(X,\mathcal{F})$ the \emph{$i$-th cohomology module} of $\mathcal{F}$. 

Similarly, for $\mathcal{F}^\bullet \in \Ob(\mathbf{D}^b(X,A))$,
we call images of the higher derived functors, $R^i\Gamma$  of the global section functor $\Gamma: \mathbf{D}^b(X,A) \rightarrow \mathbf{D}^b(A\boldmod)$, will be denoted by $\hyperH^i(X,\mathcal{F}^\bullet)$,
and we call the $A$-modules
$\hyperH^i(X,\mathcal{F}^\bullet)$ the \emph{$i$-th hyper-cohomology module} of $\mathcal{F}^\bullet$. 
\end{definition}

 Given $\mathcal{F}^\bullet \in \Ob(\mathbf{D}^b(X,A))$, and $\x\in X$, the stalk $\mathcal{F}^\bullet_\x$ is represented
 by a complex of $A$-modules (i.e.,  an object of the category $\mathbf{D}^b(A)$). Thus, the cohomology 
 groups $\HH^i(\mathcal{F}^\bullet_\x)$ are $A$-modules and vanish for $|i|\gg 0$.
 Furthermore, recalling the definitions of the truncation and shift operations 
 on complexes of $A$-modules 
 (Definition \ref{def:truncation-et-shift}), these operations
 extend naturally to  $\mathcal{F}^\bullet \in \Ob(\mathbf{D}^b(X,A))$,
 and we obtain for each $n \in \Z$, 
$\tau^{\leq n}\mathcal{F}^\bullet, \tau^{\geq n}\mathcal{F}^\bullet,
\mathcal{F}^\bullet[n]$ respectively.
 It is an immediate consequence of the definitions that:

 \begin{proposition}
 Suppose $\mathcal{F}^\bullet \in \Ob(\mathbf{D}(X,A))$. Then, for $n \in \Z$ and $\x \in X$, we have:
 
\begin{eqnarray*}
 \HH^i((\tau^{\leq n}\mathcal{F}^\bullet)_\x) &\cong&  \HH^i(\mathcal{F}^\bullet_\x) \mbox{ for } i\leq n, \\
 \HH^i((\tau^{\leq n}\mathcal{F}^\bullet)_\x) &=&  0  \mbox{ for } i >  n, \\
\HH^i((\tau^{\geq n}\mathcal{F}^\bullet)_\x) &\cong&  \HH^i(\mathcal{F}^\bullet_\x) \mbox{ for } i\geq n, \\
\HH^i((\tau^{\geq n}\mathcal{F}^\bullet)_\x) &=&  0 \mbox{ for } i < n, \\
\HH^i((\mathcal{F}^\bullet[n])_\x) &\cong&  \HH^{i+n}(\mathcal{F}^\bullet_\x) \mbox{ for all } i\in \Z.
 \end{eqnarray*}
 \end{proposition}
 
 \begin{proof}
 Follows immediately from Definition \ref{def:truncation-et-shift}.
 \end{proof}
 
 There are two important isomorphisms in the derived category that will play a role later in the paper. We record them here for convenience.
 
 Let $f: X \rightarrow Y$ be a continuous semi-algebraic map, and suppose that $X,Y$ are compact. In particular, this means that $f$ is proper. Suppose also that,
$\mathcal{F}^\bullet \in \Ob(\mathbf{D}^b(X,A)), \mathcal{G}^\bullet \in \Ob(\mathbf{D}^b(Y,A))$, $\x\in X, \y = f(\x)$.
 
 \begin{proposition}
 \label{prop:adjunction}
 For every $i \in \Z$,
 \[ \HH^i(f^{-1}(\mathcal{G}^\bullet)_\x) 
\cong 
\HH^i(\mathcal{G}^\bullet_\y).
\]
\end{proposition}  
\begin{proof}
See \cite[II.4]{Iversen}.
\end{proof}

In order to state the next isomorphism we first recall the notion of a Cartesian square.
\begin{definition}
\label{def:cartesian}
For semi-algebraic sets $X, Y,X',Y'$, we say that the following diagram of semi-algebraic maps 
\[
\xymatrix{
X' \ar[r]^{g'} \ar[d]^{f'} & X \ar[d]^{f} 
\\
Y'\ar[r]^{g} & Y
}
\]
form a \emph{Cartesian square},  if it is isomorphic to the pull-back diagram of 
\[
\xymatrix{
& X \ar[d]^{f} 
\\
Y'\ar[r]^{g} & Y
}.
\]
In this case, $X'$ 
is semi-algebraically homeomorphic to 
the semi-algebraic set $\{(\x,\y') \in X \times Y' \;\mid\;  f(\x) = g(\y') \}$.
\end{definition}

\begin{theorem} [Proper base change theorem]
\label{thm:proper-base-change}
Given a Cartesian square as above, with all sets compact, we have for every $q$, and 
$\mathcal{F}^\bullet \in  \Ob(\mathbf{D}^b(X,A))$,
\[
g^{-1}R^q f_*\mathcal{F}^\bullet \cong R^q f'_* g'^{-1}\mathcal{F}^\bullet.
\]

In particular, if $Y' = \{\y\} \subset Y$ and $g$ is the inclusion map, then 
this implies that 
\[
(R^q f_*(\mathcal{F}^\bullet))_\y \cong \HH^q((Rf_*(\mathcal{F}^\bullet))_\y) \cong  \hyperH^q(f^{-1}(\y), \mathcal{F}^\bullet).
\]
\end{theorem}

\begin{proof}
See \cite[Theorem 2.3.26]{Dimca}.
\end{proof}

\begin{theorem}\cite[I.7, Theorem 7.15]{Iversen}
Let $F_1: \mathbf{A} \rightarrow \mathbf{B}$, $F_2: \mathbf{B} \rightarrow \mathbf{C}$ be additive functors
between abelian categories with enough injectives. If $F_1$ transforms injectives into $F_2$-acyclics and $F_2$
is left exact  then 
\[
R(U \circ T) \cong RU \circ RT.
\]
\end{theorem}

In fact a similar theorem holds in the derived category. We omit the somewhat technical hypothesis which generalizes
the hypothesis that ``$F_1$ transforms injectives into $F_2$-acyclics and $F_2$
is left exact'' from the following theorem, but this technical hypothesis will always hold
in the derived categories we will encounter in this paper.

\begin{theorem}\cite[III.7]{Gelfand-Manin}
\label{thm:derived-spectral}
Let $\mathbf{A},\mathbf{B},\mathbf{C}$ be abelian categories and
let $F_1: \Kom^b(\mathbf{A}) \rightarrow \Kom^b(\mathbf{B})$, $F_2: \Kom^b(\mathbf{B}) \rightarrow \Kom^b(\mathbf{C})$ be two exact  functors (satisfying certain technical conditions that we omit).
Then, for any complex $A^\bullet \in \D^b(\mathbf{A})$ there exists a spectral sequence with
\[
E_2^{p,q} \cong R^pF_2(R^q F_1(A^\bullet))
\]
which abuts to 
$R(F_2 \circ F_1) (A^\bullet) \cong RF_2 ( RF_1 (A^\bullet))$. 
\end{theorem}

The following corollary is then an immediate consequence of Theorem \ref{thm:derived-spectral}.
\begin{corollary}
\cite[Section 2.6, page 111]{KS}
\label{cor:composition}
If $f: X\rightarrow Y$, and $g:Y \rightarrow Z$ are continuous semi-algebraic maps between compact 
semi-algebraic sets, and $\mathcal{F}^\bullet \in \Ob(\mathbf{D}^b(X,A))$, then
\[
R(g \circ f)_* (\mathcal{F}^\bullet) \cong Rg_*(Rf_* (\mathcal{F}^\bullet)),
\]
and there exists a spectral sequence $E_r^{p,q}$, with 
\[
E_2^{p,q} \cong R^pg_*(R^qf_*(\mathcal{F}^\bullet)),
\]
converging to $R^n(g \circ f)_*(\mathcal{F}^\bullet)$. 
\end{corollary}

 \begin{proof}
 The functors $Rf_*$ and $Rg_*$ satisfy the hypothesis of Theorem \ref{thm:derived-spectral}, and thus
 proposition is an immediate corollary.
  \end{proof} 
  
  Another application of Theorem \ref{thm:derived-spectral} is the following.
\begin{proposition} [Leray spectral sequence]
  \label{prop:spectral}
  Let $A$ be a field and let $X$ be a locally closed semi-algebraic set, and $\mathcal{F}^\bullet \in \Ob(\mathbf{D}^b(X,A))$.
  Then, there is a spectral sequence abutting to $\hyperH^*(X,\mathcal{F}^\bullet)$ whose $E_2$-term
  is given by 
  \[
  E_2^{p,q} \cong \HH^p(X,\mathcal{H}^q(\mathcal{F}^\bullet)),
  \]
  where $\mathcal{H}^q(\mathcal{F}^\bullet) \in \Ob(\Sh(X,A))$ is the sheaf defined by:
  $\Gamma(U,\mathcal{H}^q(\mathcal{F}^\bullet)) = \hyperH^q(\mathcal{F}^\bullet(U))$ for every open subset $U$
  of $X$.
  In particular, if 
  $\mathcal{H}^q(\mathcal{F}^\bullet)$ is a constant sheaf for each $q$, then this spectral sequence degenerates at its $E_2$-term, and we have
  \[
  \hyperH^m(X,\mathcal{F}) \cong \bigoplus_{p+q=m} \HH^p(X,A_X)\otimes\HH^q(\mathcal{F}^\bullet_\x),
  \]
  for all $\x\in X$.
  \end{proposition}

  \begin{proof}
  In Theorem \ref{thm:derived-spectral} take $F_2$  to be the global section functor,
   $\Gamma:\Kom^b(X,A)\rightarrow \Kom^b(A\boldmod)$, and $F_1$ the identity functor.  
  The second part is a consequence of the universal coefficients theorem for homology.
  \end{proof}

   \begin{proposition} (Mayer-Vietoris spectral sequence)
  \label{prop:cech}
  Let $\K$ be a 
  compact
  semi-algebraic set and 
  $\mathcal{C} = \{C_i\}_{i \in I}$ a finite covering of
  $\K$ by semi-algebraic closed subsets. Let $\mathcal{F}^\bullet \in \Ob(\mathbf{D}^b(\K,A))$. Then, there
  is a spectral sequence abutting to $\hyperH^*(X,\mathcal{F}^\bullet)$ whose $E_2$-term is given by
  \[
  E_2^{p,q} \cong \bigoplus_{(i_0,\ldots,i_p)} \hyperH^q(C_{i_0,\ldots,i_p}, \mathcal{F}^\bullet|_{C_{i_0,\ldots,i_p}}),
  \]
  where $C_{i_0,\ldots,i_p} = C_{i_0} \cap \cdots \cap C_{i_p}$.
  \end{proposition}
  
   \begin{proof}
   Follows from \cite[Lemma 2.8.2]{KS} and standard arguments using spectral sequences
   arising from double complexes.
    \end{proof}

 \subsubsection{Constructible Sheaves}
 
\begin{definition}[Constructible Sheaves]
\label{def:constructible-sheaf}
Let $X$ be a locally closed semi-algebraic set. Following \cite{KS}, 
an object $\mathcal{F}^\bullet \in \Ob(\mathbf{D}^b(X,A))$ is
said to be \emph{constructible} if it satisfies the following two
conditions:
\begin{enumerate}
\item
There exists a finite partition $X = \cup_{i \in I} C_i$ of $X$ by
locally closed semi-algebraic subsets such that for $j \in \mathbb{Z}$ and $i \in I$,
the $\HH^j(\mathcal{F}^\bullet)|_{C_i}$ are locally constant. 
This means that for each $i \in I$, and $j \in \Z$, the sheaf on $C_i$ associated to the pre-sheaf defined by 
$U \mapsto  \HH^j(\mathcal{F}^\bullet|_U)$, is a locally constant sheaf (cf. Definition \ref{def:locally-constant}). 
This is equivalent to saying that for each $i \in I$, and
$\x \in C_i$, there exists an open neighborhood $U$ of $\x$ in $C$, such that for every $\x' \in U$, the 
restriction map $r$ induces isomorphisms 
$r_*: \HH^*(\mathcal{F}^\bullet(U)) \rightarrow \HH^*((\mathcal{F}^\bullet)_{\x'})$. We call such a partition to be \emph{subordinate} to $\mathcal{F}^\bullet$. 

For each $\x \in X$ we denote by $P_{\mathcal{F}^\bullet_\x} \in \Z[T]$ the
Poincar\'e polynomial of the bounded complex $\HH^*(\mathcal{F}^\bullet_\x)$  (which is in fact a Laurent polynomial in this case) defined by
\[
P_{\mathcal{F}^\bullet_\x}(T)  =  \sum_{i \in \Z} (\dim_A \HH^i (\mathcal{F}^\bullet_\x)) T^i.
\]
\item
For each $\x \in X$, the stalk $\mathcal{F}^\bullet_\x$ has the following properties:
\begin{enumerate}
\item
for each $j \in \mathbb{Z}$, 
the cohomology groups $\HH^j(\mathcal{F}^\bullet_\x)$ are finitely 
generated, and  
\item
there exists $N$ such that $\HH^j(\mathcal{F}^\bullet_\x) = 0$ for all $\x\in X$ and $|j| >N$.
\end{enumerate}

\end{enumerate}
We will denote the category of constructible sheaves on $X$ by $\constrD^b(X,A)$.
\end{definition}

From now on we will fix $A=\Q$ for convenience. Then all $A$-modules are projective, and 
in fact vector spaces over $\Q$. We will henceforth drop in all the notation the reference to the 
ring $A$, taking $A=\Q$.

\begin{example}
Let $X \subset \mathbb{R}^n$ be a  closed semi-algebraic subset. 
Let $i: X \hookrightarrow \R^n$ be the inclusion map.
Then the sheaf $Ri_{*}(\Q_X)$ is a constructible sheaf on $\R^n$. The stalks of $Ri_{*}(\Q_X)$ are given by:
\begin{eqnarray*}
(Ri_*(\Q_X))_\x &\cong& \Q, \mbox{ for } \x \in X, \\
&\cong& 0, \mbox { otherwise},
\end{eqnarray*}
where we denote by  $\Q$, the complex $C^\bullet$ defined by
$C^0 = \Q,$ and $C^j = \mathbf{0}$ for $j \neq 0$.
\end{example}

\subsubsection{Closure of the category of constructible sheaves under certain operations}
\begin{theorem}
\label{thm:closure1}
Let $X,Y$ be semi-algebraic sets, 
$f:X \rightarrow Y$ a semi-algebraic map, and
$\mathcal{F}^\bullet,  \in \Ob(\constrD^b(X)), \mathcal{G}^\bullet \in \Ob(\constrD^b(Y))$.
Then the following holds.
\begin{enumerate}
\item
$f^{-1}\mathcal{G}^\bullet \in \constrD^b(X)$.
\item 
Suppose that $f$ is proper restricted to $\supp(\mathcal{F}^\bullet)$. Then, 
$Rf_*(\mathcal{F}^\bullet), Rf_{!}(\mathcal{F}^\bullet) \in \Ob(\constrD^b(Y))$.
\end{enumerate}
\end{theorem}

\begin{proof}
See \cite[Propositions 8.4.8, 8.4.10]{KS}.
\end{proof}

\begin{theorem}
\label{thm:closure2}
Let $X$ be a semi-algebraic set, 
$\mathcal{F}^\bullet, \mathcal{G}^\bullet \in \Ob(\constrD^b(X))$.
Then,
\[\mathcal{F}^\bullet \stackrel{L}{\otimes} \mathcal{G}^\bullet, \RHom(\mathcal{F}^\bullet, \mathcal{G}^\bullet) 
\in \Ob(\constrD^b(X)).
\]
\end{theorem}

\begin{proof}
See \cite[Proposition 8.4.10]{KS}.
\end{proof}

The following is a key example.

\begin{example}
Let $S \subset \R^m \times \R^n$ be a compact semi-algebraic set and 
$\pi : \R^m \times \R^n \rightarrow \R^n$ the projection to the second factor. Clearly the
map $\pi$ is proper restricted to $\supp(\Q_S) = S$. Then, for each $\x \in \R^n$, we have, using 
Theorem \ref{thm:proper-base-change}  and  Proposition \ref{prop:basic-isomorphism-in-derived},
the following isomorphism (in the derived category):
\[
(R\pi_*\Q_S)_\x \cong  (R\pi_{!}\Q_S)_\x \cong \bigoplus_n\HH^n(S_\x,\Q)[-n], 
\]
where $S_\x = \pi^{-1}(\x) \cap S$. In other words, the stalks of $R\pi_*\Q_S$ are isomorphic to the cohomology groups (with coefficients in $\Q$) of the fiber of $S$ over $x$ under the map $\pi$. It follows from Hardt's  theorem on triviality of semi-algebraic maps (see \cite[Theorem 9.3.2.]{BCR}), that there is a finite semi-algebraic partition of $\R^n$ into connected, locally closed semi-algebraic sets, such that for each set $C$ of the partition, the homeomorphism type of the fibers $S_\x, \x \in C$ and hence the stalks $(R\pi_*\Q_S)_\x$ stays invariant. This is the sheaf-theoretic analog of \emph{quantifier elimination} in the
first order theory of the reals. Later in the paper we prove an effective version of this result (see Theorem \ref{thm:effective} below).
\end{example}

\subsection{The polynomial hierarchy in the B-S-S model over $\R$}
\label{subsec:recall-PH}
Our ultimate goal is to generalize the set theoretic complexity classes to a more general
sheaf-theoretic context. Before doing so we recall the definitions of the standard B-S-S complexity
classes. In fact, we will need to use the ``compact versions'' of these classes which were introduced
in \cite{BZ09} in order to avoid certain difficulties arising from non-proper projections.

We first recall the definition of the B-S-S polynomial hierarchy over the $\R$.
It  mirrors the discrete case very closely (see \cite{Stockmeyer}).

\begin{definition}[The class $\mathbf{P}_\R$]
Let $k(n)$ be any polynomial in $n$.
A sequence of semi-algebraic sets
$(T_n \subset \R^{k(n)})_{n > 0}$ 
is said to belong to the class $\mathbf{P}_\R$
if there exists a 
machine $M$ over $\R$ 
(see~\cite{BSS89} or~\cite[\S3.2]{BCSS98}),
such that for all $\x \in \R^{k(n)}$, the machine $M$ 
tests membership of $\x$ in $T_n$ 
in time bounded by a polynomial in $n$.
\end{definition}

\begin{definition}\label{df:sigma}
Let $k(n),k_1(n),\ldots,k_p(n)$ be polynomials in $n$.
A sequence of semi-algebraic sets
$(S_n \subset \R^{k(n)})_{n > 0}$ is said to be in 
the complexity class ${\bf \Sigma}_{\R,p}$, 
if for each $n > 0$ the semi-algebraic set
$S_n$ is described by a 
first order formula
\begin{equation}\label{eq:alternation}
(Q_1 \Y^{1} )  \cdots (Q_p\Y^{p} )
\phi_n(X_1,\ldots,X_{k(n)},\Y^1,\ldots,\Y^p),
\end{equation}
with $\phi_n$ a quantifier free formula in the first order theory of the reals,
and for each $i, 1 \leq i \leq p$,
$\Y^i = (Y^i_1,\ldots,Y^i_{k_i(n)})$ is a block of $k_i(n)$ variables,
$Q_i \in \{\exists,\forall\}$, with $Q_j \neq Q_{j+1}, 1 \leq j < p$,
$Q_1 = \exists$,
and 
the sequence of semi-algebraic sets  $(T_n \subset \R^{k(n)+ k_1(n) + \cdots + k_p(n)})_{n >0}$ 
defined by the quantifier-free formulas $(\phi_n)_{n>0}$ 
belongs to the class $\mathbf{P}_\R$.
\end{definition}

Similarly, the complexity class ${\bf \Pi}_{\R,p}$
is defined as in Definition~\ref{df:sigma}, with the exception 
that the  alternating quantifiers in~\eqref{eq:alternation} start with $Q_1=\forall$.
Since, adding an additional block of quantifiers on the outside
(with new variables) does not change the set defined by a quantified formula
we have the following inclusions:
$$ {\bf \Sigma}_{\R,p}\subset {\bf \Pi}_{\R,p+1}, 
\text{\ and\ }
{\bf \Pi}_{\R,p}\subset{\bf \Sigma}_{\R,p+1}.
$$

Note that by the above definition the class 
${\bf \Sigma}_{\R,0} = {\bf \Pi}_{\R,0}$ 
is the 
familiar class ${\bf P}_{\R}$, 
the class ${\bf \Sigma}_{\R,1} = {\bf NP}_{\R}$ and the
class ${\bf \Pi}_{\R,1} = \textbf{co-NP}_{\R}$.

\begin{definition}[Real polynomial hierarchy] 
The real polynomial time hierarchy is defined to be the union
\[
{\bf PH}_{\R} \defeq \bigcup_{p\geq 0} 
({\bf \Sigma}_{\R,p} \cup {\bf \Pi}_{\R,p}) = 
\bigcup_{p \geq 0} {\bf \Sigma}_{\R,p}  = 
\bigcup_{p \geq 0} {\bf \Pi}_{\R,p}.
\]
\end{definition}

As mentioned before, in order to get around certain difficulties caused by
non-locally-closed sets and  non-proper maps, a restricted polynomial hierarchy 
was defined in \cite{BZ09}. We now recall the definition 
of this compact analog, ${\bf PH}_{\R}^c$,  of
${\bf PH}_{\R}$. 
Unlike in the non-compact case, we will assume all variables
vary over certain compact semi-algebraic sets 
(namely spheres of varying dimensions).

\begin{definition}[Compact real polynomial hierarchy \cite{BZ09}]
\label{def:compactpolynomialhierarchy}
Let \[
k(n),k_1(n),\ldots,k_p(n)\]
be polynomials in $n$.
A sequence of semi-algebraic sets
$(S_n \subset \Sphere^{k(n)})_{n > 0}$ 
is in the complexity class ${\bf \Sigma}_{\R,p}^c$, 
if for each $n > 0$ the semi-algebraic set
$S_n$ is described by a first order formula
\[
 (Q_1 \Y^{1} \in \Sphere^{k_1(n)})  \cdots (Q_p \Y^{p} \in 
\Sphere^{k_p(n)} )
\phi_n(X_0,\ldots,X_{k(n)},\Y^1,\ldots,\Y^p),
\]
with $\phi_n$ a quantifier-free first order formula defining a 
{\em closed} semi-algebraic subset of 
$\Sphere^{k_1(n)}\times\cdots\times
\Sphere^{k_p(n)}\times \Sphere^{k(n)}$
and for each $i, 1 \leq i \leq p$,
$\Y^i = (Y^i_0,\ldots,Y^i_{k_i})$ is a block of $k_i(n)+1$ variables,
$Q_i \in \{\exists,\forall\}$, with $Q_j \neq Q_{j+1}, 1 \leq j < p$,
$Q_1 = \exists$, and
the sequence of semi-algebraic sets
$(T_n \subset \Sphere^{k_1(n)}\times\cdots\times 
\Sphere^{k_p(n)}\times\Sphere^{k(n)})_{n > 0}$
defined by the formulas $(\phi_n)_{n >0}$ belongs to the class
$\mathbf{P}_\R$.

\begin{example}
\label{eg:compact}
The following example that appears in \cite{BZ09} is an example of a language in $\mathbf{\Sigma}_{\R,1}^c$ (i.e., 
of the compact version of $\mathbf{NP}_\R$).

Let $k(n) = \binom{n+4}{4}-1$ and identify $\R^{k(n)+1}$ with the
space of \emph{homogeneous} polynomials in $\R[X_0,\ldots,X_n]$ of degree
$4$. Let $S_n \subset \Sphere^{k(n)} \subset \R^{k(n)+1} $ be defined by 
\[
S_n = \{P \in  \Sphere^{k(n)} \;\mid\; \exists \x 
=(x_0:\cdots:x_n) \in \mathbb{P}_\R^n \mbox{ with }
P(\x) = 0 \};
\]
in other words  $S_n$ is the set of (normalized) real forms of degree $4$ 
which have a zero in the real projective space $\mathbb{P}^n_\R$.
Then 
\[
(S_n \subset \Sphere^{k(n)})_{n > 0} \in \mathbf{\Sigma}_{\R,1}^c,
\]
since it is easy to see that $S_n$ also admits the description:
\[
S_n = \{P \in  \Sphere^{k(n)} \;\mid\; \exists \x \in \Sphere^n \mbox{ with }
P(\x) = 0 \}.
\]

Note that it is \emph{not known} if $(S_n \subset \Sphere^{k(n)})_{n > 0}$ 
is $\mathbf{NP}_\R$-complete 
(see Remark \ref{rem:compact}),
while the non-compact version of this language, 
i.e., the language consisting of (possibly non-homogeneous) 
polynomials of degree at most four having a 
zero in $\mathbb{A}_\R^n$ (instead of $\mathbb{P}^n_\R$),  
has been shown to be $\mathbf{NP}_\R$-complete \cite{BCSS98}.
\end{example}

We define analogously the class ${\bf \Pi}_{\R,p}^c$, and finally 
define the \emph{compact real polynomial time hierarchy} 
to be the union
\[
{\bf PH}_{\R}^c \defeq 
\bigcup_{p \geq 0} ({\bf \Sigma}_{\R,p}^c \cup 
{\bf \Pi}_{\R,p}^c) = \bigcup_{p \geq 0} {\bf \Sigma}_{\R,p}^c =
\bigcup_{p \geq 0} {\bf \Pi}_{\R,p}^c.
\]
\end{definition}

Notice that the semi-algebraic sets belonging to any language in 
${\bf PH}_{\R}^c$ are all semi-algebraic compact 
(in fact closed semi-algebraic  subsets of
spheres). Also, note the inclusion
\[
{\bf PH}_{\R}^c \subset {\bf PH}_{\R}.
\]

\begin{remark}
\label{rem:compact}
The restriction to compact sets in \cite{BZ09} was necessitated by the fact that certain
topological results used in \cite{BZ09} required certain maps to be proper, and assuming compactness
was an easy way to ensure properness of these maps. Similarly, in this paper it will be convenient
to assume that certain maps are proper restricted to supports of some given sheaves on
a semi-algebraic set $X$. Since the support of a sheaf is always closed, the properness
is ensured if we assume $X$ is compact. In the absence of the compactness assumption, one
would have to use the derived functors $Rf_{!}$ 
instead of $Rf_{*}$, 
and always consider cohomology groups with compact supports. While this might indeed be worthwhile to do in the future to have the fullest generality, we avoid complications in this paper by
making the compactness assumption.
 
However, note that even though the restriction to compact semi-algebraic sets might appear
to be only a technicality at first glance, 
this is actually an important restriction.
For instance, it is a long-standing  open question in real complexity theory 
whether there exists an  ${\bf NP}_{\R}$-complete 
problem which belongs to the class ${\bf \Sigma}_{\R,1}^c$ (the compact version
of the class ${\bf NP}_{\R}$, 
see Example \ref{eg:compact}).  See also \cite{BZ09} for natural examples of sequences
in the class ${\bf \Sigma}_{\R,1}^c$.
\end{remark}

\begin{remark}
The topological methods used in this paper only require the sets to be compact. Using spheres to achieve this compact situation is 
a natural choice in the context of real algebraic geometry, 
since the inclusion of the space $\R^n$ 
into its one-point compactification $\Sphere^n$ is a 
continuous semi-algebraic map that sends semi-algebraic subsets of $\R^n$ 
to their own one-point compactification (see \cite[Definition~2.5.11]{BCR}).
\end{remark}
  
\subsubsection{Stability of the classes $\mathbf{P}_\R$ and $\mathbf{P}_\R^c$ under certain operations}
It is important to note that the
the B-S-S complexity class $\mathbf{P}_\R$ (as well as $\mathbf{P}^c_\R$) is stable under certain  operations.
In fact, many results (such as the analog of Toda's theorem in the B-S-S model proved in \cite{BZ09} as an
illustrative example) depend only on these stability properties of the class $\mathbf{P}_\R$ and not
on its actual definition involving B-S-S machines. We will formulate similar stability properties for
the sheaf-theoretic generalization of the class $\mathbf{P}^c_\R$.

\begin{remark}
\label{rem:warning-about-spheres}
We note here that we will sometimes identify a compact subset of $S\subset \R^{n+1}$, with the corresponding
subset of the one-point compactification of  $\R^{n+1}$ which is homeomorphic to $\Sphere ^n$ and write
$S \subset \Sphere^n$. For example, we write $\Sphere^{m} \times \Sphere^{n} \subset \Sphere^{m+n+1}$,
the implied embedding is obtained by taking the product of the standard embeddings 
$\Sphere^m \hookrightarrow \R^{m+1}$,  $\Sphere^n \hookrightarrow \R^{n+1}$, and then taking the 
one-point compactification of $\R^{m+n+2}$. 
\end{remark} 

We omit the proofs of the following two propositions which follow immediately from the definition of the
classes $\mathbf{P}_\R$ and $\mathbf{P}_\R^c$.
 \begin{proposition}
 \label{prop:set-theoretic-stability}
 Let $m(n) \in \Z[n]$ be a fixed non-negative polynomial.
  Let $\left(X_n \subset \R^{m(n)}\right)_{n >0}$ and $\left(Y_n \subset \R^{m(n)}\right)_{n >0}$ both belong to $\mathbf{P}_\R$.
   Then, 
 $\left(X_n \cup Y_n\right)_{ n > 0}$, $\left(X_n \cap Y_n\right)_{n > 0}$,  
 $\left(X_n \times Y_n\right)_{n > 0}$, $\left(\R^{m(n)} \setminus X_n\right)_{n >0}$ all belong to the class $\mathbf{P}_\R$. Moreover, 
 if $\left(X_n \subset \Sphere^{m(n)} \right)_{n >0}$ and $\left(Y_n \subset \Sphere^{m(n)}\right)_{n >0}$ both belong to $\mathbf{P}^c_\R$, then
 \[
 \left(X_n \cup Y_n\right)_{n > 0},  \left(X_n \cap Y_n\right)_{n > 0},  \left(X_n \times Y_n\right)_{n > 0}
 \] 
 all belong to the class $\mathbf{P}^c_\R$ as well.
  \end{proposition}

  Even though the B-S-S complexity class $\mathbf{P}_\R$ is by definition a sequence of 
  semi-algebraic sets, sometimes it is also convenient to have a notion of
  a polynomial time computable maps. 
  We will use the following slight abuse of notation.
 
  \begin{notation}
  \label{not:abuse2}
   Let $m_1(n),m_2(n) \in \Z[n]$ two fixed non-negative polynomials and
  let $\left(f_n:\R^{m_1(n)} \rightarrow \R^{m_2(n)}\right)_{n>0}$ be a sequence of maps. We say that the sequence  $\left(f_n:\R^{m_1(n)} \rightarrow \R^{m_2(n)}\right)_{n>0} \in \mathbf{P}_\R$
  if   
  the maps $\left(f_n:\R^{m_1(n)} \rightarrow \R^{m_2(n)}\right)_{n>0}$ are computable by a B-S-S
  machine in polynomial time.
  \end{notation}
 
  \begin{proposition}
  \label{prop:set-theoretic-pull-back}
  Let $m_1,m_2 \in \Z[n]$ two fixed non-negative polynomials and
  let $\left(f_n:\R^{m_1(n)} \rightarrow \R^{m_2(n)}\right)_{n>0} \in \mathbf{P}_\R$. 
  \begin{enumerate}
  \item
 For any sequence $\left(X_n \subset \R^{m_2(n)}\right)_{n >0}$ belonging to $\mathbf{P}_\R$,
 the sequence 
 \[
 \left(f_n^{-1}(X_n)\subset \R^{m_1(n)}\right)_{n > 0}
 \]
  also belongs to the class $\mathbf{P}_\R$.
   \item
   For any two sequences  $\left(X_n \subset \R^{m_1(n)}\right)_{n >0}$ and $\left(Y_n \subset \R^{m_1(n)}\right)_{n >0}$ belonging to $\mathbf{P}_\R$, the sequence
   as well as $\left(X_{n}\times_{f_n} Y_{n} \subset \R^{m_1(n)} \times \R^{m_1(n)}\right)_{n > 0}$
   also belongs to the class $\mathbf{P}_\R$.
  \end{enumerate}
  Similar statements hold for $\mathbf{P}_\R^c$ as well.
 \end{proposition}  
\begin{proof}
Immediate.
\end{proof}

\begin{remark}
One important special case of Proposition \ref{prop:set-theoretic-pull-back} is when the semi-algebraic
maps $f_n$ are just projections forgetting the first $n$ coordinates.  
Notice also that Propositions \ref{prop:set-theoretic-stability} and \ref{prop:set-theoretic-pull-back} together
imply that the B-S-S classes $\mathbf{P}_\R$ and $\mathbf{P}_\R^c$ are  closed under usual set theoretic operations, as well
under taking inverse images under and fiber-products over polynomially computable semi-algebraic maps.
\end{remark}

While the classes $\mathbf{P}_\R$ and $\mathbf{P}_\R^c$ are  closed under inverse images under polynomially
computable maps, the question whether the same is true under taking direct images is equivalent to the famous $\mathbf{P}_\R$ vs. $\mathbf{NP}_\R$ (respectively, $\mathbf{P}_\R^c$ vs. $\mathbf{NP}_\R^c$) question, and the prevailing belief in fact is that this is not the case. 

More formally:
\begin{conjecture} \cite{BCSS98}
\[
\mathbf{P}_\R \neq \mathbf{NP}_\R.
\]
\end{conjecture}

We also make the compact version of the above conjecture.

\begin{conjecture}
 \[
 \mathbf{P}_\R^c \neq \mathbf{NP}_\R^c.
 \]
\end{conjecture}

\subsection{Sheaf-theoretic analog of $\mathbf{P}_\R$} 
\label{subsec:define-P-sheaves}
As noted previously (for reasons of expediency) we are going to restrict to compact complexity
classes from now on. For set theoretic classes this means we only consider sequences of compact
subsets of spheres, and in the sheaf-theoretic  generalizations we will only consider sequences
of sheaves supported on spheres. Note that in this case the supports of such sheaves are 
necessarily compact.

We now define the sheaf-theoretic analog of the complexity class $\mathbf{P}_\R^c$.

\begin{definition}(The  class $\bs{\mathcal{P}}_\R$)
\label{def:sheaf-P}
The class $\bs{\mathcal{P}}_\R$ of constructible sheaves consists of sequences 
$\left(\mathcal{F}^\bullet_n \in \Ob(\constrD^b(\Sphere^{m(n)})\right)_{n>0}$, where
$m(n) \in \Z[n]$ is a non-negative polynomial
satisfying the following conditions. 
There exists a non-negative polynomial $m_1(n) \in \Z[n]$ such that:
\begin{enumerate}
\item
For each $n>0$, there is an index set $I_n$ of cardinality $2^{m_1(n)}$, and a semi-algebraic partition,
$(S_{n,i})_{i \in I_n}$,  of $\Sphere^{m(n)}$ into locally closed semi-algebraic sets $S_{n,i}$ indexed by
$I_n$, which is subordinate to $\mathcal{F}_n$.

\item
For each $n>0$ and each $\x \in \Sphere^{m(n)}$,
\begin{enumerate}
\item The dimensions $\dim_\Q \HH^j((\mathcal{F}^\bullet_{n})_\x)$ are bounded by $2^{m_1(n)}$;
\item $\HH^j((\mathcal{F}^\bullet_{n})_\x)=0$ for all $j$ with $|j| > m_1(n)$.
\end{enumerate}
The two sequences of functions $(i_n: \Sphere^{(m(n)} \rightarrow I_n)_{n>0}$, and $(p_n: \Sphere^{m(n)} \rightarrow \Z[T,T^{-1}])$ defined by 
\begin{eqnarray*}
i_n(\x) & = & i \in I_n, \mbox{ such that, } \x \in S_{n,i} \\
p_n(\x) &= &  P_{(\mathcal{F}^\bullet_n)_\x}
\end{eqnarray*} 
are computable by B-S-S machines in time polynomial in $n$ (recall from 
Definition \ref{def:constructible-sheaf} that 
$P_{\mathcal{F}^\bullet_\x}$ denotes the Poincar\'e polynomial of the stalk of $\mathcal{F}^\bullet$ at $\x$).
(Notice that the number of bits needed to represent elements of $I_n$, and the coefficients of $P_{(\mathcal{F}^\bullet_n)_\x}$ are bounded polynomially in $n$.)
\end{enumerate}
\end{definition}

One immediate property of the class $\bs{\mathcal{P}}_\R$ is the following.
In order to state it we need a new notation.

\begin{notation}
\label{not:partition-set}
For any finite family $\mathcal{P} \subset \R[X_1,\ldots,X_n]$ and a semi-algebraic set $S \subset \R^n$, we will denote by $\Pi(\mathcal{P},X)$ the partition of $X$ into the connected components of $\RR(\sigma, S)$ for each realizable sign condition $\sigma \in \{0,1,-1\}^{\mathcal{P}}$ on $S$ (cf. Notation \ref{not:realization}).
\end{notation}

\begin{proposition}
\label{prop:singly-exponential}
Let $\left(\mathcal{F}^\bullet_n \in \Ob(\constrD^b(\Sphere^{m(n)}))\right)_{n>0}$ belong to the class
$\bs{\mathcal{P}}_\R$.
Then, there exists for each $n>0$, a family of polynomials $\mathcal{P}_n \subset \R[X_0,\ldots,X_{m(n)}]$, such that the semi-algebraic partition $\Pi(\mathcal{P}_n,\Sphere^{m(n)})$ is subordinate to
$\mathcal{F}^\bullet_n$, and
moreover 
$\card(\mathcal{P}_n)$, as well as the degrees of the polynomials in $\mathcal{P}_n$, are
bounded singly exponentially as a function of $n$.
\end{proposition}

\begin{proof}
This is an immediate consequence of the fact that the sequence of functions $i_n$ is computable in
polynomial time.
\end{proof}

One connection between $\bs{\mathcal{P}}_\R$ and the standard B-S-S complexity class $\mathbf{P}^c_\R$ is 
as follows.

\begin{proposition}
\label{prop:sheaf-P-vs.-set-P}
Let $\left(X_n \subset \Sphere^{m(n)}\right)_{n > 0}$ be a sequence of compact semi-algebraic sets.
Then,  $\left(X_n \in \Sphere^{m(n)}\right)_{n > 0} \in \mathbf{P}^c_\R$ if and only if the sequence of constructible sheaves 
$\Big(\Q_{X_n} \in \Ob(\constrD^b(\Sphere^{m(n)}))\Big)_{n > 0} \in \bs{\mathcal{P}}_\R$.
\end{proposition}

\begin{proof}
For any compact semi-algebraic set $X \subset \Sphere^{m(n)}$,  the stalks $(\Q_X)_\x = 0$
for $\x \not\in X$. 
For $\x\in X$, 
\begin{eqnarray*}
\HH^j((\Q_X)_\x) &=& 0  \mbox{ for } j \neq 0, \\
				   &=& \Q \mbox{ otherwise.}
\end{eqnarray*}
Now suppose that $\left(X_n \in \Sphere^{m(n)}\right)_{n > 0} \in \mathbf{P}^c_\R$.  Then, letting for each
$n>0$, $I_n = \{0,1\}$, and 
\begin{eqnarray*}
S_{n,0} &=& \Sphere^{m(n)} \setminus X_n,\\
S_{n,1} &=&  X_n,
\end{eqnarray*}
it is easy to verify that 
$\Big(\Q_{X_n} \in \Ob(\constrD^b(\Sphere^{m(n)}))\Big)_{n > 0} \in \bs{\mathcal{P}}_\R$.

Conversely, if $\Big(\Q_{X_n} \in \Ob(\constrD^b(\Sphere^{m(n)}))\Big)_{n > 0} \in \bs{\mathcal{P}}_\R$, then
for each $n>0$, there is an index set $I_n$ of cardinality $2^{m_1(n)}$, and a semi-algebraic partition,
$(S_{n,i})_{i \in I_n}$,  of $\Sphere^{m(n)}$ into locally closed semi-algebraic sets $S_{n,i}$ indexed by
$I_n$, 
which is subordinate to 
$\Q_{X_n}$
satisfying the properties listed in Definition \ref{def:sheaf-P}.
Now, from the fact that 
the sequence of functions $\Big(p_n: \Sphere^{m(n)} \rightarrow \Z[T,T^{-1}]\Big)_{n>0}$ defined by 
\begin{eqnarray*}
p_n(\x) &= &  P_{(\Q_{X_n})_\x} \\
 &=& 1 \mbox{ if $\x \in X_n$} \\
 &=&  0 \mbox{ else},
\end{eqnarray*} 
is computable by a B-S-S machine in time polynomial in $n$,
it follows immediately that 
$\left(X_n \in \Sphere^{m(n)}\right)_{n > 0} \in \mathbf{P}^c_\R$.
\end{proof}

The class $\bs{\mathcal{P}}_\R$ is stable under standard sheaf-theoretic operations. These closure
properties are reminiscent of the stability properties of the classes $\mathbf{P}_\R$ and  $\mathbf{P}^c_\R$ 
(cf. Propositions \ref{prop:set-theoretic-stability} and \ref{prop:set-theoretic-pull-back}).

\begin{proposition}[Stability properties of the class $\bs{\mathcal{P}}_\R$]
\label{prop:stability-sheaf-P}
Let $m(n) \in \Z[n]$ be a fixed non-negative polynomial.
\begin{enumerate}
\item (Closure under direct sums, tensor products.)
If 
\[
\Big(\mathcal{F}^\bullet_n \in \Ob(\constrD^b(\Sphere^{m(n)}))\Big)_{n>0}, 
\Big(\mathcal{G}^\bullet_n \in \Ob(\constrD^b(\Sphere^{m(n)}))\Big)_{n>0} \in  \bs{\mathcal{P}}_\R,
\] 
then $(\mathcal{F}^\bullet_n\oplus \mathcal{G}^\bullet_n)_{n>0},  (\mathcal{F}^\bullet_n\otimes \mathcal{G}^\bullet_n)_{n>0}
\in  \bs{\mathcal{P}}_\R$.
\item (Closure under pull-backs.) 
For any fixed non-negative polynomial $m_1(n) \in \Z[n]$, let
$\pi_n:\Sphere^{m_1(n)} \times \Sphere^{m(n)} \rightarrow \Sphere^{m(n)}$ be the projection to the second factor,
and 
\[
\Big(\mathcal{F}^\bullet_n \in \Ob(\constrD^b(\Sphere^{m(n)}))\Big)_{n>0} \in  \bs{\mathcal{P}}_\R.
\] 
Then, $\left(\pi_{n}^{-1}(\mathcal{F}^\bullet_n) \right)_{n>0} \in \bs{\mathcal{P}}_\R$.
\item (Closure under truncations.)
For any non-negative polynomial $m_1(n) \in \Z[n]$, and a sequence
\[
\Big(\mathcal{F}^\bullet_n \in \Ob(\constrD^b(\Sphere^{m(n)}))\Big)_{n>0} 
\]
belonging to the class  $\bs{\mathcal{P}}_\R$, we have 
 \[
 \left(\tau^{\leq m_1(n)} \mathcal{F}^\bullet_n\right)_{n>0} \in  \bs{\mathcal{P}}_\R,
 \]
 \[
 \left(\tau^{\geq m_1(n)} \mathcal{F}^\bullet_n\right)_{n>0} \in  \bs{\mathcal{P}}_\R.
 \]
 \item (Closure under shifts.)
For any non-negative polynomial $m_1(n) \in \Z[n]$, and a sequence
\[
\Big(\mathcal{F}^\bullet_n \in \Ob(\constrD^b(\Sphere^{m(n)}))\Big)_{n>0} 
\]
belonging to the class  $\bs{\mathcal{P}}_\R$, we have
\[
  \left(\mathcal{F}^\bullet_n[m_1(n)]\right)_{n>0} \in  \bs{\mathcal{P}}_\R,
\]
\[
 \left(\mathcal{F}^\bullet_n[-m_1(n)]\right)_{n>0} \in  \bs{\mathcal{P}}_\R.
 \]
 \end{enumerate}
\end{proposition}

\begin{proof}
\begin{enumerate}
\item
Let for each $n >0$, $(S'_{n,i})_{i' \in I'_n}$ and $(S''_{n,i})_{i'' \in I''_n}$ be semi-algebraic
partitions of $\Sphere^{m(n)}$ subordinate to the $\mathcal{F}^\bullet_n$ and $\mathcal{G}^\bullet_n$,
and $i_n',p_n',i_n'',p_n''$ the corresponding functions (cf. Definition \ref{def:sheaf-P}).
Also, let $m_1'(n)$  (respectively, $m_1''(n)$)  be the polynomial appearing in the definition of the
class $\bs{\mathcal{P}}_\R$ for the sequence $(\mathcal{F}^\bullet_n)_{n>0}$ (respectively,
$(\mathcal{G}^\bullet_n)_{n>0}$).
Let $I_n = I'_n \times I''_n$, and for each $i = (i',i'') \in I_n$, let $S_n = S'_{n,i'} \cap S''_{n,i''}$.
Then the following hold.
\begin{enumerate}
\item
For each $i=(i',i'')\in I_n$, $j\in\Z$, and $\x \in S_{n,i}$
\[
\HH^j((\mathcal{F}^\bullet_n\oplus\mathcal{G}^\bullet_n)_\x)  \cong \HH^j((\mathcal{F}^\bullet_n)_\x) \oplus 
\HH^j((\mathcal{G}^\bullet_n)_\x)
\] 
and is clearly locally constant for $\x\in S_{n,i}$,
since $\HH^*((\mathcal{F}^\bullet_n)_\x)$ is locally constant for $\x \in S'_{n,i'}$ and 
$\HH^*((\mathcal{G}^\bullet_n)_\x)$ is locally constant for $\x \in S''_{n,i''}$.
 
 \item 
\begin{eqnarray*}
 \dim_\Q \HH^j((\mathcal{F}^\bullet_{n}\oplus\mathcal{G}^\bullet_n)_\x)
 &=& 
 \dim_\Q \HH^j((\mathcal{F}^\bullet_{n})_\x)+\dim_\Q \HH^j((\mathcal{G}^\bullet_{n})_\x) \\
 &\leq & 
 2^{m_1'(n)} + 2^{m_1''(n)};
\end{eqnarray*}
Also using the Künneth formula  and the definition of tensor products of constructible
 sheaves,  
 we have
 \begin{eqnarray*}
\dim_\Q \HH^j((\mathcal{F}^\bullet_{n}\otimes\mathcal{G}^\bullet_n)_\x) 
 &=& 
 \dim_\Q \left(\bigoplus_{p+q=j} \HH^p((\mathcal{F}^\bullet_n)_\x) \otimes  \HH^q((\mathcal{G}^\bullet_n)_\x)\right)
 \\
 &=&
 \sum_{p+q=j} \dim_\Q \HH^p((\mathcal{F}^\bullet_{n})_\x)\cdot\dim_\Q \HH^q((\mathcal{G}^\bullet_{n})_\x) \\
 &\leq& \sum_{p+q=j} 2^{m_1'(n)+m_1''(n)}\\
 &\leq&  2(m_1(n) + m_2(n)+1) 2^{m_1'(n)+m_1''(n)},
 \end{eqnarray*}
noting that 
 \[
 \HH^j((\mathcal{F}^\bullet_{n})_\x), \HH^j((\mathcal{G}^\bullet_{n})_\x)=0\]
  for all $j$ with $|j| > m_1'(n)+m_1''(n)$.
  \item 
From the fact that the functions $i'_n: \Sphere^{m(n)} \rightarrow I_n'$ and
$i_n'': \Sphere^{m(n)} \rightarrow I_n''$ are computable in polynomial time it follows that
the function 
$i_n: \Sphere^{m(n)} \rightarrow I_n$, defined by $i_n(\x) = (i_n'(\x),i_n''(\x))$ is also
computable in polynomial time. 
\item 
From the fact that the functions $p'_n: \Sphere^{m(n)} \rightarrow \Z[T,T^{-1}]$ and
$p_n'': \Sphere^{m(n)} \rightarrow \Z[T,T^{-1}]$ are computable in polynomial time it follows that
the functions 
$p_n^\oplus: \Sphere^{m(n)} \rightarrow \Z[T,T^{-1}]$,  $p_n^\otimes: \Sphere^{m(n)} \rightarrow \Z[T,T^{-1}]$, defined by 
\begin{eqnarray*}
p_n^\oplus(\x)  &=& p_n'(\x) + p_n''(\x),\\
p_n^\otimes(\x) &=& p_n'(\x)\cdot p_n''(\x),
\end{eqnarray*}
are also computable in polynomial time. 
\end{enumerate}
It follows from the above that both sequences 
$(\mathcal{F}^\bullet_n\oplus \mathcal{G}^\bullet_n)_{n>0},  (\mathcal{F}^\bullet_n\otimes \mathcal{G}^\bullet_n)_{n>0}$
belong to the class $\bs{\mathcal{P}}_\R$.

\item
Let for each $n >0$, $(S'_{n,i})_{i\in I'_n}$  be the semi-algebraic
partition of $\Sphere^{m(n)}$ subordinate to $\mathcal{F}^\bullet_n$,
and $i_n',p_n'$ the corresponding functions (cf. Definition \ref{def:sheaf-P}).
Also, let $m_1'(n)$   be the polynomial appearing in the definition of the
class $\bs{\mathcal{P}}_\R$ for the sequence $(\mathcal{F}^\bullet_n)_{n>0}$.

Let $I_n = I_n'$.
Defining $i_n: \Sphere^{m_1(n)} \times \Sphere^{m(n)} \rightarrow I_n$ and 
$p_n: \Sphere^{m_1(n)} \times \Sphere^{m(n)} \rightarrow \Z[T,T^{-1}]$ by
\begin{eqnarray*}
i_n(\y,\x) &=& i_n'(\x), \\
p_n(\y,\x) &=& p_n'(\x),
\end{eqnarray*}
it follows from the fact that the sequences $(i_n')_{n>0}, (p_n')_{n>0}$ are computable in polynomial
time, that so are the sequences $(i_n)_{n>0}$ and $(p_n)_{n >0}$.

Moreover,  for each $(\y,\x) \in \Sphere^{m_1(n)} \times \Sphere^{m(n)}$ we have for each $j\in\Z$, using
Proposition \ref{prop:adjunction} an isomorphism 
\[
\HH^j(\pi_n^{-1}(\mathcal{F}^\bullet_n)_{(\y,\x)} \cong \HH^j((\mathcal{F}^\bullet_n)_\x.
\]
This shows that $\left(\pi_{n}^{-1}(\mathcal{F}^\bullet_n) \right)_{n>0} \in \bs{\mathcal{P}}_\R$.
\end{enumerate}
The remaining parts of the proposition are immediate.
\end{proof}

\begin{question}
\label{question:closure-under-hom}
By Theorem \ref{thm:closure2} we know that the category of constructible sheaves is closed under $\RHom$.
Is it true that
if 
\[\left(\mathcal{F}^\bullet_n \in \Ob(\constrD^b(\Sphere^{m(n)})\right)_{n>0}, 
\left(\mathcal{G}^\bullet_n \in \Ob(\constrD^b(\Sphere^{m(n)})\right)_{n>0} \in  \bs{\mathcal{P}}_\R,
\] 
then 
$(\RHom^\bullet(\mathcal{F}^\bullet_n, \mathcal{G}^\bullet_n))_{n >0}
\in  \bs{\mathcal{P}}_\R$ ?
One difficulty in proving this is related to the fact that the Poincar\'e polynomials of the
stalks $\RHom^\bullet(\mathcal{F}^\bullet_n,\mathcal{G}^\bullet_n)_\x$ cannot be
computed from the Poincar\'e polynomials of the stalks 
$(\mathcal{F}^\bullet_n)_\x$ and $(\mathcal{G}^\bullet_n)_\x$ as in the other cases
(see Warning \ref{warning:stalk}).
Note that sheaves of the form $\RHom(\mathcal{F},\mathcal{G})$ are quite
important. For example, they appear
as \emph{solution sheaves} in algebraic theory of partial differential equations \cite{KS}. 
It is also important in defining the Verdier duality that plays the role of Alexander-Poincar\'e duality in
sheaf theory \cite[VI]{Iversen}.
\end{question}

We now give several illustrative examples of sequences of constructible sheaves in the class
$\bs{\mathcal{P}}_\R$ other than those coming directly from a B-S-S complexity class
$\mathbf{P}^c_\R$.

\begin{example}(Rank stratification sheaf)
\label{eg:det}
For each $n >0$, let $V_n \subset   \Sphere^{n-1}\times \Sphere^{n^2-1}$ be the semi-algebraic 
set defined by
\[
V_n = \{(\x,A)  \mid \x \in \R^n, A \in \R^{n\times n}, A\cdot\x = 0,||A||^2=1,||\x||^2 =1 \}.
\]
Let $\pi_n:\Sphere^{n-1} \times \Sphere^{n^2-1} \rightarrow \Sphere^{n^2-1} $ denote the projection  
to the second factor.

\begin{proposition}
The sequence of constructible sheaves 
\[
\left( R\pi_{n,*}\Q_{V_n} \in \Ob(\constrD^b(\Sphere^{n^2-1}))\right)_{n >0}
\] 
belongs to the class $\bs{\mathcal{P}}_\R$.
 \end{proposition}
 \begin{proof}
 It is clear that the semi-algebraic partition by rank of the matrices $A \in \Sphere^{n^2-1}$ is subordinate
 to the constructible sheaf  $R\pi_{n,*}\Q_{V_n}$. Moreover, for each $\A \in \Sphere^{n^2-1}$, 
 we have that 
 \[
 \HH^*((R\pi_{n,*}\Q_{V_n})_A) \cong \HH^*(\Sphere^{n-1-\mathrm{rk}_{n,n} (A)},\Q).
 \]
 The claim is now clear from the fact that the rank of a matrix is computable in polynomial time by
 a B-S-S machine and Theorem \ref{thm:proper-base-change}.
 \end{proof}
 This example should be compared with Theorem \ref{thm:rank} in the previous Section \ref{sec:constructible-functions}.
\end{example}

The next example might look a little artificial at first glance but shows how index sets  with exponential cardinality, as well as Poincar\'e polynomials with coefficients which are exponentially large in $n$ could
arise for a sequence in $\bs{\mathcal{P}}_\R$.

\begin{example}
\label{eg:example0}
For $n>0$ let $P_n \in \R[X_0,X_1,\ldots,X_n,Y_0,\Y_1,\ldots,\Y_{n}]$,
where each $\Y_i = (Y_{i,1},\ldots,Y_{i,i}), 1\leq i\leq n$ is a block of $i$ variables, be defined by
\[
P_n = \sum_{i=1}^{n} \Bigg((n X_i^2-1)^2 + \Big(\sum_{j=1}^i 2n Y_{i,j}^2 +(1- \sqrt{n}X_i)\Big)^2\Bigg).
\]
Let
\[
N = \sum_{i=1}^n i = \binom{n+1}{2},
\]
and let 
$V_n = \ZZ(P_n,\Sphere^n \times \Sphere^{N}_{Y_0 \geq 0})$ 
(recall that $\Sphere^n$ and $\Sphere^{N}$ are the \emph{unit} spheres centered at the origin
in $\R^{n+1}$ and $\R^{N+1}$, respectively).
Let $\pi_n: \Sphere^n \times \Sphere^{N} \rightarrow \Sphere^n$ be the projection to the first
factor.
Observe that 
\[
\pi_n(V_n) =  \{(0,x_1,\ldots,x_n) \;\mid\; x_i \in \{-1,+1\}, 1\leq i \leq n\},
\] 
and for
$\x= (0,x_1,\ldots,x_n) \in \pi_n(V_n)$, $(V_{n})_\x = \pi_n^{-1}(\x)\cap V_n$ is described by
\[
(V_n)_\x = \Sphere^{N}_{Y_0 \geq 0} \cap (T_1 \times \cdots \times T_n),
\]
where for each $i$, $1 \leq i \leq n$, 
\begin{eqnarray*}
T_i & \cong& \Sphere^{i-1} \mbox{ if }  x_i = -1, \\
     &\cong&  \{\pt\} \mbox{ if }   x_i = 1.
 \end{eqnarray*}
 
\begin{proposition}
 The sequence $\left(R\pi_{n,*}\Q_{V_n} \in \Ob(\constrD^b(\Sphere^n)) \right)_{n >0}$ belongs to the class
 $\bs{\mathcal{P}}_\R$.
\end{proposition}
 
 \begin{proof}
 Let for $ I_n = \{+1,-1\}^n \cup \{0\}$, and for $\alpha=(\alpha_1,\ldots,\alpha_n)  \in \{-1,+1\}^n$
 \[
 S_{n,\alpha} = \{\x = (x_0,x_1,\ldots,x_n) \in \Sphere^n \;\mid\; x_i = \alpha_i, 1\leq i \leq  n\},
 \]
 and let 
 \[
 S_0 = \Sphere^n \setminus \bigcup_{\alpha_ \in \{-1,+1\}^n} S_{n,\alpha}.
 \]
 
 Then the family $\left(S_{n,\alpha}\right)_{\alpha\in I_n}$ is a semi-algebraic partition of $\Sphere^n$
 into $\card(I_n) = 2^n+1$ locally closed semi-algebraic sets. It is easy to check that 
$\HH^*((R\pi_{n,*}\Q_{V_n})_\x)$ is locally constant over each element of the partition, and 

\begin{eqnarray*}
 \HH^*((R\pi_{n,*}\Q_{V_n})_\x) &=& 0 \mbox{ if } \x\in S_{n,0}, \\
                                                &\cong& \HH^*(T_1\times\cdots\times T_n,\Q) \mbox{ for } \x \in S_{n,\alpha} \mbox{ with } \alpha \in \{-1,+1\}^n,
\end{eqnarray*}
where for each $i$, $1 \leq i \leq n$, 
\begin{eqnarray*}
T_i & \cong& \Sphere^{i-1} \mbox{ if }  \alpha_i = -1, \\
     &\cong&  \{\pt\} \mbox{ if }   \alpha_i = 1.
 \end{eqnarray*}
 
Using the Künneth formula we have that 
\begin{eqnarray}
\label{eqn:example}
P_{T_1 \times \cdots \times T_n}(T) &=&  \prod_{i=1}^n P_{T_i} \nonumber\\
&=& \prod_{i=1}^n  \left(\frac{(1 - \alpha_i)}{2} (1 + T^{i-1})\right)
\end{eqnarray}
(recall from Notation \ref{not:betti} that for any locally closed semi-algebraic set $X$ we denote by
$P_X(T)$ the Poincar\'e polynomial of $X$).

It follows from \eqref{eqn:example} that the dimensions 
$\dim_\Q \HH^*((R\pi_{n,*}\Q_{V_n})_\x)$ are bounded singly exponentially in $n$. 

It is also clear that the sequence of maps
$(i_n:\Sphere^n \rightarrow I_n)_{n>0}$ and $(p_n:\Sphere^n \rightarrow \Z[T,T^{-1}])$ are computable
in polynomial time. 
Together with the above this shows that $\Big(R\pi_{n,*}\Q_{V_n} \in \Ob(\constrD^b(\Sphere^n)) \Big)_{n >0}$ belongs to the class $\bs{\mathcal{P}}_\R$.
\end{proof}

\begin{remark}
Notice that in the above example, 
it follows from the unique factorization property of $\Z[T]$, that 
$\HH^*((R\pi_{n,*}\Q_{V_n})_\x)$ is distinct over each element, $S_{n,\alpha}$, of the partition
$(S_{n,\alpha})_{\alpha \in I_n}$, and this is the \emph{coarsest possible}  partition 
which is subordinate to $R\pi_{n,*}\Q_{V_n}$,
and this
partition has size $2^n+1$ which is singly exponential in $n$.  Moreover,
$\displaystyle{\sum_{i} \dim_\Q \HH^i((R\pi_{n,*}\Q_{V_n})_\x)}$ can be as large as $2^n$, and
in fact $\displaystyle{\sum_{i} \dim_\Q \HH^i((R\pi_{n,*}\Q_{V_n})_\x) = 2^n}$, for $\x = (0,-1,-1,\ldots,-1)$
(seen by substituting $T=1$ in the corresponding Poincar\'e polynomial). 
\end{remark}
\end{example}

\begin{example}
\label{eg:example1}
Let $d,t > 0$ be fixed integers. Let for each $n > 0$, $P_n \in \R[X_1,\ldots,X_n]$ be a polynomial with $\deg(P_n) \leq d$, and such that the sequence of functions 
\[
\left(P_n:\R^n \rightarrow \R\right)_{n>0} \in \mathbf{P}_\R
\] 
(see Notation \ref{not:abuse2}). For example, we could take for $P_n$ the elementary symmetric polynomial $e_{d,n}$ of degree $d$ in $n$ variables. Let 
$V_n = \ZZ(P_n,\Sphere^{n-1})\subset \R^n$, and it is immediate that the sequence 
$\left(V_n \subset \Sphere^{n-1}\right)_{n>0} \in \mathbf{P}^c_\R$. Now for each $n>t$, $\pi_{n,t}$ denote the projection  $
\R^{n+t}$ to $\Sphere^n$ which forgets that last $t$ variables,
and where we have replaced $\R^n$ by its one-point compactification (refer to Remark \ref{rem:warning-about-spheres}).

\begin{proposition}
\label{prop:example1}
The sequence $ \left(R\pi_{n,t,*} \Q_{V_{n+t}}\right)_{n > 0}$ belongs to the class
$\bs{\mathcal{P}}_\R$. 
\end{proposition}

\begin{proof}
First  observe that the number of coefficients of the polynomial $P_{n+t}$ is bounded polynomially
in $n$ (for constant $t$ and $d$), and moreover from the fact that the sequence of functions 
$\left(P_n:\R^n \rightarrow \R\right)_{n>0} \in \mathbf{P}_\R$ it follows that these coefficients can be 
computed in polynomially many steps by a B-S-S machine.
Notice that for each $x \in \R^n$, the fiber $(V_{n+t})_\x = V_{n+t} \cap \pi_{n,t}^{-1}(x)$ can be identified with the set of zeros of the polynomials $P_{n+t}(\x,Y_1,\ldots,Y_t)$ intersected with the sphere having the equation $Y_1^2+\cdots+Y_t^2 + |\x|^2 -1$. Moreover the degree of the polynomial $P_{n+t}(\x,Y_1,\ldots,Y_t)$ in
$Y_1,\ldots,Y_t$ is at most $d$. Now, it is a standard fact (see for example \cite{BPRbook2}) that there exists
a family of polynomials $\mathcal{Q}_n\subset \R[X_1,\ldots,X_n]$, such that 
$\card(\mathcal{Q}_n)$, as well as the degrees of the polynomials in $\mathcal{Q}_n$, are bounded by
$d^{2^{O(t)}} = O(1)$, and such that for each $\x \in \R^n$ the signs of the polynomials in $\mathcal{Q}_n$ at
$x$ determines the topological type of the fiber $(V_{n+t})_\x$. The number of  of realizable sign conditions of this family $\mathcal{Q}_n$ is bounded singly exponentially in $n$. Moreover,
for each $i \in \Z$ and $\x\in \R^n$,
\[
\HH^i((R\pi_{n,t,*} \Q_{V_{n+t}})_\x) \cong \HH^i((V_{n+t})_\x,\Q).
\]

From dimension considerations it follows that $\HH^i((V_{n+t})_\x,\Q) = 0$ for $i\not\in [0,t]$. Moreover, it
follows from standard bounds on the Betti numbers of real varieties \cite{OP,T,Milnor2} that, 
$\sum_i \dim_\Q (\HH^i((V_{n+t})_\x,\Q)) =  O(d)^t$. 
It is also clear that given $\x\in\R^n$, the signs of the polynomials in $\mathcal{Q}_n$ at $\x$
as well as the dimensions, $\dim_\Q \HH^i((V_{n+t})_\x,\Q)$,  
can be computed in polynomial time. This completes the proof that 
\[ 
\left(R\pi_{n,t,*} \Q_{V_{n+t}}\right)_{n > 0} \in \bs{\mathcal{P}}_\R.
\] 
\end{proof}
\end{example}

In order to describe the next example we need a new notation.
\begin{notation}
We denote by $\Sym_{n,d}(\R)$ the $\R$-vector space  of forms over $\R$ of degree $d$ in $n+1$ variables. We
will denote by $N_{n,d} = \binom{n+d}{d} = \dim_\R \Sym_{n,d}$. Note that each form $P \in \Sym_{n,d}$ can
be identified uniquely with a point on the sphere $a_P \in \Sphere^{N_{n,d}-1}$ obtained by intersecting the sphere  $\Sphere^{N_{n,d}-1}$ with the half-ray consisting of the positive multiples of $P$. We will
identify $P$ with the point $a_P$ in what follows.
\end{notation}

\begin{proposition}
\label{prop:example2}
Let $s >0$ be fixed, and consider for each $n> 0$, the compact real algebraic set 
$V_n \subset (\Sphere^{N_{n,2}-1})^s \times \Sphere^n$ defined by 
\[
V_n = \{ (P_1,\ldots,P_s,\x)\; \mid \;  \x \in \Sphere^n, P_i \in \Sym_{n,2}, P_i(\x) = 0, 1\leq i \leq s\},
\]

Let $\pi_{n}: (\Sphere^{N_{n,2}-1})^s \times \Sphere^n \rightarrow (\Sphere^{N_{n,2}-1})^s\hookrightarrow 
\Sphere^{s(N_{n,2}-1) + s}$ be the projection to the first factor.
Then,
\begin{enumerate}
\item
The sequence $(V_n)_{n>0} \in \bs{\mathbf{P}}_\R$.
\item
The sequence $\left(R\pi_{n,*}(\Q_{V_n}) \in \Ob(\constrD^b(\Sphere^{s(N_{n,2}-1) + s})\right)_{n > 0} \in \bs{\mathcal{P}}_\R$.
\end{enumerate}
\end{proposition}

\begin{remark}
\label{rem:example2}
Notice that unlike in Example \ref{eg:example1} above, the dimensions of the fibers of the map
$\pi_n$ grows with $n$. However, the degrees of the polynomials used in the definition of the 
sets $V_n$ is restricted to $2$. 
\end{remark}

\begin{proof}
The proof of the proposition is somewhat similar to that of Proposition \ref{prop:example1} and we outline
it here, omitting details. It follows from the proof of 
the main result in \cite{BK08} that for each $n>0$ there exists a family of polynomials $\mathcal{Q}_n$,
computable in polynomial time, such that the stable homotopy types of the fibers 
$(V_n)_\x = V_n \cap \pi_n^{-1}(\x)$ stay invariant over each connected component of the realizations
of each realizable sign conditions on $\mathcal{Q}_n$. The degrees and the number of polynomials in $\mathcal{Q}_n$ are bounded polynomially in $n$ (for fixed $s$). 
Also, it follows from the main result in \cite{Bas05-top}
(see also \cite{BP'R07joa}), that the Betti numbers of the fibers $(V_n)_\x$ are computable in polynomial time (for fixed $s$).
Together they imply the proposition.
\end{proof}

\subsection{Sheaf-theoretic analog of $\mathbf{PH}_\R$}
\label{subsec:define-PH-sheaves}
We will now define the sheaf-theoretic version of the polynomial hierarchy. However, before 
doing so we motivate our definition by recalling the definitions of $\mathbf{NP}^c_\R$ 
and $\textbf{co-NP}^c_\R$.

If $m(n)$ is any fixed polynomial, then a sequence 
$\left(X_n \subset \Sphere^{m(n)}\right)_{n > 0}$ is in the class $\mathbf{NP}_\R^c$,
if there exists a polynomial $m_1(n)$ and a 
sequence
$(Y_n \subset \Sphere^{m_1(n)} \times \Sphere^{(m(n)})_{n > 0} \in \mathbf{P}_\R$ such that 
for each $n > 0$ the semi-algebraic set
$X_n$ is described by the formula
$(\exists y \in \Sphere^{m_1(n)}) ((y,x) \in Y_n)$.
  
Similarly,
  $\left(X_n \subset \Sphere^{m(n)}\right)_{n > 0}$ is in the class $\textbf{co-NP}_\R^c$,
if there exists a polynomial $m_1(n)$ and a 
sequence
$(Z_n \subset \Sphere^{m_1(n)} \times \Sphere^{(m(n)})_{n > 0} \in \mathbf{P}_\R$ such that 
for each $n > 0$ the semi-algebraic set
$X_n$ is described by the formula
$(\forall z \in \Sphere^{m_1(n)}) ((z,x) \in Z_n)$.
  
  Notice that in the second case it immediately follows from the fact that for all $N>0$,
 \begin{eqnarray*}
  \HH^j(\Q_{\Sphere^N}) &=& \Q, \mbox{ for } j=0,N, \\
                                       &=& 0, \mbox{ otherwise}
 \end{eqnarray*}
  that
  \[
  \Q_{X_n} = \tau^{\leq m_1(n)}\tau^{\geq m_1(n)}R\pi_{n,*}(\Q_{Z_n})[m_1(n)],
  \]
where for each $n>0$, $\pi_n: \Sphere^{m_1(n)}\times \Sphere^{m(n)} \rightarrow \Sphere^{m(n)}$ is the projection to the second factor.

In the first case, we use a construction used in \cite{BZ09}. Namely, given a sequence 
$(Y_n \subset \Sphere^{m_1(n)} \times \Sphere^{(m(n)})_{n > 0} \in \mathbf{P}_\R$,
the sequence 
\[
(J_{\pi_n}(Y_n)  \subset [0,1]\times \overline{\Ball^{m_1(n)}} \times \overline{\Ball^{m_1(n)}}\times \Sphere^{(m(n)})_{n > 0} \in \mathbf{P}_\R,
\] 
where $J_{\pi_n}(Y_n) \subset [0,1] \times \overline{\Ball^{m_1(n)}} \times \overline{\Ball^{m_1(n)}} \times \Sphere^{m(n)}$  is defined as the union of the following three sets:
$$\displaylines{
(0,1) \times (Y_n \times_{\pi_n} Y_n),   \cr
\{(1,y_0,y_1,x) \;\mid\; (y_0,x) \in Y_n\}, \cr
\{(0,y_0,y_1,x) \;\mid\; (y_1,x) \in Y_n\}.
}
$$
Denoting by $J(\pi_n): J_{\pi_n}(Y_n) \rightarrow \Sphere^{m(n)}$ the natural projection, it is easy
to verify that
the fibers of the projection $J(\pi_n): J_{\pi_n}(Y_n) \rightarrow \Sphere^{m(n)}$ are homotopy equivalent 
to the topological join with itself of the corresponding fibers of the projection $\pi_n: Y_n \rightarrow \Sphere^{m(n)}$,
and are thus connected if non-empty.
It is clear that the sequence $(J_{\pi_n}(Y_n))_{n >0} \in \mathbf{P}_\R$. Moreover, 
$J(\pi_n)(J_{\pi_n}(Y_n)) = X_n$, and for each $x\in X_n$, $(J(\pi_n))^{-1}(x)$ is connected.

Thus, the constructible sheaf $\Q_{X_n}$ can be expressed as 
\[
\Q_{X_n} = \tau^{\leq 0}\tau^{\geq 0}R(J(\pi_{n}))_* \Q_{J_{\pi_n}(Y_n)}.
\]

Notice that the quantifiers in the usual definition of the classes $\textbf{NP}_\R$ and 
$\textbf{co-NP}_\R$ are replaced by the direct image functor and truncation functors.

This motivates the following definition.
\begin{definition}
For each $p \geq 0$ we define a class of sequences
$\bs{\Lambda}^{(p)}\bs{\mathcal{P}}_\R$ as follows.
Let
\[ 
\bs{\Lambda}^{(0)}\bs{\mathcal{P}}_\R = \bs{\mathcal{P}}_\R.
\]
For $p > 0$, we define inductively
the class $\bs{\Lambda}^{(p)}\bs{\mathcal{P}}_\R$ as the smallest class of sequences of
constructible sheaves satisfying the following conditions.
\begin{enumerate}
\item
The class $\bs{\Lambda}^{(p)}\bs{\mathcal{P}}_\R$  contains
the class of sequences,
$\left(\mathcal{F}^\bullet_n\right)_{n > 0}$ for which there exists non-negative polynomials $m(n),m_0(n) \in \Z[n]$, 
and a  sequence $\left(\mathcal{G}^\bullet_n\right)_{n>0} \in \bs{\Lambda}^{(p-1)}\bs{\mathcal{P}}_\R$ such that
for each $n>0$, 
\begin{enumerate}
\item
$\mathcal{G}^\bullet_n \in \Ob(\constrD^b(\Sphere^{m_0(n)} \times \Sphere^{m(n)}))$;
\item
$\mathcal{F}^\bullet_n = R\pi_{n,*} \mathcal{G}^\bullet_n$, where $\pi_n: \Sphere^{m_0(n)} \times \Sphere^{m(n)} \rightarrow  \Sphere^{m(n)}$ is the projection to the second factor.
\end{enumerate}
\item
The class $\bs{\Lambda}^{(p)}\bs{\mathcal{P}}_\R$  is closed under taking direct sums, tensor products,
truncations and pull-backs.  More precisely, 
\begin{enumerate}
\item 
If $(\mathcal{F}^\bullet_n)_{n>0}, (\mathcal{G}^\bullet_n)_{n>0} \in  \bs{\Lambda}^{(p)}\bs{\mathcal{P}}_\R$, then 
\[
(\mathcal{F}^\bullet_n\oplus \mathcal{G}^\bullet_n)_{n>0},  (\mathcal{F}^\bullet_n\otimes \mathcal{G}^\bullet_n)_{n>0}
\in \bs{\Lambda}^{(p)}\bs{\mathcal{P}}_\R.
\]
\item
For any polynomial $m(n)$, and a class  $(\mathcal{F}^\bullet_n)_{n>0} \in  \bs{\Lambda}^{(p)}\bs{\mathcal{P}}_\R$, the sequences
 $(\tau^{\leq m(n)} \mathcal{F}^\bullet_n)_{n>0}, (\tau^{\geq m(n)} \mathcal{F}^\bullet_n)_{n>0}$ also belong to the
 class   $\bs{\Lambda}^{(p)}\bs{\mathcal{P}}_\R$.
\item
 For any fixed non-negative polynomial $m_1(n) \in \Z[n]$, let
$\pi_n:\Sphere^{m_1(n)} \times \Sphere^{m(n)} \rightarrow \Sphere^{m(n)}$ be the projection to the second factor,
and 
\[
\left(\mathcal{F}^\bullet_n \in \Ob(\constrD^b(\Sphere^{m(n)})\right)_{n>0} \in  \bs{\Lambda}^{(p)}\bs{\mathcal{P}}_\R.
\] 
Then, $\left(\pi_{n}^{-1}(\mathcal{F}^\bullet_n) \right)_{n>0} \in \bs{\Lambda}^{(p)}\bs{\mathcal{P}}_\R$.
\end{enumerate}
\end{enumerate}
Finally, 
$\bs{\mathcal{PH}}_\R$
  is defined as the union
$$
\displaylines{
\bs{\mathcal{PH}}_\R
=  \bigcup_{p\geq 0} \bs{\Lambda}^{(p)}\bs{\mathcal{P}}_\R
}
$$
The class 
$\bs{\mathcal{PH}}_\R$
 is the  analog 
in the category of constructible sheaves
of the class
$\mathbf{PH}_\mathbb{R}$.
\end{definition}

The following example is closely related to the language in the B-S-S complexity class $\mathbf{NP}_\R$ of all real polynomials in $n$ variables having degree at most $4$, whose set of real zeros is non-empty. It 
is well known \cite{BCSS98} that this problem is $\mathbf{NP}_\R$-complete.

\begin{example}
\label{eg:degree-four}
For each $n$ let $V_n \subset \Sphere^{n-1} \times \Sphere^{N_{n,4}-1}$ denote the real variety defined by 
\[
V_n = \{ (\x, P) \mid P \in \Sym_{n,4}\cap  \Sphere^{N_{n,4}-1} , \x \in \Sphere^{n-1} | P(\x) = 0\}.
\]
Let as usual $\pi_n: \Sphere^{n-1} \times \Sphere^{N_{n,4}-1} \rightarrow \Sphere^{N_{n,4}-1}$ denote the projection to the second factor.

It is now easy to verify that the sequence $\left(R\pi_{n,*} \Q_{V_n} \in \Ob(\constrD^b(\Sphere^{N_{n,4}-1}))\right)_{n>0}$ belongs to the class $\bs{\Lambda}^{(1)}\bs{\mathcal{P}}_\R$. 
We conjecture that this sequence does not belong to the class $\bs{\mathcal{P}}_\R$. 
\end{example}

\subsection{Some important inclusion relationships}
\label{subsec:inclusions}
In this section we relate the class of sequences of sheaves $\bs{\mathcal{P}}_\R$ with the B-S-S
complexity class $\mathbf{P}_\R^c$, and more generally the class $\bs{\mathcal{PH}}_\R$ with the
class $\mathbf{PH}^c_\R$. This allows us to relate the existing question about separating the class
$\mathbf{P}_\R^c$ from the class $\mathbf{PH}^c_\R$ in B-S-S theory, with the new question of separating the class
$\bs{\mathcal{P}}_\R$ from $\bs{\mathcal{PH}}_\R$.

\begin{proposition}
\label{prop:inclusion}
\begin{enumerate}
\item
A sequence $(S_n \subset \Sphere^{m(n)})_{n > 0}$ belongs to the class
$\mathbf{P}^c_\R$ if and only if $(\Q_{S_n} \in \Ob(\constrD^b(\Sphere^{m(n)}))_{n > 0}$ belongs to the class  $\bs{\mathcal{P}}_\R$.
\item
If a sequence $(S_n \subset \Sphere^{m(n)})_{n > 0}$ belongs to the complexity class $\mathbf{\Sigma}_{\R,p}^c
\cup \mathbf{\Pi}_{\R,p}^c$, then 
the sequence $(\Q_{S_n} \in \Ob(\constrD^b(\Sphere^{m(n)}))_{n > 0}$ belongs to the class  $\bs{\Lambda}^{(p)}\bs{\mathcal{P}}_\R$. 
\end{enumerate}
In particular,
if a sequence $(S_n \subset \Sphere^{m(n)})_{n > 0}$ belongs to the class $\mathbf{PH}^c_\R$ then  $(\Q_{S_n} \in \Ob(\constrD^b(\Sphere^{m(n)})_{n > 0}$ belongs to the class $\bs{\mathcal{PH}}_\R$.
\end{proposition}

\begin{proof}
The first part has been already proved in  Proposition \ref{prop:sheaf-P-vs.-set-P}.
Now suppose that $(S_n \subset \Sphere^{m(n)})_{n > 0}$ belongs to the complexity class $\mathbf{\Sigma}_{\R,p}^c$. We prove that
$(\Q_{S_n} \in \Ob(\constrD^b(\Sphere^{m(n)}))_{n > 0}$ belongs to the class  $\bs{\Lambda}^{(p)}\bs{\mathcal{P}}_\R$ using induction on $p$. Suppose that the claim holds for all smaller values of $p$. The base case $p=0$ follows from the first
part.

Then, since $(S_n \subset \Sphere^{m(n)})_{n > 0}$ belongs to the complexity class $\mathbf{\Sigma}_{\R,p}^c$, by definition,
there exists
polynomials
\[
m(n),m_1(n),\ldots,m_p(n)\]
and for 
for each $n > 0$ the semi-algebraic set
$S_n$ is described by a first order formula
\[
 (Q_1 \Y^{1} \in \Sphere^{m_1(n)})  \cdots (Q_p \Y^{p} \in 
\Sphere^{m_p(n)} )
\phi_n(X_0,\ldots,X_{m(n)},\Y^1,\ldots,\Y^p),
\]
with $\phi_n$ a quantifier-free first order formula defining a 
{\em closed} semi-algebraic subset of 
$\Sphere^{m_1(n)}\times\cdots\times
\Sphere^{m_p(n)}\times \Sphere^{m(n)}$
and for each $i, 1 \leq i \leq p$
$Q_i \in \{\exists,\forall\}$, with $Q_j \neq Q_{j+1}, 1 \leq j < p$,
$Q_1 = \exists$, and
the sequence of semi-algebraic sets
$(T_n \subset \Sphere^{m_1(n)}\times\cdots\times 
\Sphere^{m_p(n)}\times\Sphere^{m(n)})_{n > 0}$
defined by the formulas $(\phi_n)_{n >0}$ belongs to the class
$\mathbf{P}_\R$.

The sequence $\left(S'_n \subset \Sphere^{m_1(n)}\times \Sphere^{m(n)}\right)_{n >0}$,
and hence the sequence $\left(J_{\pi_n}(S_n') \right)_{n>0}$,
belongs to the class $\mathbf{\Pi}_{\R,p-1}^c$.
where each $S_n'$ is defined by the formula 
\[
 (Q_2 \Y^{2} \in \Sphere^{m_2(n)})  \cdots (Q_p \Y^{p} \in 
\Sphere^{m_p(n)} )
\phi_n(X_0,\ldots,X_{m(n)},\Y^1,\ldots,\Y^p),
\]
and $\pi_n: \Sphere^{m_1(n)} \times \Sphere^{m(n)} \rightarrow \Sphere^{m(n)}$ is the projection
to the second factor.
Using the inductive hypothesis, 
the sequence $\left(\Q_{J_{\pi_n}(S_n')}\right)_{n > 0} \in \bs{\Lambda}^{(p-1)}\bs{\mathcal{P}}_\R$. 
Then, for each $n>0$,
 $\Q_{S_n} = \tau^{\geq 0}\tau^{\leq 0}R\pi_{n,*} \Q_{J_{\pi_n}}(S_n')$, and it follows 
 from the definition of the class $\bs{\Lambda}^{(p-1)}\bs{\mathcal{P}}_\R$, that 
$\left(\Q_{S_n} \right)_{n > 0}$ belongs to the class  $\bs{\Lambda}^{(p)}\bs{\mathcal{P}}_\R$.

Now suppose that $(S_n \subset \Sphere^{m(n)})_{n > 0}$ belongs to the complexity class $\mathbf{\Pi}_{\R,p}^c$. 

Then,  there exists
polynomials
\[
m(n),m_1(n),\ldots,m_p(n)\]
and for 
for each $n > 0$ the semi-algebraic set
$S_n$ is described by a first order formula
\[
 (Q_1 \Y^{1} \in \Sphere^{m_1(n)})  \cdots (Q_p \Y^{p} \in 
\Sphere^{m_p(n)} )
\phi_n(X_0,\ldots,X_{m(n)},\Y^1,\ldots,\Y^p),
\]
with $\phi_n$ a quantifier-free first order formula defining a 
{\em closed} semi-algebraic subset of 
$\Sphere^{m_1(n)}\times\cdots\times
\Sphere^{m_p(n)}\times \Sphere^{m(n)}$
and for each $i, 1 \leq i \leq p$
$Q_i \in \{\exists,\forall\}$, with $Q_j \neq Q_{j+1}, 1 \leq j < p$,
$Q_1 = \forall$, and
the sequence of semi-algebraic sets
$(T_n \subset \Sphere^{m_1(n)}\times\cdots\times 
\Sphere^{m_p(n)}\times\Sphere^{m(n)})_{n > 0}$
defined by the formulas $(\phi_n)_{n >0}$ belongs to the class
$\mathbf{P}_\R$.

The sequence $\left(S'_n \subset \Sphere^{m_1(n)}\times \Sphere^{m(n)}\right)_{n >0}$
belongs to the class $\mathbf{\Sigma}_{\R,p-1}^c$
where each $S_n'$ is defined by the formula 
\[
 (Q_2 \Y^{2} \in \Sphere^{m_2(n)})  \cdots (Q_p \Y^{p} \in 
\Sphere^{m_p(n)} )
\phi_n(X_0,\ldots,X_{m(n)},\Y^1,\ldots,\Y^{p}).
\]

Using the inductive hypothesis, 
the sequence $\left(\Q_{S_n'}\right)_{n > 0} \in \bs{\Lambda}^{(p-1)}\bs{\mathcal{P}}_\R$. 
Then, for each $n>0$,
 $\Q_{S_n} = \tau^{\geq 0}\tau^{\leq 0}R\pi_{n,{*}} \Q_{J_{\pi_n}}(S_n')[m_1(n)]$, and it follows 
 from the definition of the class $\bs{\Lambda}^{(p-1)}\bs{\mathcal{P}}_\R$, that 
$\left(\Q_{S_n} \right)_{n > 0}$ belongs to the class  $\bs{\Lambda}^{(p)}\bs{\mathcal{P}}_\R$.
\end{proof}

It is an immediate consequence of Proposition \ref{prop:inclusion} that:

\begin{theorem}
\label{thm:inclusion2}
For each $p\geq 0$, 
the equality $\bs{\Lambda}^{(p)}\bs{\mathcal{P}}_\R = \bs{\mathcal{P}}_\R$ implies
$\mathbf{\Sigma}_{\R,p}^c \cup \mathbf{\Pi}_{\R,p}^c = \mathbf{P}_\R^c$.
\end{theorem}

In view of Theorem \ref{thm:inclusion2} we conjecture that:

\begin{conjecture}
\label{conj:main}
For each $p >0$,
\[
\bs{\Lambda}^{(p)}\bs{\mathcal{P}}_\R \neq  \bs{\mathcal{P}}_\R.
\]
\end{conjecture}

\begin{remark}
Notice that Conjecture \ref{conj:main} is a priori weaker than the more standard conjecture that 
$\mathbf{\Sigma}_{\R,p}^c \cup \mathbf{\Pi}_{\R,p}^c \neq \mathbf{P}_\R^c$.
\end{remark}

\begin{remark}
The reverse implication in Theorem \ref{thm:inclusion2} is probably not true, even though we do not
have a counter-example at this point. In fact, the class of  sequences 
$
\left(S_n \subset \Sphere^{m(n)}\right)_{n>0}
$ 
for which
$\left( \Q_{S_n}\right)_{n>0} $ belongs to  $\bs{\Lambda}^{(1)}\bs{\mathcal{P}}_\R$ ``interpolates''  between the classes
$\textbf{NP}_\R^c$ and $\textbf{co-NP}_\R^c$. 

For example, consider a sequence $\left(S_n \subset \Sphere^{m(n)}\right)_{n>0}$ for which there exists a sequence 
$\left(T_n \subset \Sphere^{m_1(n)}\times \Sphere^{m(n)}\right)_{n>0} \in \textbf{P}_\R^c$, such that
\[
S_n = \{x\in \Sphere^{m(n)} \;\mid\; \HH^{m_2(n)}(T_{n,x},\Q) \cong \Q\},
\] 
where $m_1(n),m_2(n)$ are non-negative
polynomials, and $T_{n,x} = \pi_n^{-1}(x) \cap T_n$ where 
$\pi_n:  \Sphere^{m_1(n)}\times \Sphere^{m(n)} \rightarrow  \Sphere^{m(n)}$ is the projection to the second factor. 
It is not clear that such a sequence will belong to either $\textbf{NP}_\R^c$ and $\textbf{co-NP}_\R^c$, even though clearly
$\left( \Q_{S_n}\right)_{n>0} \in \bs{\Lambda}^{(1)}\bs{\mathcal{P}}_\R$.
\end{remark}

\section{Singly exponential upper bounds}
\label{sec:topological-complexity-sheaves}
In this section we study the complexity of constructible sheaves from an effective point of view.
In Section \ref{subsec:topological-complexity-sheaves} we prove a singly exponential upper bound on the ``complexity'' of constructible sheaves belonging to the class 
$\bs{\mathcal{PH}}_\R$.
We then prove the existence of an algorithm having singly exponential complexity for
``generalized quantifier elimination'' defined below in Section \ref{subsec:generalized-qe}.

\subsection{Bounding the topological complexity of constructible sheaves}
\label{subsec:topological-complexity-sheaves}
A standard way in semi-algebraic geometry to measure the topological complexity of a semi-algebraic
set is by the sum of their Betti numbers. The topological complexity of a set $S$, measured by the sum of its Betti numbers,  has been shown to be related to hardness of testing membership in $S$
\cite{Yao94,Pardo96}. Also, the topological complexity is a rough guide to the best possible complexity one can hope for of algorithms computing topological invariants of a set $S$. For example, if the topological complexity is bounded by a polynomial in the input parameters, then one often gets algorithms also with polynomial complexity (see for example \cite{Bar93,Bas05-top,Bas05-euler,BP'R07joa} 
in the case of semi-algebraic sets defined by few quadratic polynomials for an instance of this
phenomena). 

For a sequence of semi-algebraic sets $(S_n \subset \R^n)$
one can study the dependence of  the Betti numbers of $S_n$ as a function of $n$, and in particular 
estimate the growth of this sequence.
It follows from complexity estimates on effective quantifier elimination  (see for example \cite{BPRbook2}), as well as estimates on the Betti numbers of semi-algebraic sets in terms of the number and degrees of the polynomials occurring in
their definition  \cite{GV07}, that the Betti numbers of the sets $S_n$, for any sequence 
$\left(S_n \subset \R^{m(n)}\right)_{n >0} \in \textbf{PH}_\R$,
are bounded singly exponentially as a function of $n$. This gives an upper bound on the growth of topological complexity of the sequence $\left(S_n \subset \R^{m(n)}\right)_{n >0}$. 
We now state this fact more formally.

\begin{theorem}
\label{thm:topological-complexity-sets}
Let $\left(S_n \subset \R^{m(n)}\right)_{n>0}$ belong to the class $\mathbf{PH}_\R$. Then,  there exists a polynomial $q(n)$ (depending on the
sequence $\left(S_n\right)_{n>0}$),   such that
\[
b(S_n) \leq 2^{q(n)}.
\]
In other words,  $b(S_n)$ is bounded
singly exponentially as a function of $n$. 
A similar statement holds for the compact class $\mathbf{PH}^c_\R$ as well.
\end{theorem}

\begin{proof}
Recall that, if $\left(S_n \subset \R^{m(n)}\right)_{n>0}$ belong to the class $\mathbf{PH}_\R$, then
for each $n > 0$ the semi-algebraic set
$S_n$ is described by a 
first order formula
\begin{equation}
(Q_1 \Y^{1} )  \cdots (Q_p\Y^{p} )
\phi_n(X_1,\ldots,X_{k(n)},\Y^1,\ldots,\Y^p),
\end{equation}
with $\phi_n$ a quantifier free formula in the first order theory of the reals,
and for each $i, 1 \leq i \leq p$,
$\Y^i = (Y^i_1,\ldots,Y^i_{k_i(n)})$ is a block of $k_i(n)$ variables,
$Q_i \in \{\exists,\forall\}$, with $Q_j \neq Q_{j+1}, 1 \leq j < p$,
$Q_1 = \exists$,
and 
the sequence of semi-algebraic sets  $(T_n \subset \R^{k(n)+ k_1(n) + \cdots + k_p(n)})_{n >0}$ 
defined by the quantifier-free formulas $(\phi_n)_{n>0}$ 
belongs to the class $\mathbf{P}_\R$.

It follows from the definition of the class $\mathbf{P}_\R$ that the sets $T_n$ have a description by
a quantifier-free first order formula with atoms of the form $P \{>,=,<\}0, P\in \R[X_1,\ldots,X_n]$,
and the number and degrees of the polynomials appearing in the formula are bounded singly
exponentially in $n$. It then follows from the complexity estimates on effective quantifier elimination
in the theory of the reals (see for example  \cite{Gri88,R92},\cite[Theorem 14.16]{BPRbook2}), that the
same is true for the sets $S_n$ as well. It now follows from \cite[Theorem 6.8]{GV07} (see also
\cite[Theorem 7.50]{BPRbook2}) that $b(S_n)$ is bounded singly exponentially in $n$.
\end{proof}

\begin{remark}
Note that for the compact class $\mathbf{PH}^c_\R$, the statement of Theorem \ref{thm:topological-complexity-sets} also follows directly from and Proposition \ref{prop:inclusion} and Theorem \ref{thm:topological-complexity-sheaves}. 
\end{remark}

In the sheaf context, in order to measure the complexity of a given constructible sheaf $\mathcal{F}^\bullet \in \Ob(\constrD^b(X))$ for any semi-algebraic set $X$, we 
we will use the sum of the dimensions of the hypercohomology groups, $\hyperH^i(X,\mathcal{F}^\bullet)$
(Definition \ref{def:hypercohomology}).

\begin{notation}
For a constructible sheaf $\mathcal{F}^\bullet \in \Ob(\constrD^b(X))$, we will denote
\begin{eqnarray*}
b^i(\mathcal{F}^\bullet) &=& \dim_\Q\hyperH^i(X,\mathcal{F}^\bullet), \\
b(\mathcal{F}^\bullet) &=& \sum_i b^i(\mathcal{F}^\bullet).
\end{eqnarray*}
\end{notation}

\begin{definition}
\label{def:topological-complexity-sheaves}
We call $b(\mathcal{F}^\bullet)$ the \emph{topological complexity} of $\mathcal{F}^\bullet$.
\end{definition}

\begin{theorem}
\label{thm:topological-complexity-sheaves}
Let $\left(\mathcal{F}^\bullet_n \in \Ob(\constrD^b(\Sphere^{m(n)})\right)_{n >0}$ be a sequence of constructible sheaves belonging to the class
$\bs{\mathcal{PH}}_\R$.
Then,  there exists a polynomial $q(n)$ (depending on the
sequence $\left(\mathcal{F}^\bullet_n\right)_{n>0}$),  such  that  
\[
b(\mathcal{F}^\bullet_n) \leq  2^{q(n)}.
\]
In other words,  $b(\mathcal{F}^\bullet_n)$ is bounded
singly exponentially as a function of $n$.
\end{theorem}

Before proving Theorem \ref{thm:topological-complexity-sheaves} we need a few preliminaries.

\begin{proposition}
\label{prop:constant}
Let $\K$ be a compact semi-algebraic set, and $\mathcal{F}^\bullet \in \Ob(\constrD^b(\K))$, and suppose that $C \subset \Sphere^n$ is a locally closed and contractible
subset of $\Sphere^n$. Suppose that $\HH^*(\mathcal{F}^\bullet)|_{C}$ is locally constant on $C$. Then, 
$\HH^*(\mathcal{F}^\bullet)|_{C'}$ is 
isomorphic to a constant sheaf in 
for any locally closed subset $C' \subset C$.
\end{proposition}

\begin{proof}
The proposition follows from the fact that any locally constant sheaf over a contractible set
is isomorphic to a constant sheaf  (Proposition \ref{prop:contractibility-implies-constant}),
and the fact that the restriction of a constant sheaf to a subspace again yields a constant sheaf. 
\end{proof}

\begin{proposition}(Covering by contractibles)
\label{prop:covering}
Let $\mathcal{P} \subset \R[X_1,\ldots,X_{n+1}]$ be a finite family of polynomials, with $\card(\mathcal{P}) =s$,
and $d = \max_{P \in \mathcal{P}} \deg(P)$. Then, there exists a family of polynomials 
$\mathcal{Q}\subset \R[X_1,\ldots,X_{n+1}]$ having the following properties.
\begin{enumerate}
\item
The degrees of the polynomials in $\mathcal{Q}$, as well as $\card(\mathcal{Q})$ are bounded by $(sd)^{n^{O(1)}}$.
\item
The signs of the polynomials in $\mathcal{P}$ are invariant at all points of any semi-algebraically connected
component $D$ of the realization $\RR(\rho,\Sphere^n)$ of any realizable sign condition $\rho \in \{0,+1,-1\}^{\mathcal{Q}}$.
\item
For each realizable sign condition $\sigma \in \{0,1,-1\}^{\mathcal{P}}$, the realization $S = \RR(\sigma,\Sphere^n)$
is a disjoint union of  locally closed semi-algebraic sets $D$. Each such $D$ is a connected component of some realizable sign condition on $\mathcal{Q}$,  and is contained in a contractible semi-algebraic subset $D' \subset S$ which is closed in $S$.
\end{enumerate}
\end{proposition}

\begin{proof}
The proposition follows by applying Algorithm 16.14 (Covering by contractible sets) in \cite{BPRbook2}
taking the family of polynomials $\mathcal{P}$ as input. We omit the details.
\end{proof}

Following the same notation as in Proposition \ref{prop:covering} we also have:
\begin{corollary}
\label{cor:covering}
Suppose $\mathcal{F}^\bullet \in \Ob(\constrD^b(\Sphere^n))$ such that $\mathcal{H}^*(\mathcal{F}^\bullet)$ is locally constant on each
connected component $C$ of the realization of each realizable sign condition $\sigma \in \{0,1,-1\}^{\mathcal{P}}$. Then, $\mathcal{H}^*(\mathcal{F}^\bullet)|_D$ is a constant sheaf for $D$ a connected
component of the realization of any realizable sign condition of $\mathcal{Q}$.
\end{corollary}

\begin{proof}
This follows immediately from Propositions \ref{prop:covering} and  \ref{prop:constant}.
\end{proof}

We will also need some properties regarding the invariance of sheaf cohomology under ``infinitesimal'' ``thickenings'' and ``shrinkings'' of the underlying semi-algebraic sets. 

\begin{remark}
\label{rem:homotopy-invariance}
In the case of ordinary cohomology (or in 
sheaf-theoretic language,  cohomology with values in a constant sheaf) these facts are standard and
follow from homotopy invariance of the cohomology groups and the local conic structure of
semi-algebraic sets \cite[Theorem 9.3.6]{BCR}. However, in general sheaf cohomology
is not  homotopy invariant (see for example \cite[Example 3.1.6]{Dimca-sheaves}), 
and so more care needs to be taken. 
\end{remark}

The following proposition 
follows directly from \cite[Proposition 2.7.1]{KS}.
\begin{proposition}(Shrinking)
\label{prop:shrinking}
Let $\K$ be a compact semi-algebraic set and $\mathcal{F}^\bullet \in \Ob(\constrD^b(\K))$. Let $X$ be a locally
closed semi-algebraic subset of $\K$ and suppose that $(X_n \subset X)_{n \in \N}$ is an increasing family of compact semi-algebraic subsets of $X$ such that,
$X_n \subset Int(X_{n+1)}$ for each $n$, and $X = \cup_{n} X_n$. Then, the natural map induced by restriction
\[
\phi_j: \hyperH^j(X,\mathcal{F}^\bullet|_X) \rightarrow \varprojlim_n \hyperH^j(X_n,\mathcal{F}^\bullet|_{X_n})
\]
is an isomorphism for all $j \in \Z$. In particular, since all the hypercohomology groups,
$\hyperH^j(X_n,\mathcal{F}^\bullet|_{X_n})$ are finite dimensional, and vanish for $|j| > N$ for some $N$, we have
that for all $n$ large enough,
\[
\hyperH^j(X,\mathcal{F}^\bullet|_X) \cong  \hyperH^j(X_n,\mathcal{F}^\bullet|_{X_n}).
\]
\end{proposition}

We also have the following.
\begin{proposition}(Thickening)
\label{prop:thickening}
Let $\K$ be a compact semi-algebraic set and $\mathcal{F}^\bullet \in \Ob(\constrD^b(\K))$. Let $Z$ be a 
closed semi-algebraic subset of $\K$ and suppose that $(Z_n \subset X)_{n \in \N}$ is a decreasing family of closed semi-algebraic subsets of $\K$ such that $Z = \cap_{n} Z_n$. Then, the natural map induced by restriction
\[
\phi_j:  \varinjlim_n \hyperH^j(Z_n,\mathcal{F}^\bullet|_{Z_n}) \rightarrow \hyperH^j(Z,\mathcal{F}^\bullet|_Z) 
\]
is an isomorphism for all $j \in \Z$. In particular, we have
that for all $n$ large enough,
\[
\hyperH^j(Z,\mathcal{F}^\bullet|_Z) \cong  \hyperH^j(Z_n,\mathcal{F}^\bullet|_{Z_n}).
\]
\end{proposition}  

\begin{proof}
See \cite[Remark 2.6.9]{KS}.
\end{proof}

In the proofs of statements that follow we are going to use the language of infinitesimals. In many works in quantitative
real algebraic geometry (see for example \cite{BPRbook2}) it is customary to consider non-archimedean extensions of a given real
closed field $\mathrm{R}$, by considering the field of algebraic Puiseux series in new variables, which are  ``infinitesimally small'' with respect to the ground field $\mathrm{R}$ (meaning that they are positive but smaller than all positive elements of $\mathrm{R}$). In this paper, we avoid working with non-archimedean extensions which necessitates the following definition.

\begin{definition}
\label{def:infinitesimals}
Let $\bar\eps = (\eps_1,\ldots,\eps_s)$ be a tuple of variables.
  By the phrase ``$0< \bar\eps \ll 1$''  we will mean 
``$0 <\eps_1 \ll \eps_2 \ll \cdots \ll \eps_{s-1} \ll \eps_{s} \ll 1$'' (i.e.,
``for all sufficiently small positive $\eps_{s}$, and
then for all sufficiently small positive $\eps_{s-1}$, etc.'').  
\end{definition}

\begin{notation}
\label{not:misc}
Let $\mathcal{P} = \{P_1,\ldots,P_s\} \subset \R[X_0,\ldots,X_n]$ be a family of
polynomials.
Let $\bar\eps = (\eps_1,\eps_2,\ldots,\eps_{3s})$.

We will denote by $\mathcal{P}_{\bar\eps}$ the family defined by 
\[
\mathcal{P}_{\bar\eps} = \cup_{1 \leq i,j \leq s} \{P_i \pm \eps_{2j-1}\eps_{2s+i}, P_i \pm \eps_{2j}\eps_{2s+i}\}.
\]

For two sign conditions $\sigma_1,\sigma_2 \in \{0,1,-1\}^{\mathcal{P}}$, we denote 
$\sigma_1 \prec \sigma_2$ if and only if $\sigma_2(P) = 0 \Rightarrow \sigma_1(P) = 0$ for all
$P \in \mathcal{P}$.

For a sign condition $\sigma \in \{0,1,-1\}^{\mathcal{P}}$, we denote
\begin{equation}
\label{eqn:level}
\mathrm{level}(\sigma) = \card(\{P \in \mathcal{P} \mid \sigma(P) = 0\}),
\end{equation}
and denote by 
$\sigma_{\bar\eps}$ be the formula defined as the conjunction of the following three
formulas, 
\[
\bigwedge_{1\leq i \leq s, \sigma(P_i) = 0} (-\eps_{2\ell-1}\eps_{2s+i} \leq P_i \leq \eps_{2\ell-1}\eps_{2s+i}), 
\]
\[
\bigwedge_{1\leq i \leq s, \sigma(P_i) > 0} (P_i \geq \eps_{2\ell}\eps_{2s+i}),
\]
\[
\bigwedge_{1\leq i \leq s, \sigma(P_i) < 0} (P \leq -\eps_{2\ell}\eps_{2s+i}),
\]
where $\ell = \level(\sigma)$.
\end{notation}

\begin{remark}
\label{rem:closed-formula}
Note that each $\sigma_{\bar\eps}$ is a $\mathcal{P}_{\bar\eps}$-closed formula defining a 
$\mathcal{P}_{\bar\eps}$-closed semi-algebraic set (cf. Notation \ref{not:P-formula}).
\end{remark}

We will use the following weak notion of ``general position'' (see \cite{BPRbook2}) for a finite 
family of polynomials.
\begin{definition}(General position).
\label{def:general-position}
Given a finite set of polynomials $\mathcal{P} \subset \R[X_1,\ldots,X_n]$, $ 0 \leq \ell \leq n$, 
and a semi-algebraic subset $\K \subset \R^{n}$, we say that $\mathcal{P}$ is in \emph{$\ell$-general position with respect to $\K$}, if no subset of $\mathcal{P}$ of cardinality larger than $\ell$, has a common zero in $\K$.
\end{definition}

The following proposition is easy to check  and we omit its proof 
(see for example \cite[Corollary 1]{BPR95a}.)

With the same notation as in Notation \ref{not:misc} above,
\begin{proposition}
\label{prop:general-position}
Let $\K \subset \R^{n}$ be a real variety of dimension $\leq \ell$.
For $0 < \bar\eps \ll 1$, $\mathcal{P}_{\bar\eps}$ is in $\ell$-general position with respect
to $\K$.
\end{proposition}

We will also use the following notation.

\begin{notation}
\label{not:cc}
Given any semi-algebraic set $S$ we will denote by $\Cc(S)$ the set of connected components of $S$.
Note that $\Cc(S) = \emptyset$ if and only if $S = \emptyset$.
\end{notation}

Also, using the same notation introduced in Notation \ref{not:misc} we have the following proposition.

\begin{proposition}
\label{prop:perturbation}
Let $\K$ be a compact semi-algebraic subset of $\R^n$. 
Suppose $\mathcal{F}^\bullet \in \Ob(\constrD^b(\K))$, and for each $q\in\Z$, let
$\mathcal{H}^q(\mathcal{F}^\bullet)$
be the sheaf associated to the pre-sheaf defined by
$\mathcal{H}^q(\mathcal{F}^\bullet)(V) = \hyperH^q(V,\mathcal{F}^\bullet|_V)$.
Suppose that for each $\sigma \in \{0,1,-1\}^{\mathcal{P}}$ and 
connected component $C$ of 
$\RR(\sigma,\K)$, $\mathcal{H}^q(\mathcal{F}^\bullet)|_C$ is a constant sheaf.

Then for all $0 < \bar\eps \ll 1$ (see Notation \ref{not:misc} above) the following is true.
\begin{enumerate}
\item
\label{item:perturbationA}
If $C \in \Cc(\RR(\sigma,\K))$ for some $\sigma \in \{0,1,-1\}^{\mathcal{P}}$, then there exists a unique connected component, $C_{\bar\eps} \in \Cc(\RR(\sigma_{\bar\eps},\K))$, such that $C \cap C_{\bar\eps} \neq \emptyset$.
\item 
\label{item:perturbationB}
The semi-algebraic set $C_{\bar\eps}$ is closed in $\K$ and homotopy equivalent to $C$.
\item
\label{item:perturbationC}
For ever $q \in \Z$, and $\x\in C$,
\[
\hyperH^q(C,\mathcal{F}^\bullet|_C) \cong \hyperH^q(C_{\bar\eps},\mathcal{F}^\bullet|_{C_{\bar\eps}}) 
\cong \bigoplus_{i+j=q} \HH^i(C,\Q)\otimes\HH^j(\mathcal{F}^\bullet_\x).
\]
\item
\label{item:perturbationD}
More generally, suppose that 
$\sigma_0,\ldots,\sigma_p \in \{0,1,-1\}^\mathcal{P}$,
and  that 
\[
C \in \Cc(\RR(\sigma_{0,\bar\eps},\K) \cap \cdots \cap \RR(\sigma_{p,\bar\eps},\K)).
\]

Then, there exists
a permutation $\pi$ of $\{0,\ldots,p\}$ such that 
$$
\displaylines{
\sigma_{\pi(0)} \prec \cdots \prec \sigma_{\pi(p)}, \mbox{ and }\cr
C \subset \RR(\sigma_{\pi(p),\bar\eps},\K).
}
$$
Moreover, 
\begin{enumerate}
\item
$C$ is closed  in $\K$,
\item
$C$ is homotopy equivalent to $C' = \RR(\sigma_{\pi(p)},\K) \cap C$, and 
\item
for every $q\in\Z$ and $\x \in C'$,
\[
\hyperH^q(C,\mathcal{F}^\bullet|_{C}) \cong \hyperH^q(C',\mathcal{F}^\bullet|_{C'})\cong 
\bigoplus_{i+j=q} \HH^i(C,\Q) \otimes \HH^j(\mathcal{F}^\bullet_\x).
\]
\end{enumerate}
\end{enumerate}
\end{proposition}

\begin{example}

The following example might be helpful. Let $n=2$, $\K = \{(x_1,x_2) \mid |x_i| \leq R, i=1,2\}$ for some
large $R>0$,  and $\mathcal{P} = \{P_1,P_2\}$, with $P_1=X_1, P_2 =X_2$.
Let $\sigma_0$ be the sign condition $P_1=0,P_2=0$ and $\sigma_1$ the sign condition $P_1>0,P_2=0$.
The  corresponding sets 
\begin{eqnarray*}
\RR(\sigma_{0,\bar{\eps}},\K) &=& \{(x_1,x_2)\;\mid\; 0 \leq|x_1| \leq \eps_3\eps_5, 0 \leq |x_2| \leq \eps_3\eps_6\}, \\
\RR(\sigma_{1,\bar{\eps}},\K) &=& \{(x_1,x_2)\;\mid\;  x_1 \geq  \eps_2\eps_5, 0 \leq|x_2| \leq \eps_1\eps_6,\},
\end{eqnarray*}
where $0 < \eps_1 \ll \cdots \ll \eps_6 \ll 1$,  are depicted schematically in 
Figure \ref{fig:example} as  lightly shaded rectangles, while their intersection is shaded with a darker shade. In this example, the set
$C'$ is the line segment defined by $C' = \{(x_1,0) \;\mid\; \eps_2\eps_5 \leq x_1 \leq \eps_3\eps_5\}$.

\begin{figure}
\label{fig:example}
\begin{picture}(0,0)%
\includegraphics{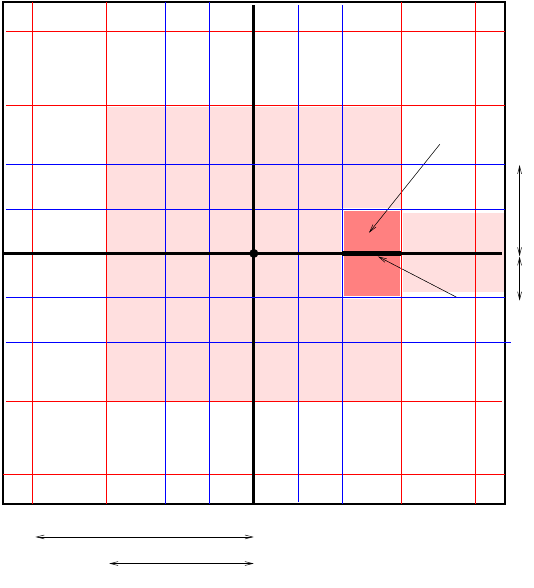}%
\end{picture}%
\setlength{\unitlength}{2072sp}%
\begingroup\makeatletter\ifx\SetFigFont\undefined%
\gdef\SetFigFont#1#2#3#4#5{%
  \reset@font\fontsize{#1}{#2pt}%
  \fontfamily{#3}\fontseries{#4}\fontshape{#5}%
  \selectfont}%
\fi\endgroup%
\begin{picture}(8114,8949)(857,-8527)
\put(1981,-3751){\makebox(0,0)[lb]{\smash{{\SetFigFont{9}{10.8}{\familydefault}{\mddefault}{\updefault}{\color[rgb]{0,0,0}$X_1=0$}%
}}}}
\put(4861,-1501){\makebox(0,0)[lb]{\smash{{\SetFigFont{9}{10.8}{\familydefault}{\mddefault}{\updefault}{\color[rgb]{0,0,0}$X_2=0$}%
}}}}
\put(7606,-1681){\makebox(0,0)[lb]{\smash{{\SetFigFont{9}{10.8}{\familydefault}{\mddefault}{\updefault}{\color[rgb]{0,0,0}$C$}%
}}}}
\put(7606,-4381){\makebox(0,0)[lb]{\smash{{\SetFigFont{9}{10.8}{\familydefault}{\mddefault}{\updefault}{\color[rgb]{0,0,0}$C'$}%
}}}}
\put(5986,-3256){\makebox(0,0)[lb]{\smash{{\SetFigFont{9}{10.8}{\familydefault}{\mddefault}{\updefault}{\color[rgb]{0,0,0}$\RR(\sigma_{1,\bar\eps},\K)$}%
}}}}
\put(3601,-2356){\makebox(0,0)[lb]{\smash{{\SetFigFont{9}{10.8}{\familydefault}{\mddefault}{\updefault}{\color[rgb]{0,0,0}$\RR(\sigma_{0,\bar\eps},\K)$}%
}}}}
\put(1666,-556){\makebox(0,0)[lb]{\smash{{\SetFigFont{9}{10.8}{\familydefault}{\mddefault}{\updefault}{\color[rgb]{0,0,0}$\K$}%
}}}}
\put(1846,-7621){\makebox(0,0)[lb]{\smash{{\SetFigFont{9}{10.8}{\familydefault}{\mddefault}{\updefault}{\color[rgb]{0,0,0}$\varepsilon_4\varepsilon_5$}%
}}}}
\put(8911,-2671){\makebox(0,0)[lb]{\smash{{\SetFigFont{9}{10.8}{\familydefault}{\mddefault}{\updefault}{\color[rgb]{0,0,0}$\varepsilon_2\varepsilon_6$}%
}}}}
\put(8956,-3886){\makebox(0,0)[lb]{\smash{{\SetFigFont{9}{10.8}{\familydefault}{\mddefault}{\updefault}{\color[rgb]{0,0,0}$\varepsilon_1\varepsilon_6$}%
}}}}
\put(3016,-8431){\makebox(0,0)[lb]{\smash{{\SetFigFont{9}{10.8}{\familydefault}{\mddefault}{\updefault}{\color[rgb]{0,0,0}$\varepsilon_3\varepsilon_5$}%
}}}}
\end{picture}%
\caption{Example illustrating Proposition \ref{prop:perturbation}}
\end{figure}
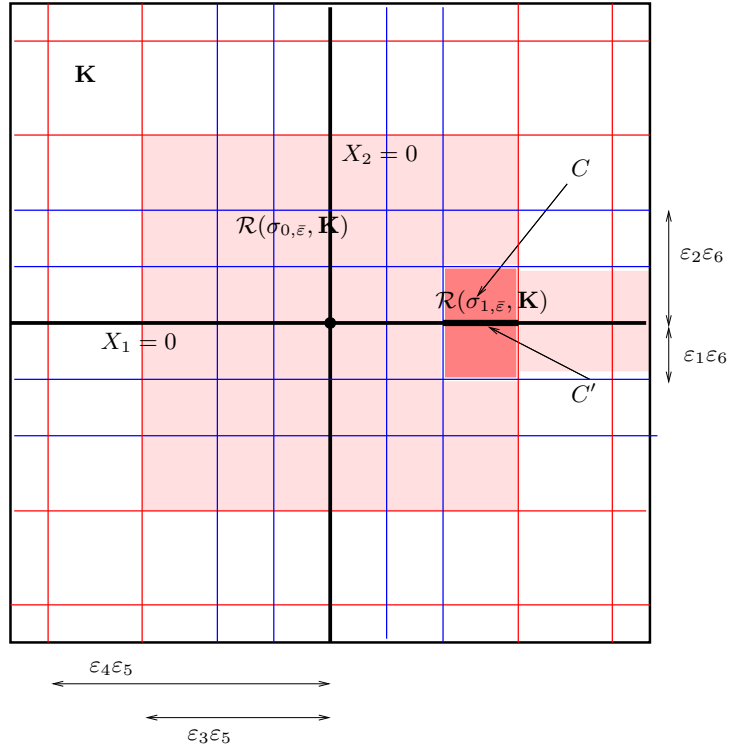
\end{example}

\begin{proof}[Proof of Proposition \ref{prop:perturbation}]
We first prove Parts 
(\ref{item:perturbationA}), (\ref{item:perturbationB},  and (\ref{item:perturbationC})
of the proposition.

Let $0 <\eps_{2s+1} \ll \cdots \ll \eps_{3s} \ll 1$ be chosen sufficiently small.
For $\gamma > 0$ let $\sigma_\gamma$ be the formula 
\[
\bigwedge_{1 \leq i \leq s, \sigma(P_i) > 0} (P_i \geq \gamma\eps_{2s+i}) \wedge
\bigwedge_{1 \leq i \leq s, \sigma(P_i) < 0} (P_i \leq -\gamma\eps_{2s+i}),
\]
and 
for $\delta > 0$, let $\sigma_{\gamma,\delta}$ be the formula defined as the conjunction of the following
formulas:
\[
\bigwedge_{1 \leq i \leq s, \sigma(P_i) = 0} (-\delta\eps_{2s+i} \leq P_i \leq \delta\eps_{2s+i}),
\]
\[
\bigwedge_{1 \leq i \leq s, \sigma(P_i) > 0} (P_i \geq \gamma\eps_{2s+i}) \wedge
\bigwedge_{1 \leq i \leq s, \sigma(P_i) < 0} (P_i \leq -\gamma\eps_{2s+i}).
\]

It is easy to show that (see for example, \cite[Lemma 16.17]{BPRbook2})  that for $0< \delta \ll \gamma \ll 1$, and for each connected component $C$ of $\RR(\sigma,\K)$ there exists a unique
connected component,
 $C_{\gamma}$ of $\RR(\sigma_{\gamma},\K)$ such that $C \cap C_\gamma \neq \emptyset$,  
 and a unique connected component
  $C_{\gamma,\delta}$ of $\RR(\sigma_{\gamma,\delta},\K)$, such that $C_{\gamma} \subset C_{\gamma,\delta}$, and
moreover 
$C_{\gamma,\delta}$ is homotopy equivalent to $C$.

Now choose a sequence $\gamma_i\searrow0$, and for each $j$ choose
$\delta_{i,j}\searrow 0$. Then for all $i$ large enough, and having chosen $i$,  for all $j$ large
enough, $C_{\gamma_i,\delta_j}$ is well defined, and it is clear that each $C_{\gamma_i,\delta_j}$ is compact and
$\cap_{j}C_{\gamma_i,\delta_i} = C_{\gamma_i}$. Using Proposition \ref{prop:thickening} we obtain that
for every $q \in \Z$, and for all $j$ large enough,
\begin{eqnarray}
\label{eqn:congruence1}
\hyperH^q(C_{\gamma_i},\mathcal{F}^\bullet|_{C_{\gamma_i}}) &\cong& \hyperH^q(C_{\gamma_i,\delta_j},\mathcal{F}^\bullet|_{C_{\gamma_i,\delta_j}}).
\end{eqnarray}
Now observe that $\cup_i C_{\gamma_i} = C$, and using Proposition \ref{prop:shrinking} we have for every 
$q \in \Z$,
\begin{eqnarray}
\label{eqn:congruence2}
\hyperH^q(C,\mathcal{F}^\bullet|_C) &\cong& \hyperH^q(C_{\gamma_i},\mathcal{F}^\bullet|_{C_{\gamma_i}}),
\end{eqnarray}
for all $i$ large enough. Moreover, noticing that $\mathcal{H}^\ell(\mathcal{F}^\bullet)|_{C_{\gamma_i}}$ is a constant sheaf for each $\ell \in \Z$
(using the fact that $\mathcal{H}^\ell(\mathcal{F}^\bullet)|_{C_{\gamma_i}}$ is constant and  Proposition \ref{prop:constant}), we have 
using Proposition \ref{prop:spectral} 
that  for every $q\in\Z$ and $\x \in C_{\gamma_i}$,
\begin{eqnarray}
\label{eqn:congruence3}
\hyperH^q(C_{\gamma_i},\mathcal{F}^\bullet|_{C_{\gamma_i}}) & \cong &  
\bigoplus_{i+j=q} \HH^i(C_{\gamma_i},\Q) \otimes \HH^j(\mathcal{F}^\bullet_\x).
\end{eqnarray}
Now using the fact that $C_{\gamma_i}$ is homotopy equivalent to $C$ for $i$ large enough, we have 
for every $q \in \Z$, 
\begin{eqnarray}
\label{eqn:congruence4}
\HH^q(C_{\gamma_i},\Q) &\cong& \HH^q(C,\Q).
\end{eqnarray}
Together \eqref{eqn:congruence1}, (\ref{eqn:congruence2}), (\ref{eqn:congruence3}) and (\ref{eqn:congruence4}) imply that 
for every $q \in \Z$, and $\x\in C$,
\[
\hyperH^q(C,\mathcal{F}^\bullet|_C) \cong \hyperH^q(C_{\gamma_i,\delta_j},\mathcal{F}^\bullet|_{C_{\gamma_i,\delta_j}}) 
\cong \bigoplus_{k+m=q} \HH^k(C,\Q)\otimes\HH^m(\mathcal{F}^\bullet_\x)
\]
for all $i$ large enough, and after having chosen $i$,  for all $j$ large enough. The first part of the proposition
now follows by replacing $\delta_i$ by $\eps_{2\mathrm{level}(\sigma)-1}$ and
$\gamma_i$ by $\eps_{2\mathrm{level}(\sigma)}$.

Part (\ref{item:perturbationD}) of the proposition follows from an induction on $p$ using the first part as base case 
the case $p=0$, which has been proved above.

Assume now by induction  that (\ref{item:perturbationD})
holds for smaller values of $p$.
First observe that if $\sigma \neq \sigma' \in \{0,1,-1\}^{\mathcal{P}}$, with $\level(\sigma) = \level(\sigma')$,
then $\RR(\sigma_{\bar\eps},\K) \cap \RR(\sigma_{\bar\eps},\K) = \emptyset$. Also, if 
$\sigma,\sigma'\in\{0,1,-1\}^{\mathcal{P}}$, and there exists $P \in \mathcal{P}$, such that 
$\sigma(P)\sigma'(P) = -1$, then $\RR(\sigma_{\bar\eps},\K) \cap \RR(\sigma_{\bar\eps},\K) = \emptyset$.
Hence, 
if $\RR(\sigma_{0,\bar\eps}) \cap \cdots \cap \RR(\sigma_{p,\bar\eps}) \neq \emptyset$, 
there exists a unique permutation $\pi$ of $\{0,\ldots,p\}$ such that 
\[
\level(\sigma_p) < \cdots < \level(\sigma_0).
\]
After permuting the indices by $\pi$
we can assume without loss of generality that 
\[
\level(\sigma_p) =\ell_p < \cdots < \level(\sigma_0) = \ell_0,
\]
and also that for $0 \leq i \leq p$, $\sigma_i$ is the sign condition given by 
$P_1 = \cdots = P_{\ell_i} = 0, P_{\ell_i+1} > 0 ,\ldots, P_s > 0$.

Now let $\K' = \K_{\eps_{2 \ell_0}, \eps_{2 \ell_0 -1}} = \K \cap \RR(\sigma_{p,\bar\eps})$. Notice that only $\eps_{2 \ell_0}, \eps_{2 \ell_0 -1}$ enters into
the definition of $\K'$. Applying the induction hypothesis to the set of sign conditions $\sigma_1,\ldots,\sigma_p$,
and the compact set $\K'$, we obtain that for each connected component $C$
of $\RR(\sigma_{1,\bar\eps},\K') \cap \cdots \cap \RR(\sigma_{p,\bar\eps},\K')$, we obtain
that
$C$ is homotopy equivalent to $C' = \RR(\sigma_{p},\K') \cap C$ for $0 < \bar\eps \ll 1$, and
for every $q\in\Z$ and $\x \in C'$,
\[
\hyperH^q(C,\mathcal{F}^\bullet|_{C}) \cong \hyperH^q(C',\mathcal{F}^\bullet|_{C'})\cong 
\bigoplus_{i+j=q} \HH^i(C,\Q) \otimes \HH^j(\mathcal{F}^\bullet_\x).
\]
This completes the induction.
\end{proof}

Before proving Theorem \ref{thm:topological-complexity-sheaves} we need another proposition.

\begin{proposition}
\label{prop:qe}
Let $\mathcal{Q} \subset \R[Y_0,\ldots,Y_m, X_0,\ldots,X_n]$, with $\card(\mathcal{Q}) = s$, and $\max_{Q \in \mathcal{Q}} \deg(Q) \leq d$. Let $\K = \Sphere^m \times \Sphere^n$,  and $\pi:\K \rightarrow \Sphere^n$ the
projection to the second factor. Then, there exists a finite family of polynomials 
$\mathcal{P} \subset \R[X_0,\ldots,X_n]$ satisfying the following.
 \begin{enumerate}
\item For $0 < \bar\eps \ll 1$, and each $\sigma \in \{0,1,-1\}^{\mathcal{Q}}$, and every connected component $C$ of $\RR(\sigma,\Sphere^{n})$,
and $\x,\x' \in C$, there is a semi-algebraic homeomorphism, 
$\phi_{\x,\x',\bar\eps}:  \pi^{-1}(\x) \rightarrow \pi^{-1}(\x')$, such that
for each tuple
$(\sigma_0,\ldots,\sigma_p)$, where for each $i$, $\sigma_i \in \{0,1,-1\}^{\mathcal{Q}}$,
$\phi_{\x,\x',\bar\eps}$ restricts to a homeomorphism between 
\[
\Big(\bigcap_{0\leq i \leq p}\RR(\sigma_{i,\bar\eps}, \K) \Big)_\x,  
\Big(\bigcap_{0\leq i \leq p}\RR(\sigma_{i,\bar\eps}, \K) \Big)_{\x'}
\]
(following the same notation as in Proposition  \ref{prop:perturbation}).
\item
The cardinality of $\mathcal{P}$ as well as $\max_{P \in \mathcal{P}} \deg(P)$ are bounded by $( s d)^{n^{O(1)}}$.
\end{enumerate}
\end{proposition}

\begin{proof}
Let
\[
\bar{\mathcal{P}}' \subset \R[\bar\eps, X_0,\ldots,X_{m(n)}]
\] 
be defined as a set of polynomials
such that for all $0 < \bar\eps \ll 1$,
\[
\bigcup_{P \in \bar{\mathcal{P}}'} \ZZ(P(\bar\eps,\cdot),\Sphere^{m(n)})
\]
contains the images of the critical points of
the projection  $\pi$ restricted to all varieties 
$\ZZ(\mathcal{Q}_{\bar\eps}'(\bar\eps,\cdot),\K_n)$, for all subsets
$\mathcal{Q}_{\bar\eps}'\subset \mathcal{Q}_{n,\bar\eps}$.  Since 
by Proposition \ref{prop:general-position}, $\mathcal{Q}_{n,\bar\eps}$ is in $m+n$ general position
with respect to $\K_n$, it suffices to consider only subsets  $\mathcal{Q}_{\bar\eps}'$ with
$\card(\mathcal{Q}_{\bar\eps}') \leq m+n$. Now writing the polynomials  in 
$\bar{\mathcal{P}}'$ as
polynomials in the variables $\bar\eps$, let 
$\mathcal{P}'$ denote the union of all the coefficients of all the polynomials in  
$\bar{\mathcal{P}}'$.

Using
arguments similar to those used in the proof of Lemma 3.9 in \cite{BV06}
it follows that for every $\x \in \Sphere^n$, for $0 < \bar\eps \ll 1$, there exists 
a connected component $D_{\bar\eps}$  of  $\Sphere^n \setminus \ZZ(\mathcal{P}'_{\bar\eps})$,
such that for any $\x' \in D_{\bar\eps}$, 
there exists a homeomorphism $\phi_{\x,\x',\bar\eps}: \pi^{-1}(\x) \rightarrow \pi^{-1}(\x')$, such that
for any $\mathcal{Q}_{\bar\eps}$-closed semi-algebraic subset $S \subset \K$, 
$\phi_{\x,\x',\bar\eps}$ restricts to a homeomorphism
$\pi^{-1}(\x)\cap S \rightarrow \pi^{-1}(\x') \cap S$.
In particular, note that 
for each tuple
$(\sigma_0,\ldots,\sigma_p)$, where for each $i$, $\sigma_i \in \{0,1,-1\}^{\mathcal{Q}}$,
$\cap_{0\leq i \leq p}\RR(\sigma_{i,\bar\eps}, \K)$ is a 
$\mathcal{Q}_{\bar\eps}$-closed semi-algebraic subset of $\K$ (see Remark \ref{rem:closed-formula}).

Note that
 the degrees of the polynomials in $\bar{\mathcal{P}}'$ are bounded by $d^{O(m)}$,
 and noting further that each one of them depends on at most 
$m+n$ of the $\eps_i$'s gives a singly exponential bound on the cardinality of $\mathcal{P}$ as well.
\end{proof}

\begin{proof}[Proof of Theorem \ref{thm:topological-complexity-sheaves}]
Let $\Big(\mathcal{F}^\bullet_n \in \Ob(\constrD^b(\Sphere^{m(n)}))\Big)_{n >0}$ be a sequence of constructible sheaves belonging to the class
$\mathbf{\Lambda}^{(p)}(\bs{\mathcal{P}}_\R)$ for some $p \geq 0$. 

We first prove that there exists a polynomial $q_1(n)$ (depending on the sequence $\left(\mathcal{F}_n \right)_{n>0}$) and 
for each $n >0$ a family of polynomials $\mathcal{P}_n \subset \R[X_0,\ldots,X_{m(n)}]$, such that:
\begin{enumerate}
\item
 both the $\card(\mathcal{P}_n)$ and the degrees of the polynomials in $\mathcal{P}_n$ are bounded by $2^{q_1(n)}$;
\item 
the semi-algebraic partition $\Pi(\mathcal{P}_n,\Sphere^{m(n)})$ (cf.  Notation \ref{not:partition-set}) is subordinate to $\mathcal{F}^\bullet_n$;  
\item
moreover, for $\x \in C$, 
\[
\sum_i \dim_\Q \HH^i((\mathcal{F}^\bullet_n)_\x) \leq 2^{q_1(n)},
\]
and
\item
\[
\HH^i((\mathcal{F}^\bullet_n)_\x) = 0,
\]
for $|i|\geq  q_1(n)$.
\end{enumerate}

The proof of the above claim is by induction on $p$. If $p = 0$, then the
claim follows directly from the definition of $\bs{\mathcal{P}}_\R$ and Proposition \ref{prop:singly-exponential}. 
Now suppose that the claim is true for all smaller values of $p$. 

Now a sequence $(\mathcal{F}^\bullet_n)_{n > 0}$ by definition belongs to $\mathbf{\Lambda}^{(p)}(\bs{\mathcal{P}}_\R)$ 
 if either,
\begin{enumerate}
\item
for each $n >0$
\[
\mathcal{F}^\bullet_n = R\pi_{n,*}(\mathcal{G}^\bullet_n),
\]
and the sequence 
\[
\Big(\mathcal{G}^\bullet_n \in \Ob(\constrD(\K_n))\Big)_{n >0} \in \mathbf{\Lambda}^{(p-1)}(\bs{\mathcal{P}}_\R),
\]
where 
$m_1(n) \in \Z[n]$ is a non-negative polynomial,
$\K_n = \Sphere^{m_1(n)} \times \Sphere^{m(n)}$,
and
$\pi_n: \Sphere^{m_1(n)} \times \Sphere^{m(n)} \rightarrow \Sphere^{m(n)}$ 
is the projection to the second factor;  or,
\item
$(\mathcal{F}^\bullet_n)_{n > 0}$
can be obtained from a constant number of sequences of type described in (A) above, by taking tensor products, direct sums, truncations and pull-backs. 
\end{enumerate}
However,
the claim to be proved (i.e., the singly exponential bound on $b(\mathcal{F}^\bullet_n)$)  is easily shown to be preserved under operations of tensor products, 
direct sums, truncations and pull-backs (by the same method
of proof as in the proof of Proposition \ref{prop:stability-sheaf-P}). Thus, it
suffices to consider only the first case. 
In what follows we will the following notation. For $\x \in \Sphere^{m(n)}$, and any $S \subset \K_n$,
we will denote $S_\x = \pi_n^{-1} \cap S$.

By the induction hypothesis there exists polynomial $q_2(n)$ (depending on the sequence $\left(\mathcal{G}^\bullet_n \right)_{n>0}$) and 
for each $n >0$ a family of polynomials 
\[
\mathcal{Q}_n \subset \R[Y_0,\ldots,Y_{m_1(n)},X_0,\ldots, X_{m(n)}],
\] 
such that:
\begin{enumerate}
\item
both  $s_n = \card(\mathcal{Q}_n)$ and the degrees of the polynomials in $\mathcal{Q}_n$ are
bounded by $2^{q_2(n)}$;
\item 
the semi-algebraic partition $\Pi(\mathcal{Q}_n,\K_n)$ is subordinate to  $\mathcal{G}^\bullet_n$;
\item
moreover, for each $\z \in D$, 
$
\sum_i \dim_\Q \HH^i((\mathcal{G}^\bullet_n)_{\z})\leq 2^{q_2(n)},
$ and
\item
$
\HH^i((\mathcal{G}^\bullet_n)_{\z})=0,
$
for $|i| \geq q_2(n)$.
\end{enumerate}

Using Proposition \ref{prop:covering} and Corollary \ref{cor:covering} we first replace $\mathcal{Q}_n$ by another family
(which we will still denote by $\mathcal{Q}_n$) such that over each 
$C \in \Cc(\sigma,\K_n), \sigma \in \{0,1,-1\}^{\mathcal{Q}_n}$, and every $q \in \Z$,
$\mathcal{H}^q(\mathcal{G}^\bullet_n)|_C$ is a constant (not just locally constant sheaf).
Moreover, the cardinality and the degrees of the polynomials appearing in this new family
are still bounded singly exponentially.
 
The proof now is now in several parts.
\begin{enumerate}
\item
We first prove that for each $\x \in \Sphere^{m(n)}$, 
$
\sum_i \dim_\Q \hyperH^i(\pi_n^{-1}(\x), \mathcal{G}^\bullet_n|_{\pi_n^{-1}(\x)})
$ 
is bounded
singly exponentially.
Let $\bar\eps = (\eps_1,\eps_2,\ldots,\eps_{2s_n})$. Consider the family of 
$2s_n^2$ polynomials $\mathcal{Q}_{n,\bar\eps}$ (cf. Notation \ref{not:misc}).

Denoting for each $\sigma \in \{0,1,-1\}^{\mathcal{Q}_n}$, $\RR(\sigma_{\bar\eps},\K_n)$ by 
$D_{\sigma,\bar\eps}$, we have by Proposition \ref{prop:cech}
for each $\x \in \Sphere^{m(n)}$, a spectral sequence, 
\begin{equation}
\label{eqn:spectral-sequence}
E_r^{p,q}(\x),
\end{equation}
whose $E_2$-term is given by
\begin{equation}
\label{eqn:abut}
E_2^{p,q}(\x) \cong  \bigoplus_{\substack
{\sigma_0,\prec \cdots \prec \sigma_p, \\
\sigma_j \in \{0,1,-1\}^{\mathcal{Q}_n}}
}
 \hyperH^q((D_{\sigma_0,\bar\eps} \cap \cdots \cap D_{\sigma_p,\bar\eps})_\x,\mathcal{G}^\bullet_n|_{(D_{\sigma_0,\bar\eps} \cap \cdots \cap D_{\sigma_p,\bar\eps})_\x})
\end{equation}
that abuts to $\hyperH^*(\pi_n^{-1}(\x),\mathcal{G}^\bullet_n|_{\pi_n^{-1}(\x)})$.

By induction hypothesis, we have that  
$
\HH^{q}((\mathcal{G}^\bullet_n)_{(\y,\x)}) = 0,
$ 
for all $q$ with $|q| \geq q_2(n)$, and $\y \in \Sphere^{m_1(n)}$.

It
follows that  for any $\x \in \Sphere^{m(n)}$,
$\hyperH^i(\pi_n^{-1}(\x), \mathcal{G}^\bullet_n|_{\pi_n^{-1}(\x)})=0$ for all $i$ with $|i| > N(n) = q_2(n) + m_1(n)$.

It then follows from the $E_2$-term \eqref{eqn:abut} of the spectral sequence \eqref{eqn:spectral-sequence} 
that
\[
\dim_\Q \hyperH^i(\pi_n^{-1}(\x),\mathcal{G}^\bullet_n|_{\pi^{-1}(\x)}) \leq 
\sum_{p+q=i, |q| \leq N(n)} \sum_{\bar\sigma} \dim_\Q\hyperH^q((D'_{\bar\sigma, \bar\eps})_\x,\mathcal{G}^\bullet_n|_{(D'_{\bar\sigma, \bar\eps})_\x}),
\]
where 
\begin{eqnarray*}
\bar\sigma &=& \sigma_0\prec \cdots \prec\sigma_p, \sigma_j \in \{0,1,-1\}^{\mathcal{Q}_n}, \\
D' _{\bar\sigma,\bar\eps}&=& D_{\sigma_0,\bar\eps} \cap \cdots \cap D_{\sigma_p,\bar\eps}.
\end{eqnarray*}

It follows from Part (\ref{item:perturbationD}) of Proposition \ref{prop:perturbation} that,
for any $\z=(\y,\x) \in (D_{\sigma_p}\cap D'_{\bar\sigma,\bar\eps})_{\x}$,
\begin{eqnarray*}
\hyperH^q((D'_{\bar\sigma,\bar\eps})_\x,\mathcal{G}^\bullet_n|_{(D'_{\bar\sigma,\bar\eps})_\x})
&\cong&
\hyperH^q((D_{\sigma_p}\cap D'_{\bar\sigma,\bar\eps})_\x,
\mathcal{G}^\bullet_n|_{(D_{\sigma_p}\cap D'_{\bar\sigma,\bar\eps})_\x}) \\ 
&\cong&
\bigoplus_{i+j = q} \HH^i(D_{\sigma_p}\cap D'_{\bar\sigma,\bar\eps},\Q) \otimes \HH^{j}(\mathcal{G}^\bullet_n)_{\z})).
\end{eqnarray*}

It follows that 
$\displaystyle{
\sum_i \dim_\Q \hyperH^i(\pi_n^{-1}(\x), \mathcal{G}^\bullet_n|_{\pi_n^{-1}(\x)})
}
$
is bounded
singly exponentially,
since the descriptions of the semi-algebraic sets  
$D_{\sigma_p}\cap D'_{\bar\sigma,\bar\eps}$ have singly exponential complexity,
$\dim(\HH^{*}(\mathcal{G}^\bullet_n)_{\z})$ is bounded singly exponentially using the induction
hypothesis.

\item
We next obtain a covering of $\Sphere^{m(n)}$ by a singly exponentially many compact semi-algebraic sets
of singly-exponential complexity such that over each element $C$ of the covering, and $q \in \Z$,
the sheaf, $\mathcal{H}^q(R\pi_{n,*}\mathcal{G}^\bullet_n)|_C$, 
where $\mathcal{H}^q(R\pi_{n,*}\mathcal{G}^\bullet_n)$ is the sheaf
associated to the pre-sheaf defined by 
$V\mapsto \HH^q(R\pi_{n,*}\mathcal{G}^\bullet_n(V))$,  is a constant
sheaf. 

It follows from Proposition \ref{prop:qe}, that there exists a finite set of
polynomials $\mathcal{P}_n' \subset \R[X_0,\ldots,X_{m(n)}]$ such that:
\begin{enumerate}
\item For $0 < \bar\eps \ll 1$, and each $\sigma \in \{0,1,-1\}^{\mathcal{P}_n'}$, and every  $C \in
\Cc(\RR(\sigma,\Sphere^{n}))$,
and $\x,\x' \in C$, there is a semi-algebraic homeomorphism, 
$\phi_{\x,\x',\bar\eps}:  \pi_n^{-1}(\x) \rightarrow \pi_n^{-1}(\x')$, such that
for each tuple
$(\sigma_0,\ldots,\sigma_p)$, where for each $i$, $\sigma_i \in \{0,1,-1\}^{\mathcal{Q}_n}$,
$\phi_{\x,\x',\bar\eps}$ restricts to a homeomorphism between 
\[
\Big(\bigcap_{0\leq i \leq p}\RR(\sigma_{i,\bar\eps}, \K_n) \Big)_\x,  
\Big(\bigcap_{0\leq i \leq p}\RR(\sigma_{i,\bar\eps}, \K_n) \Big)_{\x'}.
\]
\item
The cardinality of $\mathcal{P}'_n$ as well as $\max_{P \in \mathcal{P}'_n} \deg(P)$ are bounded 
singly exponentially.
\end{enumerate}

The  homeomorphism $\phi_{\x}$ thus induces an isomorphism between 
$E_2^{p,q}(\x)$ and $E_2^{p,q}(\x')$,  and hence 
between the  groups
$\hyperH^*(\pi_n^{-1}(\x),\mathcal{G}^\bullet_n|_{\pi_n^{-1}(\x)})$ and
$\hyperH^*(\pi_n^{-1}(\x'),\mathcal{G}^\bullet_n|_{\pi_n^{-1}(\x')})$. 

Notice that by Theorem \ref{thm:proper-base-change}, for all $\x \in \Sphere^{m(n)}$,
\[
R^q\pi_{n,*}(\mathcal{G}^\bullet_n|_{\pi_n^{-1}(C)})_\x \cong 
\mathcal{H}^q(R\pi_{n,*}\mathcal{G}^\bullet_n)_\x \cong
\hyperH^q(\pi_n^{-1}(\x), \mathcal{G}^\bullet_n|_{\pi_n^{-1}(\x)}).
\]

Thus, we have that for each $\sigma' \in \{0,1,-1\}^{\mathcal{P}_n'}$, and each  
\[
C' \in \Cc(\RR(\sigma',\Sphere^{m(n)})),
\]
$\mathcal{H}^q(R\pi_{n,*}\mathcal{G}^\bullet_n)|_{C'}$ is a \emph{locally constant} sheaf on $C'$.

Moreover, using Proposition \ref{prop:covering}, there exists a family of polynomials 
\[
\mathcal{P}_n \subset \R[X_0,\ldots,X_{m(n)}]
\]
such that the $\card(\mathcal{P}_n)$,
as well as the degrees of the polynomials in $\mathcal{P}_n$, are bounded by $(\card(\mathcal{P}_n') \max_{P \in \mathcal{P}_n'} \deg(P))^{m_1(n)^{O(1)}}$, 
and such that for each connected component $C$  of $\RR(\sigma,\Sphere^{m(n)})$, for  any
$\sigma\in \{0,1,-1\}^{\mathcal{P}_n}$:
\begin{enumerate}
\item
the signs of the polynomials in $\mathcal{P}_n'$ are constant over $C$;
\item
$C$ is contained in some contractible subset of $\RR(\sigma,\Sphere^{m(n)})$ for some 
$\sigma \in \{0,1,-1\}^{\mathcal{P}_n'}$;
\item and as a consequence of the properties (A) and (B) above and Proposition \ref{prop:constant}, 
$\mathcal{H}^q(R\pi_{n,*}\mathcal{G}^\bullet_n)|_{C}$ is a 
\emph{constant} (not just locally constant) sheaf on $C$.
\end{enumerate}
\item
Introduce another vector of variables $\bar\delta = (\delta_1,\ldots,\delta_{2 \card(\mathcal{P}_n}))$, and
consider the family $\mathcal{P}_{n,\bar\delta}$ (cf. Notation \ref{not:misc}). Consider for 
$0 < \bar\delta \ll 1$ the covering of
$\Sphere^{m(n)}$ by the compact sets $C_{\sigma,\bar\delta} = \RR(\sigma_{\bar\delta},\Sphere^{m(n)})$ 
for $\sigma \in \{0,1,-1\}^{\mathcal{P}_n}$ (cf. Notation \ref{not:misc}).

We now observe by Proposition \ref{prop:cech} that there exists a spectral sequence converging to 
$\hyperH^*(\Sphere^{m(n)}, R\pi_{n,*}\mathcal{G}^\bullet_n)$ whose $E_2$-term is given by
\begin{equation}
\label{eqn:total}
  E_2^{p,q} \cong  \bigoplus_{
  \substack
  {\bar\sigma = (\sigma_0\prec \cdots \prec\sigma_p),\\
  \sigma_i  \in \{0,1,-1\}^{\mathcal{P}_n}
  }}
  \hyperH^q(C_{\bar\sigma,\bar\delta}, R\pi_{n,*}\mathcal{G}^\bullet_n|_{C_{\bar\sigma,\bar\delta}}),
 \end{equation}
  where $C_{\bar\sigma,\bar\delta} = C_{\sigma_0,\bar\delta} \cap \cdots \cap C_{\sigma_p,\bar\delta}$.

Also, observe that by Proposition \ref{prop:spectral}, for each $q \in \Z$, 
there exists another spectral sequence converging to  
$\hyperH^*(C_{\bar\sigma,\bar\delta}, R\pi_{n,*}\mathcal{G}^\bullet_n|_{C_{\bar\sigma,\bar\delta}})$
whose $E_2$-term is given by 
\begin{equation}
\label{eqn:contrib1}
E_2^{p,q} \cong \HH^p(C_{\bar\sigma,\bar\delta},\mathcal{H}^q(R\pi_{n,*}\mathcal{G}^\bullet_n|_{C_{\bar\sigma,\bar\delta}})) \cong
 \bigoplus_{C \in  \Cc(C_{\bar\sigma,\bar\delta})} \HH^p(C,\mathcal{H}^q(R\pi_{n,*}\mathcal{G}^\bullet_n|_{C})).
\end{equation}

Now, by part (D) of Proposition \ref{prop:perturbation}, if $C \in \Cc(C_{\bar\sigma,\bar\delta})$, then
\begin{itemize}
\item[(i)]
$C$ is closed;
\item[(ii)]
$C$ is homotopy equivalent to $C' = \RR(\sigma_{p},\Sphere^{m(n)}) \cap C$, and 
\item[(iii)]
for every $p\in\Z$ and $\x \in C'$,
\begin{eqnarray}
\label{eqn:contrib2}
\nonumber
\HH^p(C,\mathcal{H}^q(R\pi_{n,*}\mathcal{G}^\bullet_n)|_{C}) &\cong& 
\HH^p(C',\mathcal{H}^q(R\pi_{n,*}\mathcal{G}^\bullet_n)|_{C'})\\
&\cong& 
\HH^p(C',\Q) \otimes \hyperH^q(\pi_n^{-1}(\x),\mathcal{G}^\bullet_n|_{\pi_n^{-1}(\x)}),
\end{eqnarray}
\end{itemize}
\end{enumerate}
where the last isomorphism is by  Theorem \ref{thm:proper-base-change}.

The theorem follows by adding up the contributions from  \ref{eqn:contrib2},
noting that 
\begin{itemize}
\item[(i)]
the number of different $C'$'s, and the dimensions of the corresponding vector spaces $\HH^*(C',\Q)$ are all 
bounded singly exponentially using 
standard bounds \cite[Theorem 7.50]{BPRbook2}, since each $C'$ has a singly exponential sized description, 
and 
\item[(ii)]
the dimensions of the various
$\hyperH^q(\pi_n^{-1}(\x),\mathcal{G}^\bullet_n|_{\pi_n^{-1}(\x)})$
are bounded
singly exponentially by
part (A) of the proof.
\end{itemize}
\end{proof}

\subsection{Complexity of generalized quantifier elimination} 
\label{subsec:generalized-qe}
The following result which follows directly from the proof of Theorem \ref{thm:topological-complexity-sheaves} above but which
does not refer to any complexity classes could be of independent interest. It is the sheaf-theoretic
analog of an effective singly exponential complexity bound for eliminating one block of quantifiers
in the first order theory of the reals \cite{R92, BPRbook2}. In particular, the implied algorithm in the
following theorem could be viewed as the sheaf-theoretic analog of Algorithm 14.1 (Block Elimination) in \cite{BPRbook2} restricted to the compact situation.  We omit the proof of this theorem which is
embedded in the proof of the intermediate claim inside the the proof of Theorem \ref{thm:topological-complexity-sheaves} above.

\begin{theorem}[Complexity of generalized quantifier elimination]
\label{thm:effective}
Let 
$$
\displaylines{
\mathcal{F}^\bullet \in \Ob(\constrD^b(\Sphere^m \times \Sphere^n)),
}
$$
and let $\mathcal{Q}\subset \R[Y_0,\ldots,Y_m,X_0,\ldots,X_n]$ be a finite set of polynomials, such that the semi-algebraic partition $\Pi(\mathcal{P}, \Sphere^m \times \Sphere^n)$ is subordinate to $\mathcal{F}^\bullet$. Moreover, suppose that $\card(\mathcal{Q}) = s$, and also that the degrees of the polynomials in 
$\mathcal{Q}$ are  all bounded by $d$. Let $\pi: \Sphere^m \times \Sphere^n \rightarrow \Sphere^n$ denote the projection
to the second factor. Then, there exists a family of polynomials $\mathcal{P} \subset \R[X_0,\ldots,X_n]$,
with $\card(\mathcal{P})$, as well as the degrees of the polynomials in $\mathcal{P}$ bounded by 
$(sd)^{(m+n)^{O(1)}}$, such that the semi-algebraic partition 
$\Pi(\mathcal{P}, \Sphere^n)$ is subordinate to the
constructible sheaf $R\pi_*\mathcal{F}^\bullet \in \Ob(\constrD^b(\Sphere^n))$. Moreover, there exists an
algorithm for computing semi-algebraic description of the partition $\Pi(\mathcal{P}, \Sphere^n)$, given
$\mathcal{Q}$ as input,  with complexity bounded by  $(sd)^{(m+n)^{O(1)}}$.
\end{theorem}
\begin{proof}

The proof of Theorem \ref{thm:topological-complexity-sheaves} involves successively constructing certain  families of polynomials starting from the family of polynomials constructed in the previous
step beginning with the given family $\mathcal{Q}$. In order to prove Theorem \ref{thm:effective} it suffices to check that each step in this 
process has singly exponential complexity. 

\begin{enumerate}
\item The family $\mathcal{Q}$ is replaced by another family (also denoted by $\mathcal{Q}$) using Corollary
\ref{cor:covering}. The complexity of this step is singly exponential,  and this follows from the complexity
of Algorithm 16.14 (Covering by contractible sets) in \cite{BPRbook2}.

\item Another family, $\mathcal{P}'$  is constructed from $\mathcal{Q}$. The complexity of this step is again bounded singly exponentially, since only subsets of $\mathcal{Q}$ of cardinality at most $m+n$ enters into the construction, and at each step the arithmetic computations using polynomials involve at most $O(m+n)$ infinitesimals.

\item
Finally, the family $\mathcal{P}$ is constructed from $\mathcal{P}'$. The complexity of this step
is also clearly bounded singly exponentially.
\end{enumerate}
\end{proof}  
 
\begin{remark}
The restriction to spheres in Theorem \ref{thm:effective} is to ensure properness of the projection
$\pi$. It is possible that with more work it would be possible to extend the theorem to the non-compact
case and consider projections $\pi:\R^m \times \R^n \rightarrow \R^n$, and consider not just the functor
$R\pi_{*}$ but the derived image functor with proper support, $R\pi_{!}$ as well. We do not undertake
this task in the current paper.
\end{remark}

\begin{remark}
The name  ``generalized quantifier elimination'' in Theorem \ref{thm:effective} is justified as follows.
Let in
Theorem \ref{thm:effective}, 
$\mathcal{F}^\bullet$ be the constant sheaf $\Q_{\Sphere^m \times \Sphere^n}$ 
 (considered as a complex concentrated at degree $0$ as usual).
 
Let $\Phi$ be a $\mathcal{Q}$-closed formula (cf. Notation \ref{not:P-formula}). Let
$\Phi_\exists$ and $\Phi_\forall$ be the (quantified) formulas defined by 
\begin{eqnarray*}
\Phi_\exists(\X) &=& (\exists \y \in \Sphere^m) \Phi(\Y,\X), \\
\Phi_\forall(\X) &=& (\forall \y \in \Sphere^m) \Phi(\Y,\X). \\
\end{eqnarray*}
Let $S_\exists,S_\forall$
be the semi-algebraic subsets of $\Sphere^n$ defined by $\Phi_\exists,\Phi_\forall$ respectively,
and 
$\Pi_\exists(\Phi)$ (respectively, $\Pi_\forall(\Phi)$) denote the partition of $\Sphere^n$ into the sets 
$S_\exists$ and $\Sphere^n \setminus S_\exists$ (respectively, $S_\forall$ and $\Sphere^n \setminus
S_\forall$).
Let $\mathcal{P}$ be the family of polynomials as in Theorem \ref{thm:effective}.
Then, 
\[
\Pi(\mathcal{P},\Sphere^n)\prec \Pi_\exists(\Phi),\Pi_\forall(\Phi)
\]
(cf. Notation \ref{not:partition-set} and Example \ref{eg:forall}).
In other words,
 the partition of $\Sphere^n$ induced by the polynomials
$\mathcal{P}$ is finer than the induced partition obtained by eliminating the block of quantifiers (either existential or universal) from any formula of the form $(\exists \y \in \Sphere^m) \Phi(\Y,\X)$ or $(\forall \y \in \Sphere^m) \Phi(\Y,\X)$, where 
$\Phi$ is a $\mathcal{Q}$-closed formula.
\end{remark}

\begin{remark}
Another important point to note is the following. Designing an algorithm with singly exponential complexity  
that computes all the Betti numbers of a given real
algebraic variety or semi-algebraic set  is a long standing open problem in algorithmic semi-algebraic
geometry. Algorithms with singly exponential complexity are known only for computing the first few 
(i.e.,  a constant number of) Betti numbers \cite{BPRbettione,Bas05-first}, as well as the Euler-Poincar\'e characteristic  \cite{Basu1},  of semi-algebraic sets -- but
no algorithm with singly exponential complexity 
is known for computing all the Betti numbers (even in the case when the number of polynomials appearing in the 
description of the given set, as well as their degrees, are bounded by some constant).  Note that for any compact
semi-algebraic set $S \subset \R^n$, computing the Betti numbers of $S$ is equivalent to computing the 
dimensions of the groups $\HH^*(S,\Q_S)$. 
Theorem \ref{thm:effective} does not say anything about
this problem. If the constructible sheaf $\mathcal{F}^\bullet$ in Theorem \ref{thm:effective} is taken to be the
sheaf $\Q_S$ where 
$S \subset  \Sphere^m \times \Sphere^n$ is a $\mathcal{P}$-closed semi-algebraic subset, then for any element $C$ of the semi-algebraic partition of $\Sphere^n$ constructed by the algorithm, the Poincar\'e polynomial $P_{S_\x}$ stays invariant for $\x \in C$ (denoting $S_\x = S \cap \pi^{-1}(\x)$). However, the algorithm is not able to compute this polynomial. If the algorithm was able to compute this
polynomial (with the same complexity bound), then it would
imply the existence of an algorithm with singly exponential complexity  for computing the Betti numbers of real algebraic sets defined by polynomials having degrees bounded by any fixed constant.
\end{remark}

\section{Constructible functions and sheaves: Toda's theorem}
\label{sec:connection-to-Toda}
In this section, we discuss the connections between the complexities of 
constructible functions and constructible sheaves. We formulate a new
conjecture that that could be seen as analogous to Toda's theorem in
discrete complexity theory \cite{Toda}. Toda's theorem gives an inclusion of 
the polynomial hierarchy $\mathbf{PH}$ in the class $\mathbf{P}^{\#\mathbf{P}}$ 
where the right hand side is the set of languages accepted by a Turing machine in 
polynomial times but with access to an oracle computing functions in $\#\mathbf{P}$.
The class $\#\mathbf{P}$ consists of sequences of functions, $(f_n: \{0,1\}^n \rightarrow \N)_{n > 0}$
which counts the number of points in the fibers of a linear projection of a language in $\mathbf{P}$.
We refer the reader to \cite{BZ09} where this geometric definition of the class $\#\mathbf{P}$ is
elaborated.

In \cite{BZ09} (respectively, \cite{Basu-complex-toda}) a geometric definition was given of a class,
$\#\mathbf{P}_\R^{\dagger}$ (respectively, $\#\mathbf{P}_\C^{\dagger}$)  of sequences
of functions $(f_n: \Sphere^n \rightarrow \Z[T])_{n >0}$ (respectively, $(f_n: \PP^n_\C \rightarrow \Z[T])_{n>0}$) where the functions $f_n$ took values in the Poincar\'e polynomials of the fibers of
a projection of a language (in the B-S-S sense) in $\mathbf{NP}_\R^c$ (respectively, $\mathbf{NP}_\C^c$).
We omit the precise definitions of these classes, but point out that these functions are in fact
constructible functions (or more precisely each component of these functions corresponding to the 
different coefficients of the image polynomial is a constructible function). An analog of Toda's theorem
was proved in \cite{BZ09}.

 \begin{theorem} \cite{BZ09}
 \[
 \mathbf{PH}_\R^c  \subset \mathbf{P}_\R^{\#\mathbf{P}_\R^{\dagger}}.
 \]
 \end{theorem}
 A similar result was proved in the complex case in \cite{Basu-complex-toda}.

The relationship implicit in Toda's theorem (and its real and complex analog) raises the interesting question of whether such a relationship is also true in the sheaf-theoretic case. In particular, the following 
proposition (see \cite{Schurmann}) is very suggestive.

We first need a new notation.

\begin{notation}
\label{not:CF}
For $X$ a semi-algebraic set we denote by $\CF(X)$ the set of 
constructible functions on $X$.
\end{notation}

\begin{notation}
If $X$ is a semi-algebraic set, and $\mathcal{F}^\bullet \in \Ob(\constrD^b(X))$, then we will denote by 
$\mathrm{Eu}(\mathcal{F}^\bullet)$ the constructible function on $X$ defined by 
\[
\mathrm{Eu}(\mathcal{F}^\bullet)(\x) = \sum_{j} (-1)^j \dim_\Q \HH^j(\mathcal{F}^\bullet_\x).
\]
\end{notation}

\begin{proposition}\cite{Schurmann}
\label{prop:sheaves-to-functions}
Let $X,Y$ be compact semi-algebraic sets, and $f:Y \rightarrow X$ a semi-algebraic continuous
map.
Then, we have the following commutative diagram:
\[
\xymatrix{
\Ob(\constrD^b(Y)) \ar[r]^{Rf_*} \ar[d]^{\mathrm{Eu}}& \Ob(\constrD^b(X)) 
\ar[d]^{\mathrm{Eu}}
\\
\CF(Y)\ar[r]^{f_{*} = \int \;\cdot\; \dd\;\chi} & \CF(X).
}
\]
\end{proposition}

\begin{proof}
See \cite[Lemma 2.3.2]{Schurmann}.
\end{proof}

We now define the sheaf-theoretic analog of the class $\#\mathbf{P}$ and its generalizations.

\begin{definition}
Let $m(n) \in \Z[n]$ be a non-negative polynomial. We say that a sequence of constructible functions
$\left(f_n: \Sphere^{m(n)} \rightarrow \Z\right)_{n > 0}$ is in the class $\#\bs{\mathcal{P}}_\R$, if there exists
a sequence of constructible sheaves $\left(\mathcal{F}^\bullet_n\right)_{n > 0} \in \mathbf{\Lambda}^{(1)}\bs{\mathcal{P}}_\R$ such that for each $n>0$, 
$f_n = \textrm{Eu}(\mathcal{F}^\bullet_n)$.
More generally, we will say that $\left(f_n: \Sphere^{m(n)} \rightarrow \Z\right)_{n > 0}$ is in the class $\mathbf{Eu}(\mathbf{\Lambda}^{(p)}\bs{\mathcal{P}}_\R)$, 
if there exists a sequence of constructible sheaves $\left(\mathcal{F}^\bullet_n\right)_{n > 0} \in \mathbf{\Lambda}^{(p)}\bs{\mathcal{P}}_\R$ such that for each $n>0$, 
$f_n = \textrm{Eu}(\mathcal{F}^\bullet_n)$. We say that 
$\left(f_n: \Sphere^{m(n)} \rightarrow \Z\right)_{n > 0}$ is in the class $\mathbf{Eu}(\bs{\mathcal{PH}}_\R)$, 
if there exists a sequence of constructible sheaves $\left(\mathcal{F}^\bullet_n\right)_{n > 0} \in \bs{\mathcal{PH}}_\R$ such that for each $n>0$, $f_n = \textrm{Eu}(\mathcal{F}^\bullet_n)$.
\end{definition}

 We have the following conjecture which can be seen as a reformulation of Toda's fundamental theorem \cite{Toda}
 in sheaf-theoretic terms.

\begin{conjecture}
\[
\mathbf{Eu}(\bs{\mathcal{PH}}_\R) = \#\bs{\mathcal{P}}_\R.
\]
\end{conjecture}

\begin{remark}
\label{rem:BZvs.BC}
Note that the definition of the class $\#\bs{\mathcal{P}}_\R$ as defined above is closer in spirit to the class,
${\#}\mathbf{P}_\R$ defined by 
Burgisser and Cucker in \cite{Burgisser-Cucker06, Burgisser-Cucker-survey} using the Euler-Poincar\'e characteristic 
rather than the one of the class ${\#}\mathbf{P}^{\dagger}_\R$ in \cite{BZ09} which used the
Poincar\'e polynomial instead of the Euler-Poincar\'e polynomial as the ``counting'' function. One motivation behind the definition of  ${\#}\mathbf{P}^{\dagger}_\R$ 
in \cite{BZ09} was precisely the
ability to prove an  analog of Toda's theorem.
\end{remark}

\section{Conclusion and future directions}
In this paper we have begun the study of a complexity theory of constructible functions and sheaves patterned along the
line of the Blum-Shub-Smale theory for constructible/semi-algebraic sets. We have formulated
versions of the $\mathbf{P}$ vs. $\mathbf{NP}$ questions for classes of constructible functions as well as
sheaves.
An immediate goal would be to develop an analog of ``completeness'' results in classical complexity
theory and identify certain constructible sheaves to be complete in their class.  
Another goal is to incorporate the functor $\RHom$ and in particular  Verdier duality in the complexity theory
of sheaves. A related interesting problem is to investigate the role of adjointness of functors in the complexity theory
of sheaves,  noting that the functors $Rf_*$ and $f^{-1}$, and tensor product and $\RHom$, are adjoint pairs of functors (see also Section \ref{subsec:adjoint}).

 As mentioned in the introduction, aside from in semi-algebraic geometry, constructible functions and sheaves appear in many areas of mathematics, in particular
in the theory of linear partial differential equations and micro-local analysis as developed by Kashiwara and Schapira \cite{KS},
motivic integration \cite{Cluckers-Loeser}, as well as in a more applied setting of signal processing \cite{Ghrist2010}, and
computational geometry \cite{Schapira89, Schapira93, Schapira95}. A more distant goal would be to study these applications from the complexity viewpoint.

\section{Acknowledgments}
The author would like to thank the anonymous referees for a very careful reading of a previous version of this paper, and for many comments and suggestions that helped to improve the paper substantially.

\bibliographystyle{abbrv}
\bibliography{master}
\end{document}